\documentclass[12pt, journal, onecolumn]{IEEEtran}
\usepackage[a4paper, margin=1in]{geometry}

\RequirePackage{amsthm,amsmath,amsfonts,amssymb}
\usepackage{cite}
\RequirePackage[colorlinks,citecolor=blue,urlcolor=blue]{hyperref}
\RequirePackage{graphicx}
\usepackage{lineno}
\usepackage{array}
\allowdisplaybreaks[2]

\theoremstyle{plain}
\newtheorem{thm}{Theorem}[section]
\newtheorem{lem}[thm]{Lemma}
\newtheorem{prop}[thm]{Proposition}
\newtheorem{cor}[thm]{Corollary}

\theoremstyle{definition}

\newtheorem{conj}{Conjecture}[section]

\newcommand{\R}{\mathbb{R}}

\newcommand{\N}{\mathbb{N}}

\newcommand{\E}{\mathbb{E}}
\renewcommand{\P}{\mathbb{P}}

\newcommand\curly[1]{\left\{#1\right\}}
\newcommand\brac[1]{\left[#1\right]}
\newcommand\paren[1]{\left(#1\right)}
\newcommand\abs[1]{\left|#1\right|}

\renewcommand{\hat}{\widehat}
\DeclareMathOperator{\var}{Var}
\DeclareMathOperator{\cov}{Cov}

\let\epsilon\varepsilon
\DeclareMathOperator{\plug}{MLE}

\DeclareMathOperator{\li2}{Li_2}

\title{The Harmonic Entropy Estimator: Minimax Optimality and Semiparametric Efficiency for Infinite Alphabets}

\author{Octavio C\'esar Mesner \\
	\small Bursky School of Public Health, Washington University in St. Louis \\
	\small \texttt{mesner@wustl.edu}
}

\date{}

\begin{document}

\maketitle

\begin{abstract}
	Shannon entropy estimation for discrete distributions is a classical nonparametric problem whose statistical limits are well characterized under finite support.
	Whether minimax and efficiency guarantees extend to infinite-support distributions remains open, in part because the expectation and variance of logarithm-transformed multinomial counts lack closed forms, forcing analyses to rely on bounding arguments rather than exact moment expressions.
	We address this obstacle with an algebraic framework built on an exact identity relating harmonic-transformed counts to their log-probabilities, yielding the \emph{harmonic entropy estimator}.
	With this framework, we obtain two main results characterizing the statistical limits of this problem for distributions with sufficiently fast-decaying tails.
	First, for the class of distributions with at least quadratically decaying tails ($p_j\lesssim j^{-2}$), we establish that the $L_2$-minimax convergence rate is $1/n$, a rate anchored by estimator variance for both finite- and infinite-support distributions.
	Second, under the stronger condition $p_j =o(j^{-2})$, we demonstrate that the harmonic estimator is semiparametrically efficient: 
	$\sqrt{n} [\hat H(\mathbf{X}_n) -H(X)]$ converges weakly to a normal distribution with variance achieving the semiparametric efficiency bound, $\var[\log p(X)]$.
	Furthermore, the harmonic estimator is universally less biased than the plug-in estimator by circumventing the approximation error inherent to expectations of logarithm-transformed counts.
	More broadly, the harmonic approach may provide a foundation for developing new theory in entropy-based estimation, such as inference for mutual information under similar nonparametric regimes.
\end{abstract}


\section{Introduction}

Estimating the Shannon entropy of a discrete distribution from a finite sample is a fundamental problem in nonparametric functional estimation.
The properties of entropy estimators are well-established for discrete distributions with finite support.
For example, the plug-in estimator is asymptotically normal when sample size is much greater than a fixed support size \cite{basharin1959discEntropy}, and more recent advances have established the minimax rates of convergence in the challenging large-alphabet regime \cite{jiao2015minimax, wu2016large}.
Extending minimax guarantees to distributions with countably infinite support, however, has proven to be challenging \cite{cohen2023dimension} and, to our knowledge, no literature currently exists addressing asymptotic normality or efficiency for unbounded distributions.
This may be due to a structural mismatch between the transcendental logarithm and the algebraic operations underlying the combinatorial summations required to evaluate bias and variance, both complicating mathematical analysis and introducing approximation error.

This paper introduces an algebraic approach to entropy estimation resting on the harmonic number function that leads to an estimator with provably optimal theoretical guarantees on distributions with sufficiently fast tail decay.
The direction is inspired by the classic Kozachenko-Leonenko nearest-neighbor estimator for differential entropy \cite{kozachenko1987sample, gao2018demystifying}, whose analysis rests on an identity linking the digamma function to the expected log-probability in the neighborhood of a sample point.
The search for a discrete-distribution analogue prompted the investigation of harmonic-transformed binomial counts, $J(n)-J(M)$, where $M$ is a binomial random variable of $n$ trials and $J(m) := 1+1/2+\cdots+1/m$ is the $m$th harmonic number.\footnote{$H_m$ is the standard notation. We use $J$ to avoid confusion with entropy, denoted $H$.}
Proposition \ref{prop:mathind} uncovers a direct link to the logarithm of the success probability, the target in entropy estimation.
This identity allows for an exact, non-asymptotic characterization of the estimator's bias, bypassing the need for approximation bounds.
Second-order identities (Propositions \ref{prop:scndmnt}, \ref{prop:multiProd}) facilitate a precise analysis of the variance, in particular, the challenging covariance analysis of transformed multinomial counts.
The resulting algebraic expressions enable the isolation of the leading term of the estimator's variance, a task that has remained elusive for other methods on unbounded distributions.

These identities allow us to establish two main results.
First, we demonstrate that the $L_2$-minimax rate of convergence for entropy estimation is of order $1/n$ for the class of distributions with tails decaying as $p_j \lesssim j^{-2}$ (Theorem~\ref{thm:minimax}).
This rate arises from the variance convergence for both finite- and infinite-support distributions.
Second, we prove that the harmonic estimator is semiparametrically efficient for distributions with tails decaying as $p_j = o(j^{-2})$ (Theorem~\ref{clt}).
Specifically, the scaled, centered estimator $\sqrt n [\hat H (\mathbf{X}_n) -H(X)]$ converges weakly to a normal distribution whose variance is the theoretical lower bound.
Taken together, these results define the sharp statistical limits for this problem within the defined class and prove they are attained by the harmonic entropy estimator.

Additionally, the harmonic estimator is universally less biased than the plug-in (Proposition~\ref{prop:plugbias}).
The excess bias associated with the plug-in estimator arises from approximation error introduced by the logarithm in a binomial expectation.
Interestingly, bias-correction strategies for the plug-in estimator \cite{miller1955entropyBias, vu2007coverageEntropy} are first-order approximations of the harmonic estimator (see Section~\ref{sec:plugbias}).
We also demonstrate that the plug-in is minimax rate-optimal over $p_j \lesssim j^{-2}$ (Propositions~\ref{prop:plugbias} and~\ref{thm:varbound}) and we conjecture that the plug-in estimator is semiparametrically efficient for $p_j =o(j^{-2})$ due to its functional proximity to the harmonic estimator.
Despite these similarities, the harmonic estimator retains a distinct advantage over the plug-in: its analytical tractability and precision, positioning it well for future theoretical development.
Its reduced bias comes without added complexity since it shares the plug-in's simple, one-step functional form, potentially making it an attractive substitute in practical applications as well.

\textbf{Notation}: $a_j\lesssim b_j$ means there is a universal constant $C$ such that $\sup_j a_j/b_j \leq C$.
$a_j \asymp b_j$ is equivalent to $a_j \lesssim b_j$ and $b_j \lesssim a_j$.
$X^{(i)}$ is the $i$th sample point and $\mathbf{X}_n =\left(X^{(1)}, \cdots, X^{(n)}\right)$ is the sample of $n$ points.
$m^{(i)}$ is the number of occurrences of sample point $X^{(i)}$.
If $X^{(i)}=j$, we may also write $m_j$ for the number of occurrences of $j$ in the sample, and $m_j = m^{(i)}$.
$J(m) := 1+1/2+\cdots +1/m$ is the $m$th harmonic number with $J(0):=0$.

\subsection{Background and Prior Work} \label{sec:background}

Let $X$ be a discrete random variable with distribution $p:\N\rightarrow [0,1)$ over the natural numbers $\N = \{1, 2, 3, \dots\}$ such that $p_j =p(j)\geq 0$.
Without loss of generality, we assume that $p$ is monotonically decreasing.\footnote{This reflects a relabeling of the support rather than a restriction.  Entropy depends only on the probability values $\{p_j\}$; its estimators depend on the empirical counts they induce, and not on which symbol carries which probability.
Relabeling the support so that $p_1 \geq p_2 \geq \cdots$ leaves entropy and estimator performance unchanged.}
The Shannon entropy of $X$ is defined as $H(X) = -\E[\log p(X)]$.\footnote{$\log$ refers to the natural logarithm; we define $0\log 0 := 0$.}
Estimating this quantity from observed data is a classical problem in nonparametric statistics, with broad applications ranging from the construction of decision trees \cite{quinlan1996tree}, graphical model learning \cite{chow1968tree}, and feature selection \cite{peng2005feature} to genetics \cite{hausser2009entropy}, and algorithmic fairness \cite{martin2014fairness}.

A natural starting point for entropy estimation is the plug-in (or the maximum likelihood) estimator, $\hat{H}_{\plug} = -\sum_j \hat{p}_j \log \hat{p}_j$, where $\hat{p}_j$ is the empirical frequency of symbol $j$.
In the simplest case of a known, finite support where the sample size is large, this estimator is well-behaved and asymptotically efficient \cite{basharin1959discEntropy}.
While universally consistent \cite{antos2001convergence}, its performance degrades severely outside of this ideal regime due to negative bias.
Although biased estimation is unavoidable for finite samples \cite{paninski2003estimation}, its bias can be characterized as $-(s-1)/(2n) + \mathcal{O}(n^{-2})$ \cite{miller1955entropyBias}, where $s$ is the number of support points in $p$.
While this suggests a simple correction, the approach fails when $s$ is unknown, large, or infinite, which is often the case in practice.
In this more challenging setting, the bias is dominated by unobserved symbols in the sample, and the problem of entropy estimation becomes fundamentally linked to the missing mass problem \cite{good1953missingmass, orlitsky2016optimal}.

This core difficulty has motivated a rich body of research aimed at correcting the plug-in estimator's flaws.
One line of work addresses the large alphabet regime.
Recognizing that the logarithmic transformation is particularly problematic for small, non-zero probabilities, Jiao et al. \cite{jiao2015minimax} and Wu et al. \cite{wu2016large} developed sophisticated two-stage estimators.
They first classify symbols based on their empirical counts and then apply either a bias-corrected plug-in estimate for high-frequency symbols or a carefully chosen polynomial approximation for low-frequency symbols.
These methods achieve minimax optimality for finite-support distributions where $s$ may grow with $n$, establishing the benchmark for that class.

Another approach, taken by Valiant and Valiant \cite{valiant2013unseen}, tackles the problem by first estimating the full histogram of the distribution, including the probabilities of unobserved symbols, then estimating entropy.
Concurrently, Bayesian methods have offered an alternative paradigm.
By placing a prior over the space of distributions, such as a Dirichlet-mixture prior \cite{nemenman2001entropy} or a Pitman-Yor process mixture for infinite support \cite{archer2014bayesian}, these methods naturally average over the uncertainty of unobserved symbols, thereby mitigating bias.

Despite these significant advances, the literature lacks a unified framework that provides sharp, minimax-optimal rates for entropy estimation over broad classes of distributions that include those with infinite support.
Existing frequentist methods either rely on finite-support assumptions or provide rate-suboptimal error bounds \cite{antos2001convergence, zhang2012entropy, cohen2023dimension}.
While Bayesian methods naturally handle missing mass, their minimax optimality typically relies on the prior correctly matching the true tail decay, which is unknown in the frequentist setting.

\subsection{Main Results}\label{mainresults}

Leveraging the definition of entropy as an expectation, we approach the estimation problem as an empirical mean of local entropy estimates over all sample points.
From this paradigm, the central challenge is estimating the local entropy $-\log p(X^{(i)})$, similar to many continuous estimators.
As Proposition \ref{prop:mathind} shows, the expectation of harmonic-transformed empirical counts is equal to the local entropy minus a remainder governed by sample size.
This relationship motivates our estimator.

Let $\mathbf X_n :=\paren{X^{(1)},\dots,X^{(n)}}$ be an i.i.d.\ sample from $p$ and $m^{(i)} = \sum_{k=1}^n I(X^{(k)} = X^{(i)})$ be the number of occurrences of the symbol $X^{(i)}$ in the sample $\mathbf{X}_n$.
Define the harmonic estimator as
\begin{equation}\label{eq:propest}
	\hat H(\mathbf X_n) = \frac{1}{n} \sum_{i=1}^n \brac{J(n-1) -J\paren{m^{(i)}-1}}~.
\end{equation}

Our first main result establishes the $L_2$-minimax rate of convergence for the class of distributions with quadratically decaying tails.
\begin{thm}[$L_2$-Minimax Risk] \label{thm:minimax}
	Let $\mathcal H$ be the set of entropy estimators and $\mathcal{P} := \{p: p_j \lesssim j^{-2}\}$ be the class of distributions with quadratically decaying tails.
	Then as $n\rightarrow\infty$
	\begin{equation}
		\inf_{f\in\mathcal H} \sup_{p\in \mathcal P} \E_p\brac{ \paren{f(\mathbf X_n) -H(X) }^2 } \asymp \frac{1}{n}~.
	\end{equation}
\end{thm}
This result complements the minimax results of \cite{jiao2015minimax, wu2016large}, which apply to different, overlapping settings.
The results in \cite{jiao2015minimax, wu2016large} allow support size to grow with sample size where variance dominates when $s \ll n$ and squared bias dominates otherwise.
Our result applies to the fixed parameter space $\mathcal P$ as $n$ increases, where the convergence rate of the variance always governs.
However, all results agree in convergence rate for fixed, finite support as $n$ increases.

Our second main result demonstrates that under slightly stronger tail conditions, the harmonic estimator is semiparametrically efficient, as the asymptotic variance is the minimum possible for regular estimators, matching the semiparametric variance bound for this problem, $\var[\log p(X)]$ (Section~\ref{sec:clt}).
\begin{thm}[Semiparametric Efficiency]\label{clt}
	If $p_j = o\paren{j^{-2}}$, then $\var[\log p(X)]<\infty$ and
	\begin{equation}
		\sqrt n\brac{\hat H\paren{\mathbf X_n}-H(X)} \leadsto \text{N}(0, \var[\log p(X)])
	\end{equation}
	 as $n\rightarrow \infty$ where $\leadsto$ indicates convergence in distribution.
\end{thm}

The condition $p_j\lesssim j^{-2}$ represents the boundary where the squared bias becomes $\mathcal{O}(1/n)$, matching the variance.
While this achieves the optimal rate, the bias is of the same order as the standard deviation, so $\sqrt n[\hat H(\mathbf{X}_n) - H(X)]$ does not converge to a zero-mean normal.
Thus, the stronger condition $p_j = o(j^{-2})$ is necessary for the $\sqrt n$-scaled bias to vanish, achieving asymptotic normality and efficiency.

\section{Risk Analysis of the Harmonic Entropy Estimator}\label{theory}

Central to this analysis are three novel, combinatorial identities (Propositions \ref{prop:mathind}, \ref{prop:scndmnt}, \ref{prop:multiProd}) that yield exact, non-asymptotic expressions for first- and second-order moments of harmonic-transformed binomial and multinomial counts.
These identities circumvent the need for inequality bounds on estimator moments, simplifying mathematical analysis and eliminating approximation error.
The approach follows an algebraic strategy from Knuth \cite[Sec. 1.2.7]{knuth1997art} on summations involving harmonic numbers and binomial coefficients.
The estimator has a more parsimonious form if expressed with the digamma function.\footnote{$\psi(z) =J(z-1)-\gamma$ for $z\in\N$ where $\gamma$ is the Euler-Mascheroni constant. $\hat H(\mathbf{X}_n) = \sum_i [\psi(n) -\psi(m^{(i)})]/n$.}
Despite its similarity to other digamma-based estimators \cite{kozachenko1987sample, grassberger2003entropy}, we highlight the simpler harmonic number function because it is sufficient to derive our central identities.

These moment identities alone require no restrictions on the underlying distribution:
One can fully evaluate the combinatorial sums involved in the harmonic estimator's expected value and variance for any distribution with these identities.\footnote{The resulting expressions still contain expectations with respect to the underlying distribution.}
However, explicit estimator characterizations on unbounded random variables require some control over the probability that extreme values appear.
Our results depend on a rate parameter $\alpha \in (0,1)$, where $p_j \lesssim j^{-1/\alpha}$ characterizes how quickly a distribution's tail decays.
Distributions with tails decaying faster than any power law, such as those with exponential decay or finite support, satisfy $p_j = o(j^{-1/\alpha})$ for every $\alpha \in (0,1)$;
thus, results hold for such distributions with arbitrarily small $\alpha$ though not at $\alpha=0$ itself.
The following results build on this foundation, providing an exact characterization of the harmonic estimator's asymptotic $L_2$ risk in terms of $\alpha$.

\begin{cor} [$L_2$ Risk]\label{thm:risk}
	If $p_j \lesssim j^{-1/\alpha}$ with $\alpha\in(0,1)$, then
	\begin{equation}\label{ineq:risk}
		\E\brac{ \paren{ \hat H(\mathbf X_n) -H(X)}^2 } = \frac{ \var\brac{ \log p(X) } }{n} +\mathcal{O}\paren{ \frac{1}{n^{2(1-\alpha)}} }~.
	\end{equation}
\end{cor}
\begin{proof}
	Decomposing $L_2$ risk into variance and squared bias, then applying Proposition~\ref{bias} and Theorem~\ref{var} (both proved later in this section) to each respectively,
	\begin{align*}
		&\E\brac{ \paren{ \hat H(\mathbf X_n) -H(X)}^2 } \\
		&= \E\brac{ \paren{ \hat H(\mathbf X_n) -\E\brac{\hat H(\mathbf{X}_n)} }^2 } +\paren{ \E\brac{\hat H(\mathbf{X}_n)} -H(X)}^2 \\
		&= \frac{ \var\brac{ \log p(X) } }{n} +\mathcal{O}\paren{ \frac{\log n}{n^{2-\alpha}} } +\mathcal{O}\paren{ \frac{1}{n^{2(1-\alpha)}} }~.
	\end{align*}
	Note that the second remainder dominates the first for every $\alpha\in (0,1)$.
\end{proof}

The asymptotic equality reveals a phase transition: 
For $\alpha \in (0,1/2]$ (tails decaying as fast or faster than $j^{-2}$), the risk is dominated by the variance and converges at the $\mathcal{O}(1/n)$ rate.
For $\alpha \in (1/2,1)$ (heavier tails), the risk is dominated by the squared bias and converges at the slower $\mathcal{O}(n^{-2(1-\alpha)})$ rate.
The phase transition characterizes the influence of bias in infinite-support settings, as the bias rate depends on tail decay whereas the variance rate remains $\mathcal{O}(1/n)$.

\subsection{Bias}
A primary challenge in nonparametric entropy estimation has been the difficulty of accurately characterizing the bias of the plug-in estimator, particularly in the infinite-support setting.
A standard derivation of the plug-in estimator's bias via Taylor expansion illuminates this challenge.
Consider $-\E\brac{\hat p_j \log \hat p_j}$ for symbol $j$ with empirical probability $\hat p_j := m_j/n$.
Expanding $f(x) = -x\log x$ about $p_j$, the $k$th-order derivative is $f^{(k)}(p_j) = (-1/p_j)^{k-1} (k-2)!$ for $k\geq 2$.
Taking the expectation term by term, the first-order term vanishes since $\E[\hat p_j] =p_j$; the second-order term evaluates to $-(1-p_j)/(2n)$ using $\var[\hat p_j] = p_j(1-p_j)/n$.
Together with the remaining terms,
\begin{equation}
	-\E\brac{ \hat p_j \log \hat p_j} 
	= -p_j \log p_j -\frac{1-p_j}{2n}-\sum_{k=3}^\infty \frac{ \E\brac{\paren{p_j -\hat p_j}^k } }{k(k-1) p_j^{k-1} } ~. \label{eq:plugtaylor}
\end{equation}
Summing~\eqref{eq:plugtaylor} over $j$ gives the plug-in estimator's bias explicitly, with leading term $-\sum_j (1-p_j)/(2n) = -(s-1)/(2n)$ dependent on the distribution's support size.
The higher-order terms, which include central moments $\E[(p_j -\hat p_j)^k]$, are not easily characterized.
Unfortunately, any remainder bound will increase as $p_j$ approaches zero due to $p_j^{k-1}$ in the denominator.

The structure of the harmonic estimator allows for an exact characterization of its bias, without an approximation bound.
The key is the following identity for the expected value of a harmonic-transformed binomial random variable.
\begin{prop}[Expected Value of Harmonic-Transformed Binomial]\label{prop:mathind}
	Let $M$ be a binomial random variable of $n\in\N$ trials and success probability $p\in (0,1]$.
	Then
	\begin{equation}
		\E[J(n) -J(M)] = \sum_{m=1}^n \frac{(1-p)^m}{m} = -\log p -\sum_{m=n+1}^\infty \frac{(1-p)^m}{m}~.
	\end{equation}
\end{prop}

\begin{proof}[Proof of Prop. \ref{prop:mathind}]
	The proof relies on pairing the harmonic recursion $J(m+1) =J(m) +\frac{1}{m+1}$ with the absorption identity $\frac{n+1}{m+1} \binom{n}{m} =\binom{n+1}{m+1}$ to evaluate the binomial sum.

	Let $S_n := \sum_{m=0}^n J(m) \binom{n}{m} p^m(1-p)^{n-m}$ denote $\E[J(M)]$ for $M\sim\text{Binomial}(n, p)$, and proceed by induction on $n$.
	The base case holds because $S_1=p$.

	For the inductive step, we apply Pascal's identity $\binom{n+1}{m} = \binom{n}{m} +\binom{n}{m-1}$ then re-index the second sum:
	\begin{align*}
		S_{n+1}
		&= \sum_{m=0}^{n} J(m) \binom{n}{m} p^m(1-p)^{n+1-m} +\sum_{m=0}^{n} J(m+1) \binom{n}{m} p^{m+1}(1-p)^{n-m} \nonumber \\
	\intertext{We apply the harmonic recursion to $J(m+1) =J(m)+1/(m+1)$ in the second sum and distribute, then factor the remaining sums containing $J(m)$:}
		&= (1-p) \sum_{m=0}^{n} J(m) \binom{n}{m} p^m(1-p)^{n-m} +p \sum_{m=0}^{n} J(m) \binom{n}{m} p^{m}(1-p)^{n-m} \nonumber \\
		&\quad +\sum_{m=0}^{n} \frac{1}{m+1} \binom{n}{m} p^{m+1}(1-p)^{n-m} \nonumber \\
	\intertext{Combine $(1-p)S_n +pS_n$.  The remaining fractional term simplifies using the absorption identity $\frac{1}{m+1} \binom{n}{m} =\frac{1}{n+1} \binom{n+1}{m+1}$, which further reduces via the binomial theorem:}
		&= S_n +\frac{1}{n+1} \sum_{m=1}^{n+1} \binom{n+1}{m} p^m(1-p)^{n+1-m} 
		= S_n +\frac{ 1-(1-p)^{n+1}}{n+1} ~.
	\end{align*}

	Telescoping from $S_1=p$ gives $S_n =\sum_{k=1}^n \frac{1-(1-p)^k}{k} =J(n) -\sum_{k=1}^n \frac{(1-p)^k}{k}$. 
	Thus,
	\begin{equation*}
		\E[J(n) -J(M)] = J(n) -S_n = \sum_{k=1}^n \frac{(1-p)^k}{k} ~.
	\end{equation*}
	The summation is the first $n$ terms of the Taylor series for $-\log x$ centered at $1$, which converges for $x\in(0,2]$.
\end{proof}

This proposition is our bridge to the logarithm of the success probability.
Specifically, the exactness of the harmonic recursion property combined with the absorption identity yields the $n$-th degree Taylor polynomial of $-\log p_j$ centered at 1.
The identity is useful in the context of entropy estimation because it resolves the binomial expectation and holds uniformly over $n\in\N$ and $p_j\in (0,1]$.
Proposition~\ref{prop:mathind} allows us to provide an exact, non-asymptotic characterization of the harmonic estimator's bias and practical bounds.

\begin{prop} \label{bias}
	If $H(X) <\infty$, then
	\begin{equation} \label{eq:genbias}
		\E\brac{\hat{H}(\mathbf X_n)} -H(X) = -\sum_{k=n}^\infty \frac{\E\brac{(1-p_X)^k}}{k}~.
	\end{equation}
	If $p_j \lesssim j^{-1/\alpha}$ for $\alpha \in (0,1)$, then for some constant $C>0$ that does not depend on $n$,
	\begin{equation} \label{eq:biasrate}
		\abs{ \E\brac{\hat{H}(\mathbf X_n)} -H(X) } \leq \frac{C}{n^{1-\alpha}}~.
	\end{equation}
	Additionally, if $p_j =o(j^{-2})$, then 
	\begin{equation} \label{eq:biassq}
		\abs{ \E\brac{\hat{H}(\mathbf X_n)} -H(X) } =o\paren{\frac{1}{\sqrt n}}
	\end{equation}
	as $n\rightarrow\infty$~.
\end{prop}

\begin{proof}[Proof of equation~\ref{eq:genbias}]
	The count $m^{(i)}$ is the number of times the value $X^{(i)}$ appears in the sample: $m^{(i)} = 1 + \sum_{k=1, k\neq i}^n I(X^{(k)} = X^{(i)})$, where the leading 1 accounts for $X^{(i)}$ itself.
	Conditioned on $X^{(i)}$, each of the remaining $n-1$ points independently matches its value with probability $p_{X^{(i)}}$, so $m^{(i)} - 1 \sim \text{Binomial}(n-1, p_{X^{(i)}})$.
	
	Let $X$ be an independent copy from the same distribution, with $m_X \sim \text{Binomial}(n-1, p_X)$ given $X$. By linearity, the expected value of the estimator simplifies to a single term:
	\begin{align}
		\E\brac{\hat{H}(\mathbf X_n)}
		&= \E\brac{\frac{1}{n} \sum_{i=1}^n \brac{ J(n-1) - J(m^{(i)}-1) } } \nonumber \\
		&= \E\brac{ J(n-1) - J(m_X) }~. \nonumber \\
	\intertext{Applying the law of total expectation by conditioning on $X$, and invoking Proposition~\ref{prop:mathind}:}
		&= \E\brac{ \E\brac{ J(n-1) - J(m_X) \mid X } } \nonumber \\
		&= \E\brac{ -\log p_X - \sum_{k=n}^\infty \frac{(1-p_X)^k}{k} }~. \nonumber \\
	\intertext{Recognizing that $-\E[\log p_X] = H(X)$ and pushing the expectation inside the sum completes the derivation:}
		&= H(X) - \sum_{k=n}^\infty \frac{\E\brac{ (1-p_X)^k }}{k}~. \nonumber
	\end{align}
	The bounds in \eqref{eq:biasrate} and \eqref{eq:biassq} are established via integral comparison, detailed in Appendix~\ref{apdx:bias}.
\end{proof}

The general bias expression~(\ref{eq:genbias}) vanishes as $n$ increases for any distribution with $H(X)<\infty$, not only those satisfying the tail decay condition, illustrating why the harmonic estimator is well-suited for infinite-support problems.
Moreover, the expression has a natural probabilistic interpretation that connects it to the missing mass literature \cite{good1953missingmass, berend2012missing, ben2017concentration}.
The quantity $\E\brac{(1-p_X)^k}$ is the expected probability mass of symbols not included in a sample of size $k$.
Thus, the bias is a weighted sum of the expected missing mass at all sample sizes $k\geq n$, with weights $1/k$.
This representation makes the harmonic estimator's bias behavior intuitive: entropy is underestimated by a quantity that reflects the cumulative expected missing mass beyond the observed sample, and this deficit vanishes as the sample size grows to cover the support.
The exact expression also provides a natural foundation for bias correction by employing a Good-Turing style estimator for the missing mass.
In contrast to the plug-in estimator, the harmonic estimator's bias does not directly depend on the expected number of distinct symbols in a sample, as we will show in Section~\ref{sec:plugbias}.

The rate bound~\eqref{eq:biasrate} explicitly shows that the estimator's bias is sensitive to a distribution's power-law tail behavior as the sample increases, with the rate of convergence deteriorating as $\alpha$ approaches $1$.
A bias calculation on a zeta-distributed random variable reveals that this bound is tight for $\alpha\in (0,1)$ (Proposition \ref{prop:tightbias}), meaning that this bound cannot be improved without additional assumptions.
The boundary case $\alpha = 1/2$, corresponding to $p_j\asymp j^{-2}$, is of particular significance because at this boundary the squared bias equals the variance rate, marking the transition between minimax optimality and semiparametric efficiency, as we develop in Sections~\ref{sec:minimax}~and~\ref{sec:clt}.

\subsection{Variance} \label{sec:var}

The primary goal of this section is to establish the exact leading term of the estimator's variance.
Achieving this requires directly analyzing the variance, in contrast to the variance-bounding techniques common in the literature, which do not isolate the leading term \cite{antos2001convergence, jiao2015minimax, wu2016large, cohen2023dimension}.
Directly analyzing the variance of entropy estimators can be challenging in discrete settings due to the statistical dependence among empirical counts,\footnote{The number of occurrences of each outcome in the sample.} which must sum to a fixed sample size.
Our approach navigates this challenge by analyzing the local variance separately from the local covariance.
Again, we apply mathematical induction to fully evaluate the second-order moments associated with the local variance and covariance, yielding a precise asymptotic characterization.

The leading term of the harmonic estimator's variance is what makes it possible to establish semiparametric efficiency.
As we demonstrate in Section~\ref{sec:clt}, the leading term's coefficient $\var[\log p(X)]$ is the estimator's asymptotic variance, which achieves the semiparametric efficiency bound for this estimation problem.
The following theorem provides this precise characterization of the harmonic estimator's variance:

\begin{thm}\label{var}
	If $p_j \lesssim j^{-1/\alpha}$ for $\alpha \in (0,1)$, then $\var[\log p(X)]<\infty$ and
	\begin{equation}
		\var\brac{\hat H\paren{\mathbf X_n}} = \frac{\var\brac{\log p(X)}}{n} +\mathcal O\paren{\frac{\log n}{n^{2-\alpha}} }~.
	\end{equation}
\end{thm}

\begin{proof}[Proof Sketch]
	We begin by decomposing the estimator's variance into its local variance and covariance components by exchangeability:
	\begin{align*}
		&\var\brac{ \hat{H}(\mathbf X_n) } =\var\brac{ \frac{1}{n} \sum_{i=1}^n \brac{ J(n-1) -J\paren{m^{(i)} -1} } } \\
		&= \frac{1}{n} \var\brac{ J(n-1) -J\paren{m^{(1)}-1} } \\
		&\quad +\frac{n-1}{n} \cov\brac{ J(n-1)-J\paren{m^{(1)}-1}, J(n-1)-J\paren{m^{(2)}-1} }
	\end{align*}
	We analyze the local variance and covariance separately.

	\textit{Local Variance.}
	Recall that $m^{(1)}-1 \sim \text{Binomial}\paren{n-1, p(X^{(1)})}$ conditioning on $X^{(1)}$.
	Let $X := X^{(1)}$ and $m_X\sim \text{Binomial}\paren{n, p(X)}$ given $X$, relabeling the sample size for brevity, so that we work with $J(n) -J(m_X)$ in place of $ J(n-1) -J\paren{m^{(1)}-1} $.
	Note that this has no impact on the asymptotic calculations.

	Writing the local variance as
	\begin{align*}
		\var\brac{ J(n) -J(m_X)} 
		&= \E\brac{ \paren{ J(n) -J(m_X) }^2 } -\paren{ \E\brac{  J(n) -J(m_X) } }^2,
	\end{align*}
	we analyze each term separately.
	Proposition~\ref{prop:scndmnt} uses mathematical induction to provide an exact expression for the second moment of the local entropy estimate:
	\begin{align*}
		&\E \brac{ (J(n)-J(m_X))^2 }
		= \E_X\brac{ \E\brac{  (J(n)-J(m_X))^2 \middle| X} } \\ 
		&= \E_X\brac{ \sum_{m=2}^n \sum_{k=1}^{m-1} \frac{ (1-p_X)^m }{k(m-k)}  +\sum_{m=1}^n \sum_{k=n-m+1}^n \frac{ (1-p_X)^m }{mk} }~. \\
		\intertext{Observe that the first double summation comprises the leading terms of the Cauchy product formed by squaring the Taylor series of $\log p_X$ about $1$:}
		&= \E\brac{ (\log p_X)^2 } -\sum_{m=n+1}^\infty \sum_{k=1}^{m-1} \frac{ \E\brac{(1-p_X)^m} }{k(m-k)} +\sum_{m=1}^n \sum_{k=n-m+1}^n \frac{ \E\brac{(1-p_X)^m} }{mk} ~.
	\end{align*}
	Proposition~\ref{prop:varremainders} bounds the remaining terms as $\mathcal{O}(\log(n)/n^{1-\alpha})$.

	Moving to the squared expectation, we apply Proposition~\ref{prop:mathind}, bounding the remaining terms with Proposition~\ref{bias}:
	\begin{align*}
		\paren{ \E\brac{ J(n)-J(m_X) } }^2 
		&= \paren{ -\E\brac{\log(p_X)} -\sum_{m=n+1}^\infty \frac{ \E\brac{(1-p_X)^m} }{m} }^2 \\
		&= \paren{ \E\brac{\log(p_X)} }^2 +\mathcal{O}\paren{\frac{1}{n^{1-\alpha}}}~.
	\end{align*}

	Together these calculations establish the local variance:
	\begin{align*}
		\var\brac{ J(n) -J(m_X)} 
		&=\E\brac{ (\log p_X)^2 } -\paren{ \E\brac{\log p_X} }^2 +\mathcal{O}\paren{\frac{\log(n)}{n^{1-\alpha}}} \\
		&=\var[\log p_X] +\mathcal{O}\paren{\frac{\log(n)}{n^{1-\alpha}}}~.
	\end{align*}

	\textit{Local Covariance.}
	Let $(m_X, m_Y, n-m_X-m_Y)$ be a multinomial random vector with probabilities $(p(X), p(Y),1-p(X)-p(Y))$ given $X :=X^{(1)}$ and $Y :=X^{(2)}$, again changing the sample size for brevity.\footnote{This sketch ignores the $X=Y$ case which introduces an error of order $\mathcal{O}(\log(n)/n^{2-\alpha})$, which is absorbed into the existing remainders (Proposition~\ref{prop:addterms}).}
	Using the law of total covariance,
	\begin{align*}
		&\cov\brac{J(n)-J(m_X), J(n)-J(m_Y)} \\
		&= \E_{X,Y}\brac{ \cov\brac{J(n)-J(m_X), J(n)-J(m_Y) \middle| X, Y} } \\
		&\quad +\cov\brac{ \E\brac{ J(n)-J(m_X) \middle| X,Y }, \E\brac{ J(n)-J(m_Y) \middle| X,Y } }~. \\
		\intertext{The covariance of expectations is zero because the expectations are functions of independent random variables.
		Writing the joint expectation over $X$ and $Y$ as separate sums over $X=Y$ and $X\neq Y$, we have}
		&= \sum_i p_i^2 \var\brac{J(n)-J(m_i)} +\sum_i \sum_{j\neq i} p_i p_j \cov\brac{J(n)-J(m_i), J(n)-J(m_j)}~. 
	\end{align*}
	We proceed by analyzing the cases where $X=Y$ and $X\neq Y$ separately.

	Beginning with the $X=Y$ case, recall from the variance analysis that $\E\brac{ (J(n) -J(m_i))^2 }$ and $\paren{ \E[J(n) -J(m_i)] }^2$ both resolve to Taylor polynomial products approaching $(\log p_i)^2$ plus a remainder using Propositions~\ref{prop:scndmnt} and~\ref{prop:mathind}, respectively.
	Because the variance is within the summation, common terms from the two representations of $(\log p_i)^2$ cancel, leaving a remainder:
	\begin{align*}
		&\sum_i p_i^2 \var\brac{J(n)-J(m_i)} \\
		&= \sum_i p_i^2 \brac{ \E\brac{ (J(n) -J(m_i))^2 } -\paren{ \E[J(n) -J(m_i)] }^2 } \\
		&= \sum_{m=1}^n \sum_{k=n-m+1}^n \frac{ \sum_i p_i^2 (1-p_i)^m }{mk} +\mathcal O\paren{ \frac{1}{n^{2-\alpha}}} ~. 
	\end{align*}
	We retain this summation exactly rather than bound it here, since it must be combined with the covariance's other remainder in order to achieve the final bound.

	The $X\neq Y$ case is considerably more involved with much of the analysis compartmentalized into other propositions.
	We begin by evaluating $\E[(J(n)-J(m_i)) (J(n)-J(m_j))]$.
	Proposition \ref{prop:multiProd}, a third exact algebraic identity, simplifies this joint expectation.
	The result combines partial sums of dilogarithms,\footnote{$\li2(z) = \sum_{k=1}^\infty z^k/k^2$.} $\zeta(2)$,\footnote{Riemann Zeta function: $\zeta(s) =\sum_{k=1}^\infty 1/k^s$.} and symmetric double sums that resemble, but do not equal, the Cauchy products of logarithm Taylor series for $(\log p_i) (\log p_j)$.
	By representing this symmetric double series as iterated integrals over the simplex, Proposition~\ref{prop:sumtoint} links these symmetric double sums with an expression of logarithm products, dilogarithms, and $\zeta(2)$.
	Together, these propositions show that $\E[(J(n)-J(m_i)) (J(n)-J(m_j))]$ is equal to a polynomial approaching $(\log p_i) (\log p_j)$ plus a remainder.

	Relying on this, Proposition~\ref{covarMulti} simplifies $\cov[J(n)-J(m_i), J(n)-J(m_j)]$ by canceling common terms among $\E[(J(n)-J(m_i)) (J(n)-J(m_j))]$ and $\E[J(n)-J(m_i)] \E[J(n)-J(m_j)]$, as both are polynomial representations of $(\log p_i) (\log p_j)$, leaving a remainder:
	\begin{align*}
		&\sum_i \sum_{j\neq i} p_i p_j \cov\brac{J(n)-J(m_i), J(n)-J(m_j)} \\
		&= \sum_i \sum_{j\neq i} p_i p_j \paren{ \E\brac{ (J(n)-J(m_i)) (J(n)-J(m_j))} - \E\brac{J(n)-J(m_i)} \E\brac{ J(n)-J(m_j)}} \\
		&= -\sum_{m=n+1}^\infty \frac{ \sum_i p_i (1-p_i) }{m^2} +\mathcal O\paren{\frac{1}{n^{2-\alpha}}}~. 
	\end{align*}
	Again, we retain this summation since it is not bounded by $o(1/n)$.

	We finish the local covariance analysis by combining the remainders.
	\begin{align*}
		&\cov\brac{J(n)-J(m_X), J(n)-J(m_Y)} \\
		&= \sum_i p_i^2 \var\brac{J(n)-J(m_i)} +\sum_i \sum_{j\neq i} p_i p_j \cov\brac{J(n)-J(m_i), J(n)-J(m_j)} \\
		&= \sum_{m=1}^n \sum_{k=n-m+1}^n \frac{ \sum_i p_i^2 (1-p_i)^m }{mk} -\sum_{m=n+1}^\infty \frac{ \sum_i p_i (1-p_i) }{m^2} +\mathcal O\paren{\frac{1}{n^{2-\alpha}}} ~.\\
		\intertext{For the first sum, apply the identity $\sum_{k=n-m+1}^n 1/k = m\sum_{k=n+1}^\infty 1/(k(k-m))$ (derived with partial fractions).
		For the second sum, multiply by the geometric series written as $1 =\sum_{m=0}^\infty p_i(1-p_i)^m$.}
		&= \sum_{m=1}^n \sum_{k=n+1}^\infty \frac{ \sum_i p_i^2 (1-p_i)^m }{k(k-m)} - \sum_{m=1}^\infty \sum_{k=n+1}^\infty \frac{ \sum_i p_i^2 (1-p_i)^m }{k^2} +\mathcal O\paren{\frac{1}{n^{2-\alpha}}} ~.
	\end{align*}
	We combine the first summation and the terms corresponding to $m=1,\dots,n$ of the second summation, then apply Proposition~\ref{prop:varremainders}, and bound the remaining summations with integrals, yielding an upper bound of $\mathcal{O}(1/n^{2-\alpha})$.

	Putting the local variance and covariance calculations together,
	\begin{align*}
		\var\brac{\hat H(\mathbf X_n)} 
	&= \frac{1}{n} \brac{ \var[\log p(X)] +\mathcal O\paren{\frac{\log n}{n^{1-\alpha}}} } +\frac{n-1}{n}\cdot \mathcal O\paren{\frac{1}{n^{2-\alpha}}} \\
	&= \frac{\var[\log p(X)]}{n} +\mathcal O\paren{\frac{\log n}{n^{2-\alpha}}}~.
	\end{align*}
	See Appendix \ref{apdx:var} for the full proof.
\end{proof}

Common alternative techniques for bounding variance, such as \textit{Poissonization} or the Efron-Stein inequality, often provide a simpler path to establishing a rate of convergence but are too blunt to capture the exact leading constant.
Poissonization, for instance, simplifies the analysis by rendering the empirical counts between symbols independent by assuming the sample size is a Poisson-distributed random variable, but obscures the true covariance structure.
Translating the subsequent result back to a fixed-sample-size model can at most bound but not isolate the leading term.
Similarly, while the Efron-Stein inequality can readily establish a $1/n$ rate (Proposition \ref{thm:varbound}), it is a general bound that cannot establish the exact asymptotic constant, $\var[\log p(X)]$, required for proving efficiency.
Thus, this direct approach is necessary for demonstrating that the harmonic estimator achieves the fundamental efficiency bound.

Central to this proof is a correspondence between the second-order harmonic-transformed binomial and multinomial moments and products of logarithm-transformed success probabilities: $\E[(J(n)-J(m_i))^2] \mapsto (\log p_i)^2$ (Proposition~\ref{prop:scndmnt}) and $\E[(J(n)-J(m_i))(J(n)-J(m_j))] \mapsto (\log p_i)(\log p_j)$ (Propositions~\ref{prop:multiProd} and~\ref{prop:sumtoint}), omitting remainder terms.
The proof sketch makes clear the correspondence's contribution to establishing the leading coefficient from the local variance.
For the covariance, the correspondence's primary purpose is canceling like terms between the joint expectations (second-order moments) and products of expectations, both of which are polynomial representations of the same products of log-probabilities.
Interestingly, this cancellation is not sufficient to bound the components of the covariance due to the remainders from the second-order moments; only by combining the remainders does the covariance become asymptotically negligible.
This illustrates that while the correspondence mentioned above is central, the algebraic precision of our approach to express the remainders exactly is just as necessary for establishing this result.

\section{Minimax Optimality} \label{sec:minimax}

This section benchmarks the statistical limits for estimating Shannon entropy in the worst-case scenario over the class of distributions with at least quadratically decaying tails, $\mathcal{P} = \{p: p_j \lesssim j^{-2}\}$.
Using the statistical decision theory framework, we construct a minimax lower bound as a function of sample size, and demonstrate that the harmonic estimator's risk matches this rate.

We begin with a lower bound for all non-constant distributions using Le Cam's method \cite[Section 2.4.2]{tsybakov2009nonparametric}.

\begin{prop}[$L_2$-Risk Lower Bound]\label{thm:minimaxlower}
	Let $\mathcal H$ be the set of entropy estimators and $\mathcal U=\curly{p: p_j\neq 1, j\in\N}$ be the class of non-constant distributions.
	Then
	\begin{equation}
		\inf_{f \in \mathcal H} \sup_{p\in \mathcal U} \E_p\brac{\paren{ f(\mathbf X_n) -H(X) }^2} \gtrsim \frac{1}{n}~.
	\end{equation}
\end{prop}

\begin{proof}[Proof of Proposition \ref{thm:minimaxlower}]
	The proof employs Le Cam's standard two-point method \cite[Section 2.4.2]{tsybakov2009nonparametric}, constructing two ``nearby'' distributions $p$ and $q$ that cannot be reliably distinguished from a sample of size $n$. 

	Fix $\epsilon \in (0,1)$ and let $p=\paren{\frac{1}{3}, \frac{2}{3}}$ and $q=\paren{\frac{1+\epsilon}{3}, \frac{2-\epsilon}{3}}$.
	We calculate the Kullback-Leibler divergence and bound it using the inequality $\log x \leq x-1$,
	\begin{equation*}
		\text{KL}(p, q)
		= \frac{1}{3} \log\paren{ \frac{1}{1 +\epsilon}} +\frac{2}{3} \log\paren{\frac{2}{2-\epsilon}} \leq \frac{ \epsilon^2 }{(2-\epsilon) (1+\epsilon) } \leq \epsilon^2~.
	\end{equation*}

	Next, we calculate the difference of entropies of $q$ and $p$, then using $\log x \geq (x-1)/x$ to bound the difference below,
	\begin{align*}
		H(q) -H(p)
		= \frac{1}{3} \brac{\epsilon \log\paren{\frac{2-\epsilon}{1+\epsilon}} -\log\paren{\frac{(1+\epsilon)(2-\epsilon)^2}{4}} } 
		\geq \frac{\epsilon \log 2 }{3} -\epsilon^2~.
	\end{align*}

	With Le Cam's two-point method, choosing $\epsilon = \frac{1}{\sqrt n}$,
	\begin{align*}
		&\inf_{f\in \mathcal H} \sup_{p\in\mathcal U} \E_p \brac{\paren{ f(\mathbf X_n) -H(X)}^2} \\
		&\geq \frac{1}{4} \brac{H(p) -H(q)}^2 e^{-n \text{KL}(p, q)} \geq \frac{1}{4} \paren{ \frac{\log 2 }{3\sqrt n} -\frac{1}{n} }^2 e^{-1} \gtrsim \frac{1}{n} ~.
	\end{align*}
\end{proof}

It may seem surprising that distributions with only two support points are sufficient to establish a universal lower bound for all non-constant distributions, $\mathcal U$.
While any pair of distributions in $\mathcal U$ technically yields \emph{a} lower bound via Le Cam's method, the challenge is finding a pair that achieves the \emph{greatest} lower bound.
Because $p$ and $q$ are very simple, one might suspect their resulting $1/n$ rate is loose for the more complex, infinite-support members of $\mathcal U$.
However, Theorem~\ref{thm:minimax} demonstrates that for $\mathcal P \subset \mathcal U$ this bound is in fact sharp:
\begin{equation}
	\inf_{f \in \mathcal{H}} \sup_{p \in \mathcal{P}} \E_p[ (f(\mathbf{X}_n) - H(X))^2 ] \asymp \frac{1}{n}
\end{equation}
where $\mathcal{H}$ is the set of all entropy estimators.
\begin{proof}[Proof of Theorem \ref{thm:minimax}]
	The proof proceeds by establishing matching lower and upper bounds for the $L_2$-risk on $\mathcal P$.
	The lower bound in Proposition \ref{thm:minimaxlower} is $\Omega(1/n)$, setting a benchmark for best possible performance for any estimator over $\mathcal{P}\subset\mathcal{U}$.
	The matching upper bound rate is constructively established by the performance of the harmonic estimator in Corollary~\ref{thm:risk}.
	When $p\in \mathcal{P}$, its MSE is $\mathcal{O}(1/n)$.
\end{proof}

The minimax rate of $1/n$ is fundamentally anchored by the variance convergence rate on $\mathcal P$.
Because the squared bias does not exceed the $\mathcal{O}(1/n)$ variance rate for all distributions in this class, overall estimation risk is governed by the variance for the fixed class.
This explains why simple, two-point distributions are sufficient to set this lower bound:
the increased statistical complexity of estimation on infinite-support distributions, which manifests as increased squared bias on distributions with thicker tails, remains eclipsed by estimation variance for distributions that converge sufficiently fast (i.e., $p_j \lesssim j^{-2}$).

This setting is distinct from the large-alphabet regime, where support size $s$ may increase with sample size $n$, and worst-case risk is taken over all distributions with support $s$.
As such, risk becomes a function of both sample and support size.
Convergence of $L_2$-minimax risk in this regime requires $s/\log s \lesssim n$ \cite[Thm. 1]{wu2016large} and \cite[Thm. 1]{jiao2015minimax}, so it does not apply to distributions with infinite support.
In contrast, our setting considers worst-case risk over the fixed class $\mathcal P$, allowing only $n$ to increase, and tacitly assuming that $n$ is sufficiently large.
Despite these distinctions, there is agreement on minimax risk where these regimes overlap, namely where $s$ is fixed and $n$ is sufficiently large, i.e. $n \gg s/\log s$:
in the large-alphabet paradigm, minimax risk is $(\log s)^2/n$ as variance dominates, but as $s$ is fixed, the rate depends only on $n$, yielding the $1/n$ minimax risk of our setting under a fixed space.

Whether this minimax rate holds for a broader class of distributions beyond $\mathcal{P}$ remains an open question.
In a regime parameterized by information moments, Cohen et al.~\cite{cohen2023dimension} narrow the $L_1$-minimax risk by providing loose upper and lower bounds.
Importantly, assumptions on information moments are weaker than tail decay, so apply to a broader class, including those based on tail decay.
To extend the sharp $L_2$-minimax rate of $1/n$ beyond $\mathcal P$, one would need to demonstrate that an estimator's squared bias remains dominated by $\mathcal{O}(1/n)$ in heavier-tailed settings (i.e. $p_j \not\lesssim j^{-2}$) while maintaining an $\mathcal{O}(1/n)$ variance rate, a property that the harmonic estimator does not possess.
Alternatively, heavier-tailed distributions may increase statistical complexity and enforce a minimax rate slower than $1/n$.
Determining which of these governs minimax risk for heavier-tailed distributions remains a topic of future research.

\section{Semiparametric Efficiency} \label{sec:clt}

The previous section established that the harmonic estimator's $L_2$ convergence rate, as a function of sample size, cannot be improved: no estimator achieves a rate faster than $1/n$ for distributions in $\mathcal P$.
This characterization suppresses the leading coefficient, which is $\var[\log p(X)]$ for the harmonic estimator when $p_j =o(j^{-2})$.\footnote{To see this, decompose $L_2$ risk into squared bias and variance then apply Theorem~\ref{var} and Proposition~\ref{bias} with $p_j=o(j^{-2})$.}
For efficiency, however, this coefficient is critical because it corresponds to asymptotic variance, as we will show.

In parametric settings, asymptotic efficiency requires an estimator's limiting distribution to be normal with variance achieving the Cramer-Rao lower bound.
Because our setting leaves the underlying distribution unspecified, placing it in the infinite-dimensional space $\{p: p_j =o(j^{-2})\}$, the classical bound does not apply since Fisher information as a matrix is only defined for finite-dimensional models.

The semiparametric efficiency paradigm extends the Cramer-Rao approach to such spaces by applying it to the least favorable, finite-dimensional parametric submodels, yielding a well-defined, generalized efficiency bound despite the infinite-dimensional parameter space.
This bound is given by the variance of the efficient influence function, which is derived from the score of the least favorable submodel, associated with the target parameter $\theta$ \cite[Ch. 3]{bickel1993efficient}, \cite[Ch. 4]{tsiatis2006semiparametric}.
The bound holds for regular estimators, meaning that the limiting distribution of the scaled estimator $\sqrt n (\hat \theta -\theta)$ remains invariant under local $1/\sqrt n$ perturbations of the true distribution.
This condition rules out pathological cases, such as Hodges' superefficient estimator \cite[Ch. 3.1]{tsiatis2006semiparametric}, that beat the bound only at isolated points.

Letting our target parameter be the entropy functional $\theta =H(p)$, the efficient influence function is $\psi_{\text{eff}}(x; p) = -\log p(x) -H(p)$ \cite[Ex. 4.2]{castillo2015bernstein}.\footnote{This reference derives the efficient influence function for continuous distributions but applies here as well using a counting measure, changing the integral into a sum.}
Its variance $\var[\log p(X)]$ sets the semiparametric efficiency bound for this problem, and is finite for $p_j=o(j^{-2})$ (Theorem~\ref{var}).
A regular entropy estimator is semiparametrically efficient if its asymptotic variance achieves this fundamental limit.
Theorem~\ref{clt} establishes that the harmonic estimator achieves this goal: If $p_j = o(j^{-2})$, then as $n \to \infty$,
\begin{equation*}
	\sqrt{n}\brac{\hat{H}(\mathbf{X}_n) - H(X)} \leadsto N(0, \var[\log p(X)])~.
\end{equation*}
Further, Theorem~\ref{clt} establishes that the harmonic estimator is asymptotically linear with influence function equal to the efficient influence function; since this automatically implies regularity, the estimator is regular asymptotically linear (RAL) without a separate argument.

\begin{proof}[Proof Sketch of Theorem~\ref{clt}]
	Let $d(U,V) := \sup_{f\in\text{BL}(\R)} |\E[f(U)]-\E[f(V)]|$ denote the bounded Lipschitz metric between random variables, which metrizes weak convergence, and let $Z\sim N(0,\var[\log p(X)])$.
	Let $H^*(\mathbf{X}_n) := -\frac{1}{n}\sum_{i=1}^n \log p(X^{(i)})$ denote the theoretical oracle estimator, which applies the true log-probabilities to each sample point.
	By the triangle inequality,
	\begin{align} 
		&d\paren{ \sqrt n[\hat H(\mathbf{X}_n)-H(X)],\, Z} \nonumber \\
		&\leq d\paren{ \sqrt n[\hat H(\mathbf{X}_n)-H(X)],\, \sqrt n[H^*(\mathbf{X}_n)-H(X)]}  \label{eq:oracleharmonic} \\
		&\quad +d\paren{ \sqrt n[H^*(\mathbf{X}_n)-H(X)],\, Z}, \label{eq:normoracle}
	\end{align}
	so it suffices to show that both distances \eqref{eq:oracleharmonic} and \eqref{eq:normoracle} vanish.

	\textit{Harmonic-Oracle Equivalence.}
	Because $f\in\text{BL}(\R)$ is Lipschitz, the square of the first term~\eqref{eq:oracleharmonic} is bounded by $n\, \paren{\E\abs{ \hat H(\mathbf{X}_n)-H^*(\mathbf{X}_n)}}^2 \leq n\,\E\brac{(\hat H(\mathbf{X}_n)-H^*(\mathbf{X}_n))^2}$, which decomposes as
	\begin{equation} \label{eq:cltdecomp}
		n\paren{\E\brac{\hat H(\mathbf{X}_n)} -H(X)}^2 + n \var\brac{\hat{H}(\mathbf{X}_n) - H^*(\mathbf{X}_n)} =o\paren{1} +\mathcal{O}\paren{\frac{\log n}{n^{1/2}}} \rightarrow 0,
	\end{equation}
	where the bias term vanishes by Proposition~\ref{bias} and the variance term is bounded by Proposition~\ref{prop:vardiff}, which builds on the variance characterization of Theorem~\ref{var}.

	\textit{Oracle's CLT.}
	A Berry-Esseen-style Wasserstein central limit theorem for i.i.d.\ sums gives
	\begin{align*}
		d\paren{\sqrt n[H^*(\mathbf{X}_n)-H(X)],\, Z} \leq \frac{3\,\E\abs{-\log p(X)-H(X)}^3}{\sqrt n\, \var[\log p(X)]} \rightarrow 0,
	\end{align*}
	where the third absolute moment is finite for $p_j = o(j^{-2})$ by a standard integral comparison.

	Because both distances \eqref{eq:oracleharmonic} and \eqref{eq:normoracle} vanish, $\sqrt n[\hat H(\mathbf{X}_n)-H(X)]$ converges weakly to $Z$. The full technical proof, including the moment-finiteness argument and the metric-to-weak-convergence step, is provided in Appendix~\ref{apdx:clt}.
\end{proof}

The sufficient condition $p_j = o(j^{-2})$ for efficiency is stronger than that for minimax optimality $p_j\lesssim j^{-2}$.
Equation \eqref{eq:cltdecomp} illustrates its necessity:
At the minimax boundary $p_j \asymp j^{-2}$, the estimator achieves the optimal $L_2$-convergence rate but fails the central limit theorem due to non-vanishing asymptotic bias.
That is, the bias is $\mathcal{O}(n^{-1/2})$ at the boundary, making the $n$-scaled squared bias $\mathcal{O}(1)$, which does not vanish.
Only under the stronger $o(j^{-2})$ condition does the bias become $o(n^{-1/2})$, ensuring the $n$-scaled squared bias is $o(1)$.

Previously, Basharin \cite{basharin1959discEntropy} demonstrated the plug-in estimator to be asymptotically normal with variance $\var[\log p(X)]$ on finite-support distributions; whether this extends to infinite support distributions for the plug-in remains an open question (see Section~\ref{sec:plugvar}).
We establish an analogous central limit theorem, with identical variance, for the harmonic estimator when $p_j = o(j^{-2})$.
This places it within the framework of semiparametric efficiency.

Achieving both minimax rate optimality and semiparametric efficiency establishes the harmonic estimator as asymptotically optimal for $p_j=o(j^{-2})$.
Minimax optimality is a global guarantee, ensuring that the estimator's worst-case risk over $\mathcal{P}$ remains bounded by $\mathcal{O}(1/n)$.
Semiparametric efficiency is a local guarantee, ensuring that for any distribution satisfying $p_j =o(j^{-2})$, no regular estimator can achieve a smaller asymptotic variance than $\var[\log p(X)]$ against the hardest local perturbations of $p$.
Thus, no regular estimator can improve either the rate or the leading constant for asymptotic risk within this class.

Theorem~\ref{clt} is a step toward translating this estimator into a practical, minimum-variance, nonparametric inferential tool.
While it provides the limiting distribution and theoretical asymptotic variance, it does not explicitly establish an estimator for the asymptotic variance.
A natural candidate is the sample variance of transformed observations:
\begin{equation*}
	\hat\sigma^2 = \frac{1}{n} \sum_{i=1}^n \brac{ J(n-1) -J\paren{m^{(i)}-1} }^2 -\paren{\hat H(\mathbf{X}_n)}^2~.
\end{equation*}
Demonstrating its consistency would require bounding a fourth-order moment of harmonic-transformed counts, a non-trivial undertaking, but one for which our algebraic framework provides a natural foundation.
We conjecture that identities analogous to Propositions~\ref{prop:mathind}, \ref{prop:scndmnt}, and \ref{prop:multiProd} hold.
Such a result would enable valid, Wald-style inference for entropy.

\section{Comparison with the Plug-In Estimator}\label{sec:plug}

The plug-in estimator is frequently the starting point for entropy estimation and likely the most widely used in practice.
Both estimators are functionally similar as transformed averages of multinomial counts, $\hat H_f(\mathbf{X}_n)= \sum_{i=1}^n  \brac{ f(n) -f\paren{m^{(i)}} }/n $, where $f(m) =J(m-1)$ for the harmonic estimator and $f(m)= \log(m)$ for the plug-in.
Further, the harmonic numbers and the natural logarithm are closely related: $\log(m) < J(m) \leq \log(m)+1$, with $J(m)-\log(m)\rightarrow \gamma$, where $\gamma$ is the Euler-Mascheroni constant.

This section demonstrates that the harmonic estimator is less negatively biased than the plug-in and conjectures that the plug-in estimator's variance also achieves the semiparametric efficiency bound.
Importantly, it argues that the harmonic approach provides a tractable mathematical foundation for new theoretical properties of entropy and entropy-based estimators.

\subsection{Bias Comparison} \label{sec:plugbias}

Exact expressions characterizing the plug-in estimator's bias, such as the Taylor expansion from \eqref{eq:plugtaylor}, are difficult to bound directly due to probabilities raised to negative powers and challenging higher-order moments.
Antos and Kontoyiannis \cite[Thm. 7]{antos2001convergence} circumvent exact calculations via Jensen's inequality.
In contrast, Proposition~\ref{bias} directly calculates the bias of the harmonic estimator before bounding the exact quantity.
Though these approaches differ, they arrive at convergence rates of the same order:
under $p_j\lesssim j^{-1/\alpha}$ for $\alpha\in (0,1)$, both estimators have bias of $\mathcal{O}\paren{n^{-(1-\alpha)}}$.

The Frullani integral representation of the logarithm provides another exact, non-asymptotic expression of the plug-in estimator's bias.
While this representation is similarly challenging to bound directly, it reveals a structural difference between the two estimators.
\begin{prop} \label{prop:plugbias}
	If $H(X)<\infty$, then
	\begin{equation}
		\text{Bias}_{\plug}
		= -\sum_{k=n}^\infty \frac{\E\brac{(1-p_X)^k}}{k} 
		-\sum_{k=1}^\infty |G_k| \paren{\E\brac{\frac{B_{p_X}(k,n)}{p_X^k}} -B(k,n)} \label{eq:plugbias}
	\end{equation}
	where $G_k$ is the $k$th Gregory coefficient, $B$ and $B_p$ are the beta and incomplete beta functions, respectively, and
	\begin{equation}
		\frac{\brac{\sum_j 1-(1-p_j)^n} -1}{2n} \leq \sum_{k=1}^\infty |G_k| \paren{\E\brac{\frac{B_{p_X}(k,n)}{p_X^k}} -B(k,n)}~. \label{eq:plugbiaslower}
	\end{equation}
	Further, if $p_j\lesssim j^{-1/\alpha}$ for $\alpha\in (0,1)$, then $\text{Bias}_{\plug} = \mathcal{O}\paren{n^{-(1-\alpha)}}$ and
	\begin{equation*}
		\frac{\brac{\sum_j 1-(1-p_j)^n} -1}{2n} =\mathcal{O}\paren{\frac{1}{n^{1-\alpha}}}~.
	\end{equation*}
\end{prop}

\begin{proof}[Proof Sketch of Proposition~\ref{prop:plugbias}]
	Following a similar approach to Proposition~\ref{bias}, let $M$ be a binomial random variable of $n-1$ trials with success probability $p$.
	Using the Frullani integral representation $\log(x) = \int_0^\infty \paren{e^{-t}-e^{-xt}}/t\,dt$ for $x>0$, and applying Fubini's theorem to take the expectation inside the integral, we obtain:
	\begin{align*}
		-\E\brac{\log\paren{\frac{M+1}{n}}}
		&= \E\brac{ \int_0^\infty \frac{ e^{-(M+1)t} -e^{-nt} }{t}\,dt } \\
		&= \int_0^\infty \frac{ e^{-t}\paren{1-p+pe^{-t}}^{n-1} -e^{-nt} }{t} \,dt~. \\
		\intertext{Applying the substitution $u=1-e^{-t}$:}
		&= \int_0^1 \frac{ (1-pu)^{n-1} -(1-u)^{n-1} }{-\log(1-u)} \,du~. \\
		\intertext{We Taylor expand $\frac{-u}{\log(1-u)} = 1 - \sum_{k=1}^\infty |G_k| u^k$, where $G_k$ is the $k$-th Gregory coefficient. Distributing the terms yields:}
		&= \int_0^1 \frac{ (1-pu)^{n-1} -(1-u)^{n-1} }{u} \,du \\
		&\quad - \sum_{k=1}^\infty |G_k| \int_0^1 u^{k-1} \brac{(1-pu)^{n-1} -(1-u)^{n-1}} \,du~. \\
		\intertext{The first integral resolves to $-\log p -\sum_{m=n}^\infty \frac{(1-p)^m}{m}$ via the Frullani approach and geometric series. The terms in the summation evaluate to the incomplete beta function $B_p(k,n) = \int_0^p u^{k-1} (1-u)^{n-1} \,du$ (after a change of variables) and the standard beta function $B(k,n)$, resulting in:}
		&= -\log p -\sum_{m=n}^\infty \frac{(1-p)^m}{m} -\sum_{k=1}^\infty |G_k| \paren{ \frac{B_p(k,n)}{p^k} -B(k,n) }~. 
	\end{align*}
	We use this to take the expected value of the plug-in estimator with respect to $p$.
	The lower bound in (\ref{eq:plugbiaslower}) follows from recognizing that each term in the additional bias summation is positive and evaluating the term $k=1$ of the sum.
	The full proof is deferred to Appendix~\ref{apdx:plug}. 
\end{proof}

The proposition shows that the plug-in estimator is universally more biased than the harmonic estimator: its bias in \eqref{eq:plugbias} decomposes as the harmonic estimator's exact bias, $-\sum_{k=n}^\infty \E[(1-p_X)^k]/k$ (\ref{eq:genbias}), minus an additional, strictly positive term.
That additional term is positive because each summand satisfies $|G_k|>0$ and
\begin{equation*}
	\frac{B_p(k,n)}{p^k} -B(k,n) = \int_0^1 u^{k-1}\brac{(1-pu)^{n-1}-(1-u)^{n-1}}\,du > 0,
\end{equation*}
the inequality following from $(1-pu)^{n-1} > (1-u)^{n-1}$ for $p < 1$.

The plug-in estimator's excess bias stems from the gap between $\log(1+1/m)$ and $1/m$.
Recall that Proposition~\ref{prop:mathind} leveraged the harmonic recursion $J(m) -J(m-1) =1/m$ to evaluate $\E[J(M)]$ directly where $M$ is a binomial random variable of $n$ trials.\footnote{Here $M$ counts all draws rather than $n-1$.}
Analogously, evaluating $\E[\log(M+1)]$ directly via mathematical induction, we encounter $\log(m+1)-\log(m) =\log(1+1/m)$, which has no closed polynomial form.
It equals $\sum_{k=1}^\infty (-1)^{k+1}/(km^k)$, forcing a practical evaluation to rely on approximation bounds.
In fact, the sum over $k\geq 2$ corresponds to the excess bias from Proposition~\ref{prop:plugbias}.

To see this, let $S_n = \sum_{m=0}^n \log(m+1)\binom{n}{m} p^m (1-p)^{n-m}$.
Applying Pascal's identity and reindexing,
\begin{align*}
	S_{n+1} 
	&= S_n +\sum_{m=1}^{n+1} \log\paren{ 1+\frac{1}{m} } \binom{n}{m-1} p^m (1-p)^{n+1-m}~. \\
	\intertext{Replacing $\log(1+1/m)$ with its expansion, and applying the absorption identity by factoring $1/m$ from each term:}
	&= S_n +\frac{1}{n+1} \sum_{m=1}^{n+1} \sum_{k=1}^\infty \frac{(-1)^{k+1}}{km^{k-1}} \binom{n+1}{m} p^m (1-p)^{n+1-m}~.
\end{align*}
Summing the $k=1$ term over $m$ gives $[1-(1-p)^{n+1}]/(n+1)$, directly corresponding to the evaluation of $\E[J(M)]$, solved in Proposition~\ref{prop:mathind}.
Thus, following the same induction, grouping the remaining terms corresponding to $k\geq 2$ as $\mathcal{E}_n$, we have
\begin{equation*}
	-\E\brac{ \log\paren{\frac{M+1}{n}} }
	= -\log p -\sum_{m=n+1}^\infty \frac{(1-p)^m}{m} +\mathcal{E}_n~.
\end{equation*}
Comparing this to the exact expression established in Proposition~\ref{prop:plugbias} equates $\mathcal{E}_n$ with the same excess-bias term containing Gregory coefficients.
In effect, $\mathcal{E}_n$ is the penalty for approximating $J(m)$ with $\log(m+1)$ within the combinatorial summation.

Now, consider the excess bias itself from \eqref{eq:plugbias}:
Evaluating the first term, where $k=1$, provides the lower bound shown in \eqref{eq:plugbiaslower}.
Under $p_j\lesssim j^{-1/\alpha}$, the lower bound is of order $\mathcal{O}(n^{-(1-\alpha)})$, which is tight for $\alpha\in (0,1)$ by Proposition~\ref{prop:plugbiastight}, indicating that the excess bias associated with the plug-in is of the same order as the bias itself.
A closer examination of this lower bound reveals a link to bias-corrections techniques for the plug-in estimator.
The quantity $\sum_j [1-(1-p_j)^n]$ is the expected number of distinct values in a sample of size $n$.
Recall that the Miller correction \cite{miller1955entropyBias} adds $(s-1)/(2n)$ to the plug-in estimate when support size $s$ is known and finite. 
A natural adaptation of the Miller correction to unknown support replaces $s$ with the number of distinct symbols observed in the sample, $s_{\text{obs}}$, found in \cite{chao2003diversity, vu2007coverageEntropy}, giving a correction of $(s_{\text{obs}}-1)/(2n)$, which directly mirrors the lower bound.\footnote{The coverage-adjusted estimator of Chao and Shen \cite{chao2003diversity} takes a more sophisticated approach, using a Good-Turing estimate of the coverage to reweight empirical probabilities; Vu, Yu, and Kass \cite{vu2007coverageEntropy} prove its consistency.}
Thus, these bias-correction strategies for the plug-in can be seen as a first-order approximation of the harmonic estimator.

This theoretical framework suggests a potential improvement in finite-sample bias correction for entropy estimation.
Historically, explicit bias-correction approaches, ranging from adjustments based on sample size \cite{miller1955entropyBias} to sophisticated polynomial approximations \cite{jiao2015minimax, wu2016large}, have utilized the plug-in estimator as their foundational baseline.
However, as demonstrated by Proposition~\ref{prop:plugbias}, the harmonic estimator is universally superior in this regard.
Its bias is completely characterized by a simple, tractable summation, and importantly does not include the excess penalty associated with the plug-in estimator.
Further, the expression characterizing the harmonic estimator's bias is a weighted sum of the expected missing mass of samples of size $k\geq n$, a well-studied problem \cite{good1953missingmass, berend2012missing, ben2017concentration}.
Correcting the harmonic estimator's bias could naturally leverage this established literature, potentially positioning the harmonic estimator as a mathematically transparent baseline for future bias-correction methodologies.

\subsection{Variance Comparison} \label{sec:plugvar}

Because the harmonic estimator's variance achieves the semiparametric efficiency bound, as shown in Section~\ref{sec:clt}, its asymptotic variance is the minimum possible for this estimation problem among regular estimators for $p_j =o(j^{-2})$.
The plug-in estimator can, at best, match the harmonic estimator's asymptotic variance, but cannot improve on it for large samples.

Antos and Kontoyiannis~\cite{antos2001convergence} bound the plug-in estimator's variance convergence rate universally at $\mathcal{O}((\log n)^2/n)$.
For distributions satisfying $p_j \lesssim j^{-1/\alpha}$ with $\alpha\in (0,1)$, the rate tightens to $\var\brac{\hat H_{\plug} (\mathbf{X}_n)} =\mathcal{O}(1/n)$ via the Efron-Stein inequality (see Proposition~\ref{thm:varbound}).
As we note in Section~\ref{sec:var}, this concentration-based approach only bounds the variance; it does not yield the exact asymptotic variance necessary to establish asymptotic normality and evaluate efficiency.
However, the functional proximity of the estimators makes it seem likely that the two share the same asymptotic variance, though this remains an open question.

\begin{conj}\label{conj:plugvar}
	If $p_j \lesssim j^{-1/\alpha}$ for $\alpha \in (0,1)$, then
	\begin{equation}\label{eq:plugvar}
		\var\brac{\hat H_{\plug} (\mathbf{X}_n)} 
		= \frac{\var\brac{\log p(X)}}{n} +\mathcal{O}\paren{\frac{\log n}{n^{2-\alpha}}}~.
	\end{equation}
\end{conj}

\begin{proof}[Justification]
Decompose the plug-in estimator into its local variance and covariance components: 
\begin{align*}
	&\var\brac{\hat H_{\plug}(\mathbf{X}_n)}
	= \var\brac{-\frac{1}{n} \sum_{i=1}^n \log\paren{\frac{m^{(i)}}{n}}} \\
	&= \frac{1}{n} \var\brac{\log\paren{\frac{m^{(1)}}{n}}} +\frac{n-1}{n} \cov\brac{ \log\paren{\frac{m^{(1)}}{n}}, \log\paren{\frac{m^{(2)}}{n}}}~.
\end{align*}

Beginning with the variance, we use Proposition~\ref{prop:varcomparison}: $\var[f(X)] \geq \var[g(X)]$ when $f, g$ and $f-g$ are non-decreasing functions for a real-valued random variable $X$.
Then for a binomial random variable, $M$, $\var[J(M+1)] \leq \var[\log(M+1)] \leq \var[J(M)]$, because $J(m) -\log(m+1)$ and $\log(m+1) -J(m+1)$ are both non-decreasing.
Because $n$ is a constant, applying the variance property $\var[X+c] =\var[X]$,
\begin{equation*}
	\var\brac{ J(n) -J\paren{m^{(1)}}}
	\leq \var\brac{\log\paren{\frac{m^{(1)}}{n}}}
	\leq \var\brac{ J(n-1) -J\paren{m^{(1)}-1 } }~.
\end{equation*}
Building on the infrastructure of Theorem~\ref{var}, one can show that $\var\brac{\log(m^{(1)}/n)} = \var\brac{\log p(X)} +\mathcal{O}(1/n^{1-\alpha})$ by demonstrating that these upper and lower bounds are asymptotically equivalent.

Bounding the covariance term is the primary obstacle to a complete proof.
One can express the local plug-in covariance as a local harmonic covariance and cross-difference covariances:
\begin{align*}
	\cov[\log(m^{(1)}), \log(m^{(2)})] 
	&= \cov[J(m^{(1)}), J(m^{(2)})] +2\cov[J(m^{(1)}), \log(m^{(2)})-J(m^{(2)})] \\
	&\quad+\cov[\log(m^{(1)})-J(m^{(1)}), \log(m^{(2)})-J(m^{(2)})]~.
\end{align*}
Lemma~\ref{lem:cov} can be used to bound the local harmonic covariance, but bounding the other terms is non-trivial.
A complete resolution requires calculating these dependencies separately for the cases where $X^{(1)} = X^{(2)}$ and $X^{(1)} \neq X^{(2)}$.
This technical hurdle remains the primary barrier to fully resolving the conjecture.
\renewcommand{\qedsymbol}{}
\end{proof}

Together, Propositions~\ref{prop:plugbias} and \ref{thm:varbound}, bounding the plug-in estimator's bias and variance, imply that the plug-in estimator is also minimax optimal over $\mathcal P$.
That is, its $L_2$ convergence rate is $\mathcal{O}(1/n)$ for $p\in \mathcal P$, matching the lower bound rate established in Proposition~\ref{thm:minimaxlower}.
If Conjecture~\ref{conj:plugvar} holds, combining it with the bias bound of Proposition~\ref{prop:plugbias} may imply that the plug-in estimator is also semiparametrically efficient for $p_j =o(j^{-2})$.
To establish asymptotic normality, it would be sufficient to follow the procedure in the proof of Theorem~\ref{clt}.

Unlike the bias comparison, this section does not analytically demonstrate that the harmonic estimator outperforms the plug-in estimator with respect to variance; on the contrary, we conjecture that their asymptotic variances are identical.
Yet, the penalty for the logarithm in these summations, in comparison to the harmonic number function, seems to be a loss of analytical tractability.
Any direct, induction-based evaluation of the plug-in estimator's variance will necessarily encounter additional complexity, similar to that with the bias calculation, but compounded due to second-order moments.
Other approaches, such as the delta method or $\log$-based inequalities, may not provide the necessary analytical precision for the problem.
These limitations may explain the lack of similar analyses on asymptotic normality and efficiency for the plug-in estimator on unbounded random variables.

Our approach for justifying Conjecture~\ref{conj:plugvar} relies on the machinery of Theorem~\ref{var} and the proximity of the harmonic numbers and natural logarithm.
While demonstrating $\var\brac{\log(m^{(1)}/n)} = \var\brac{\log p(X)} +\mathcal{O}(1/n^{1-\alpha})$ is straightforward, the covariance creates difficulty.
Recall that bounding the total harmonic covariance requires separating the cases where $X^{(1)}=X^{(2)}$ and $X^{(1)}\neq X^{(2)}$ in Lemma~\ref{lem:cov}, neither of which is $o(1/n)$ alone, but combine to $\mathcal{O}(1/n^{2-\alpha})$.
The calculations rely on the cancellation of polynomial approximations of $[\log p(X^{(1)})] [\log p(X^{(2)})]$, and expressing their remainders with sufficient precision to combine and bound them.

Although mirroring the harmonic approach may be feasible, it would require considerable additional bookkeeping to track logarithm-associated error.
Our conjecture leverages the existing harmonic covariance analysis, yielding two covariance quantities involving harmonic/logarithm differences that may be bounded by $\mathcal{O}((\log n)/n^{2-\alpha})$.
Ultimately, the harmonic estimator establishes a theoretical floor for asymptotic variance that the plug-in estimator cannot surpass.
If establishing Conjecture~\ref{conj:plugvar} requires leveraging the harmonic estimator or mirroring its approach, this would itself be informative, further underscoring the plug-in's comparative lack of analytical tractability relative to the harmonic approach developed throughout this paper.
While we cannot rule out an alternative proof strategy, none appears evident despite the extensive literature devoted to the plug-in estimator.
We leave a complete resolution to future work.

\begin{table}[ht]
	\centering
	\caption{Comparison of the harmonic and plug-in estimators.}
	\label{tab:comparison}
	\begin{tabular}{|>{\raggedright\arraybackslash}p{2.7cm}|>{\raggedright\arraybackslash}p{5.7cm}|>{\raggedright\arraybackslash}p{5.7cm}|}
		\hline
		\textbf{Property} & \textbf{Harmonic} & \textbf{Plug-in} \\
		\hline
		Bias & Exact form (Proposition~\ref{bias}) & Exact form = harmonic's bias minus a positive excess term (Proposition~\ref{prop:plugbias}) \\
		\hline
		Bias rate (${p_j\lesssim j^{-1/\alpha}}$) & $\Theta(n^{-(1-\alpha)})$, proven tight (Proposition~\ref{prop:tightbias}) & Same order; the excess term is itself $\Theta(n^{-(1-\alpha)})$, proven tight (Proposition~\ref{prop:plugbiastight}) \\
		\hline
		Variance, no tail assumption & $\mathcal O((\log n)^2/n)$ (Proposition~\ref{thm:varbound}) & $\mathcal O((\log n)^2/n)$ \cite{antos2001convergence} \\
		\hline
		Variance, under $\alpha\in(0,1)$ & $\var[\log p(X)]/n + \mathcal O(\log n/n^{2-\alpha})$ (Theorem~\ref{var}) & $\mathcal O(1/n)$ (Proposition~\ref{thm:varbound}); conjectured to match harmonic (Conjecture~\ref{conj:plugvar}) \\
		\hline
		Minimax Optimal & Proven over $\mathcal P$ (Theorem~\ref{thm:minimax}) & Proven over $\mathcal P$ (Propositions~\ref{prop:plugbias}, \ref{thm:varbound}, \ref{thm:minimaxlower})\\
		\hline
		CLT & Proven over $p_j=o(j^{-2})$ (Theorem~\ref{clt}) & Open, contingent on covariance analysis \\
		\hline
		Semiparametric efficiency & Proven over $p_j=o(j^{-2})$ (Theorem~\ref{clt}) & Open, contingent on Conjecture~\ref{conj:plugvar} \\
		\hline
	\end{tabular}
\end{table}

\section{Experiments}\label{apdx:experiments}

This section illustrates the empirical performance of the harmonic estimator compared to other estimators on simulated data.
While the primary contribution of this work is theory, the following experiments confirm that the theory established earlier in this paper is reflected in finite-sample behavior.
Our aim is validation, not an exhaustive empirical survey.
For a comprehensive comparison of estimators across distributional regimes, see \cite{pinchas2024comparative}.
The simulations here are instead restricted to a small set of estimators chosen for their relevance to the paper's theoretical claims: the plug-in estimator (labeled \emph{Plug-In}), directly compared to the harmonic estimator in Section~\ref{sec:plug}; a Miller-style correction \cite{miller1955entropyBias} (labeled \emph{Miller}), also compared in Section~\ref{sec:plugbias}; and two modern, theoretically-motivated alternatives, the estimator of \cite{jiao2015minimax} (labeled \emph{JVHW}) and the Bayesian estimator of \cite{archer2014bayesian} (labeled \emph{PYM}).
For the Miller estimator, we use the number of distinct observed symbols $s_{\text{obs}}$ rather than the theoretical support size as discussed in Section~\ref{sec:plugbias}, reflecting that $s$ is often unknown.
The JVHW estimator has theoretical guarantees in settings where support size scales with sample size, but is included in all simulations, including infinite support distributions, for comparison.
The Pitman-Yor Mixture (PYM) estimator is developed in a Bayesian paradigm, which assumes that entropy is itself a random variable; the simulations presented here follow a frequentist model assuming that entropy is fixed.

To simulate various settings, datasets are repeatedly generated from four discrete distributions with varying sample sizes.
Sample sizes range from 200 data points to 2000 in increments of 200.
The distributions presented here represent (1) a small, finite-support distribution, (2) a large, finite-support distribution, (3) an unbounded distribution with exponential decay, and (4) an unbounded distribution with power-law decay.
For each distribution and sample size, the simulation generates 100 datasets then applies the estimation methods to each dataset.

The estimation results are displayed as empirical mean squared error (MSE) and violin plots\footnote{Mirrored kernel density estimation plots.} representing the distribution of entropy estimates by sample size.
Violin plots provide additional information on bias and variance separately, which is not available in the empirical MSE plots, but is helpful to corroborate theoretical claims presented earlier.
Thicker regions of the violin plot correspond to a greater density of entropy estimates of the 100 total for that sample size.
The ``--'' markers above and below each violin indicate the maximum and minimum estimates for each violin.
The ``$\times$'' marker within each violin indicates the mean value of the 100 entropy estimates; mean values of each estimator are connected to better visualize change as the sample size increases.
The dashed, horizontal line indicates the true entropy.

\begin{figure}[ht]
	\centering
	\includegraphics[width=0.45\textwidth, page=2]{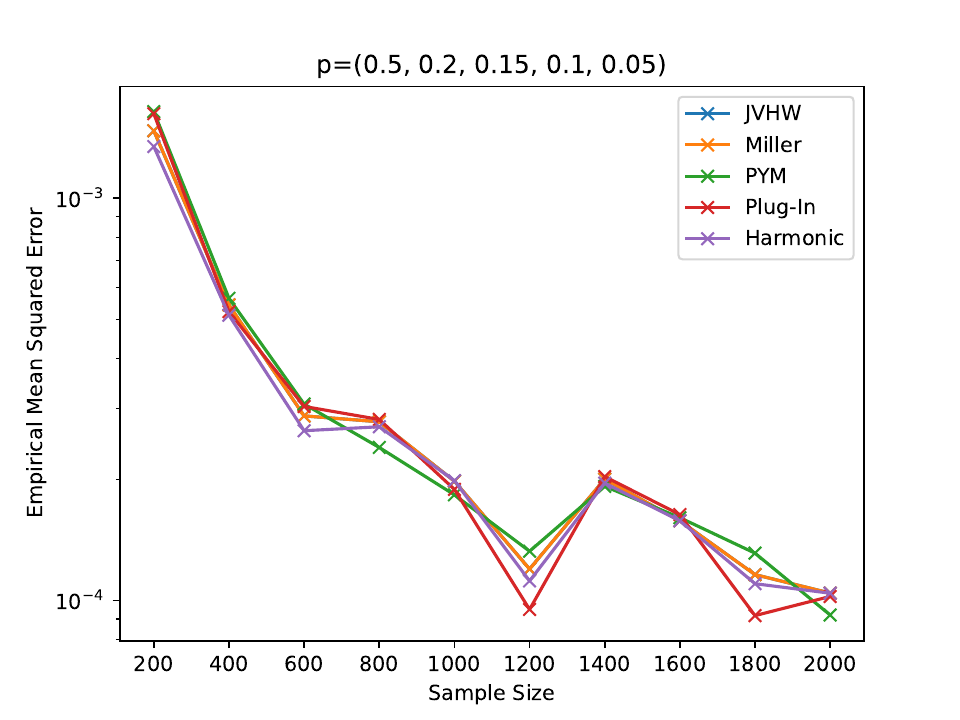}
	\includegraphics[width=0.45\textwidth, page=2]{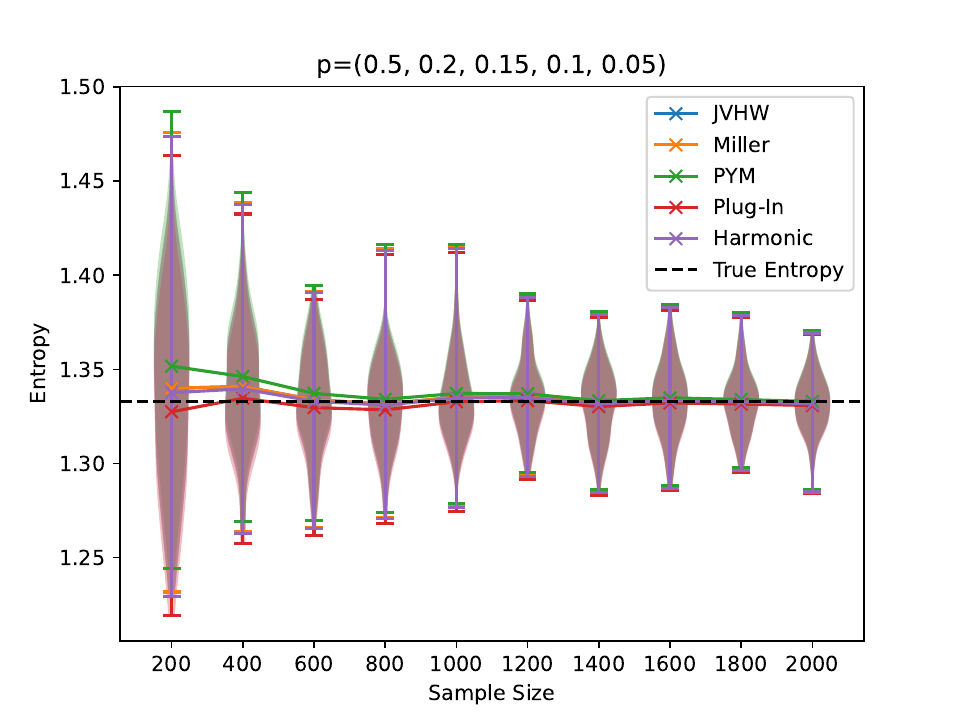}
	\caption{The figure above shows plots for a multinomial distribution with probability vector $(0.5, 0.2, 0.15, 0.1, 0.05)$.
	On the left are average mean squared errors (MSE) by sample size and on the right are corresponding violin plots.}\label{multi}
\end{figure}

Figure~\ref{multi} shows empirical MSE and violin plots for entropy estimates where the underlying distribution has five support points with $p_1=0.5, p_2=0.2, p_3=0.15, p_4=0.1, p_5=0.05$ and entropy $H(X) \approx 1.333074$.
The true MSE convergence rate for the harmonic estimator is $\mathcal O(1/n)$ by applying Corollary~\ref{thm:risk}.
Because the support is small in comparison to each sample size, estimator bias diminishes rapidly for all estimators as shown in the violin plots.
The decrease in MSE is due primarily to variance convergence.
While there are some differences in mean entropy between methods for smaller samples, the large amount of overlap between the different violins for each sample size indicates that all methods have similar performance for this distribution.

\begin{figure}[ht]
	\centering
	\includegraphics[width=0.45\textwidth, page=3]{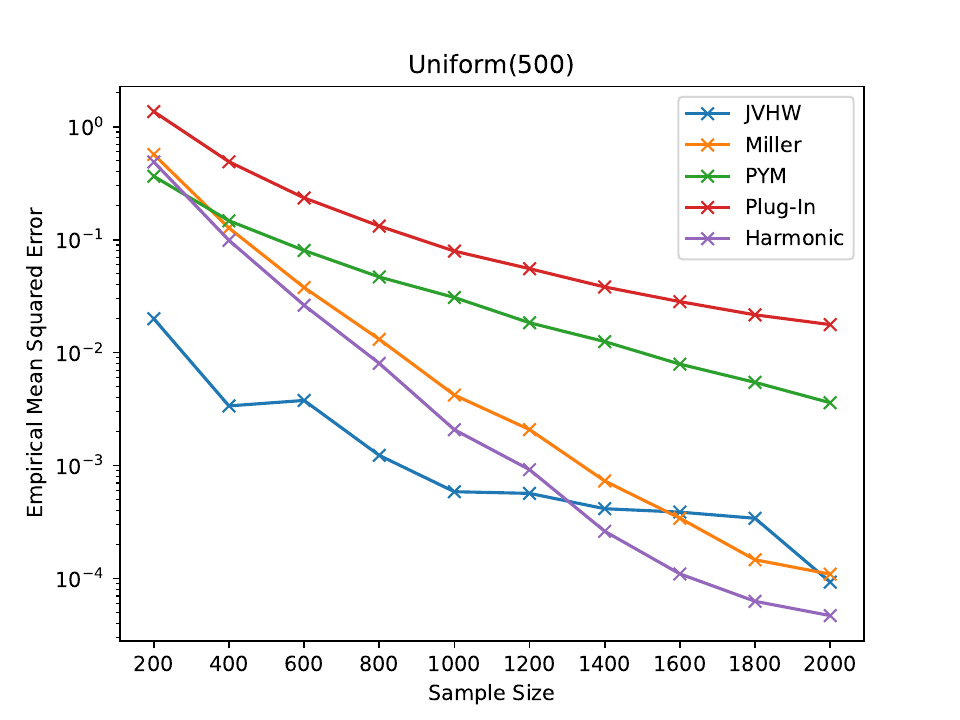}
	\includegraphics[width=0.45\textwidth, page=3]{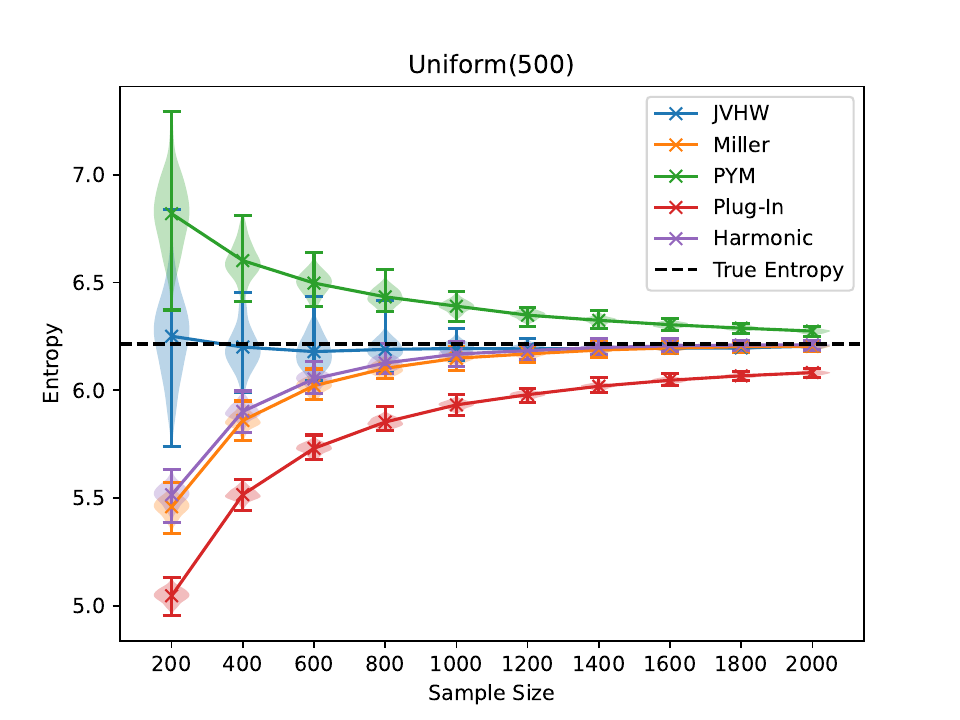}
	\caption{The figure above shows plots for the uniform distribution with $s=500$.
	On the left are average mean squared errors (MSE) by sample size and on the right are corresponding violin plots.}\label{uniform}
\end{figure}

Figure~\ref{uniform} shows simulation plots for uniformly distributed data over a support of 500 points, with a true entropy of $H(X) \approx 6.214608$.
Again, the MSE convergence for the harmonic estimator is $\mathcal O(1/n)$.
This distribution illustrates the challenge of estimating entropy when sample size is smaller than support size.
The JVHW estimator performs best for smaller samples in terms of MSE, likely because of its smaller bias, as shown in the violin plots.
While the harmonic estimator underestimates entropy for smaller samples, it eventually overtakes JVHW in empirical MSE.
The violin plots clearly illustrate the bias correction of the harmonic estimator over the plug-in demonstrated in Proposition~\ref{prop:plugbias}.
The similar sizes of the violins themselves illustrate the similarity in variance of the two estimators, offering support for Conjecture~\ref{conj:plugvar}.
The violin plots also affirm the relationship between the harmonic and Miller-style estimators discussed in Section~\ref{sec:plugbias}: the correction to the plug-in approaches the harmonic estimator, but it remains slightly more biased as it removes only the first-order term of the plug-in's excess bias.

\begin{figure}[ht]
	\centering
	\includegraphics[width=0.45\textwidth, page=1]{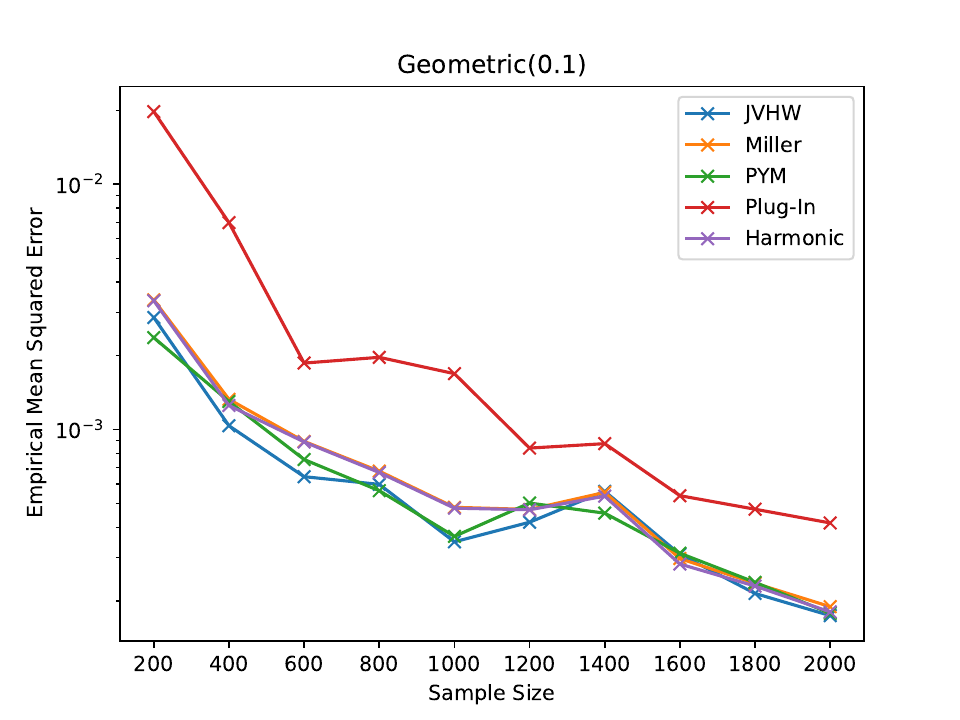}
	\includegraphics[width=0.45\textwidth, page=1]{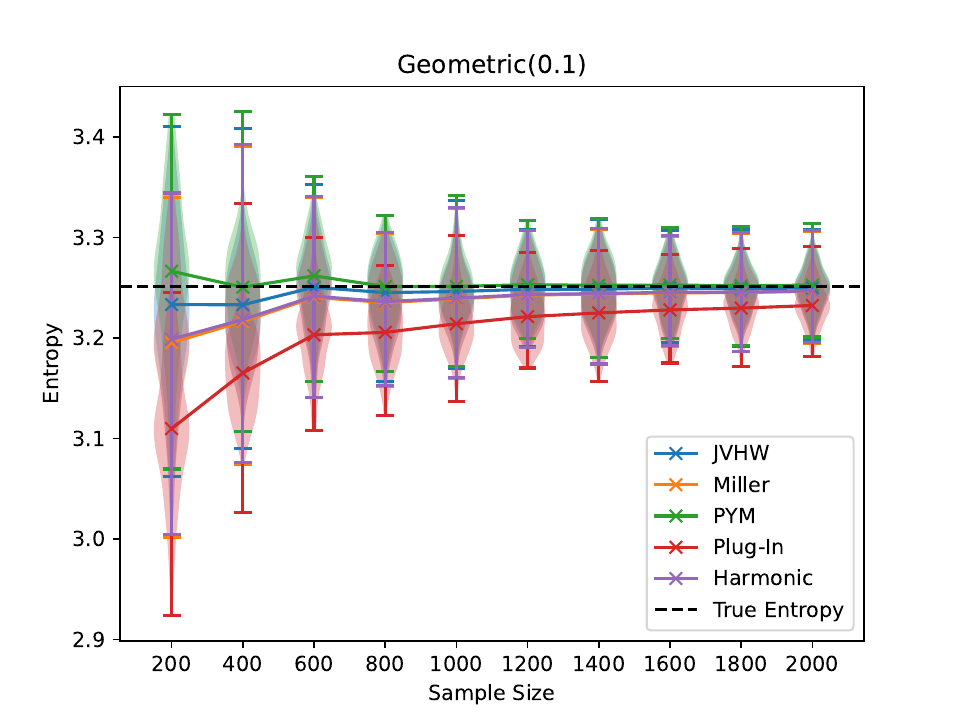}
	\caption{The figure above shows plots for the geometric distribution with parameter $p=0.1$.
	On the left are average mean squared errors (MSE) by sample size and on the right are corresponding violin plots.}\label{geom}
\end{figure}

In Figure~\ref{geom}, the underlying data come from the geometric distribution with parameter $p=0.1$.
Because the geometric distribution decays exponentially, faster than any power law, Corollary~\ref{thm:risk} applies for arbitrarily small $\alpha$, so the mean squared error convergence rate for the harmonic estimator is $\mathcal O(1/n)$.
Its true entropy is $H(X) \approx 3.250830$ (computed via the standard closed-form expression for the geometric distribution).
In this setting, estimators other than the plug-in perform similarly, with a relatively small bias for all sample sizes and shrinking variance as sample size increases.

\begin{figure}[ht]
	\centering
	\includegraphics[width=0.45\textwidth, page=4]{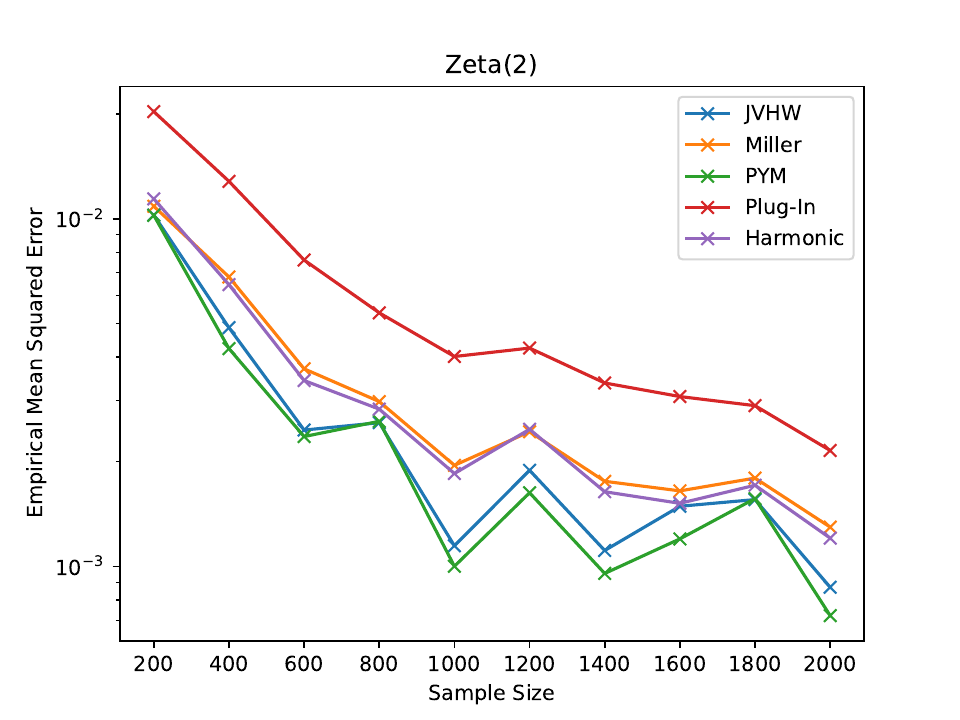}
	\includegraphics[width=0.45\textwidth, page=4]{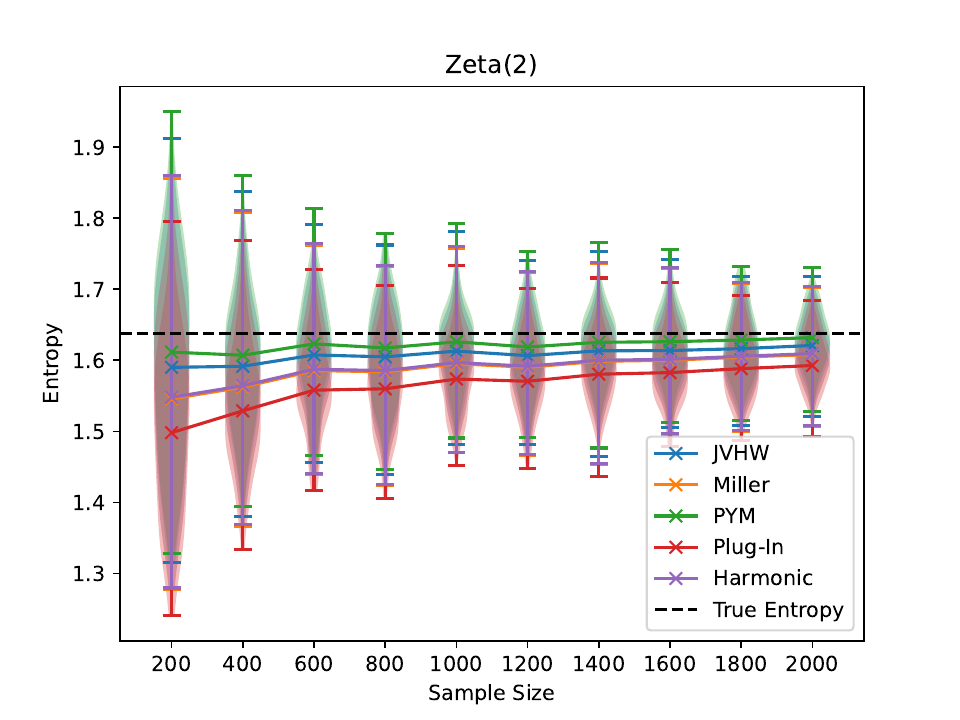}
	\caption{The figure above shows plots for the zeta distribution with parameter $\gamma=2$.
	On the left are average mean squared errors (MSE) by sample size and on the right are corresponding violin plots.}\label{zeta}
\end{figure}

Figure~\ref{zeta} shows the case where the underlying data are generated from a zeta distribution $p_j = j^{-\gamma}/\zeta(\gamma)$ where $\zeta(\gamma) = \sum_{j=1}^\infty j^{-\gamma}$ is the Riemann zeta function, used here as a normalizing constant.
The data are generated with parameter $\gamma=2$, so that $p_j\asymp j^{-2}$ ($\alpha=1/2$).
In this case, the tail of the distribution converges too slowly for the conditions for asymptotic normality (Theorem~\ref{clt}) to be met for the harmonic estimator.
The entropy for a zeta-distributed random variable with parameter $\gamma=2$ is
\begin{equation*}
	H(X) =\frac{1}{\zeta(2)} \sum_{j=1}^\infty j^{-2} \log\brac{j^2 \zeta(2)} \approx 1.637622~.
\end{equation*}
For this distribution, all methods underestimate entropy on average with PYM consistently closest to the true entropy in mean and the plug-in farthest away.
By Proposition~\ref{bias}, the estimator's bias is bounded by $\mathcal{O}(n^{-1/2})$, matching the bound of the standard deviation in this setting; this is also the slowest rate of bias convergence of the examples chosen here.
This theoretical bias is clearly reflected in the violin plots (Figure \ref{zeta}), where the harmonic estimator's mean is visibly centered below the true entropy line, validating our theoretical risk decomposition.

\section{Conclusion and Future Work} \label{sec:conclusion}

This paper introduces the harmonic entropy estimator and shows that it attains the sharp statistical limits of entropy estimation over discrete distributions with sufficiently fast tail decay ($p_j \lesssim j^{-2}$), including both finite- and infinite-support distributions.
For this class, the estimator attains the $\Omega(1/n)$ minimax lower bound.
Under the slightly stronger condition $p_j = o(j^{-2})$, the centered, scaled estimator is asymptotically normal with variance equal to the semiparametric efficiency bound, $\var[\log p(X)]$.
Together, these results show that the estimator achieves the global robustness of an optimal worst-case convergence rate while maintaining the local precision required for asymptotic efficiency.

These results rest on a subtle but important estimator modification: replacing the logarithm with the harmonic number function.
Although these functions are closely related, this substitution makes sharp characterization possible where prior approaches have only established rates \cite{antos2001convergence, jiao2015minimax, wu2016large, cohen2023dimension}.
This contrast is clearest in Corollary~\ref{thm:risk}, which isolates the leading constant of the harmonic estimator's $L_2$ risk;
to our knowledge, no analogous results exist for logarithm-based estimators on infinite-support distributions.
As Section~\ref{sec:plug} illustrates, the logarithm has no exact relationship to the combinatorics of sampling from a discrete distribution; the asymptotically equivalent harmonic number function does.
Specifically, the alignment of the estimator to the combinatorial environment facilitates the precise evaluation of estimator moments (Propositions~\ref{prop:mathind}, \ref{prop:scndmnt}, \ref{prop:multiProd});
the resulting expressions correspond to log-probabilities or their products plus a remainder.
This relationship underlies the sharpness of our results.

Its continuation to higher-order moments holds promise for further development, particularly related to inference.
Theorem~\ref{clt} sets a foundation by establishing the estimator's asymptotic distribution but stops short of providing a consistent estimator for $\var[\log p(X)]$.
A variance estimator built on harmonic-transformed counts, as described in Section~\ref{sec:clt}, would require establishing a bound on a fourth-order estimator moment.
Together these would enable valid, Wald-style confidence intervals for entropy.
A more general demonstration of this relationship with higher-order moments would aid in the development of an exponential concentration inequality for entropy using the harmonic estimator.
Such results could extend to mutual information and other entropy-based measures for nonparametric inference.

Conjecture~\ref{conj:plugvar} foreshadows similarly sharp results for the plug-in estimator.
However, even if this conjecture holds, the harmonic estimator retains practical and theoretical advantages.
Most notably, the plug-in estimator's asymptotic variance cannot improve on that of the harmonic estimator for $p_j =o(j^{-2})$, as the latter already achieves the semiparametric efficiency bound.
Furthermore, the harmonic estimator is structurally less biased in all regimes; the simulations in Section~\ref{apdx:experiments} affirm the theory of Proposition~\ref{prop:plugbias} across various settings and visualize the magnitude of this excess plug-in bias.

Analytically, the harmonic framework may serve as a vehicle for eventually proving the conjecture, as the justification in Section~\ref{sec:plugvar} lays out a promising proof trajectory.
This would further underscore the value of the harmonic approach for clarifying and simplifying complex risk analysis.
Ultimately, the harmonic approach provides a more tractable platform for downstream analysis.
By fully resolving its estimator moments into closed-form expressions, it bypasses the complex combinatorial series, e.g., equations~\eqref{eq:plugtaylor} and~\eqref{eq:plugbias}, that plague exact expectations of logarithm-transformed counts.
More broadly, the harmonic approach holds promise not just for entropy, but as a foundation for the development and sharp characterization of other harmonic-based information-theoretic estimators.

The harmonic estimator also has limitations: it is negatively biased in all settings.
In heavy-tailed settings ($p_j \not\lesssim j^{-2}$), its $L_2$ convergence rate exceeds the minimax lower bound of $\Omega(1/n)$ due to increased bias.
Whether the minimax rate of $1/n$ extends into this regime remains open.
A formal bias-correction procedure for the harmonic estimator remains to be developed but may extend this rate beyond $p_j \lesssim j^{-2}$.
An extension would require demonstrating that the bias-corrected estimator maintains an $L_2$ risk dominated by $\mathcal{O}(1/n)$ on heavier-tailed distributions.
Alternatively, heavier tails may fundamentally add to statistical complexity, slowing the minimax rate.
This scenario would require establishing a lower bound rate greater than $\Omega(1/n)$ with an estimator matching this new rate.

Looking beyond exclusively discrete distributions, the harmonic estimator's functional form reveals a striking parallel between discrete and continuous estimation: the identity linking harmonic numbers to log-probabilities in this paper is a direct analogue of the identity linking the digamma function to log-probabilities of neighbors in the Kozachenko-Leonenko nearest-neighbor estimator \cite{kozachenko1987sample, gao2018demystifying}, a connection hinted at in \cite{mesner2020conditional}.
This shared mathematical foundation points toward a unifying principle in information-theoretic functional estimation for general distributions.

\appendices

\section{Bias Analysis and Technical Lemmas}\label{apdx:bias}

\begin{proof}[Proof of Prop. \ref{bias}]
	If $p_j\lesssim j^{-1/\alpha}$, then $H(X) <\infty$ by a standard integral comparison.
	Starting from \eqref{eq:genbias} (proved in the main body), we establish the bound in \eqref{eq:biasrate} using Proposition~\ref{lem:missingmass} (below) with $m=1$, $\E\brac{(1-p_X)^k} \leq C/k^{1-\alpha}$ for $k\in\N$ where $C>0$ is not a function of $k$.
	Then, because $\sum_{k=n}^\infty f(k) \leq \int_n^\infty f(x)\,dx +f(n)$ for decreasing $f$, we have
	\begin{equation}
		\sum_{k=n}^\infty \frac{\E\brac{ (1-p_X)^k }}{k}
		\leq \sum_{k=n}^\infty \frac{C}{k^{2-\alpha}}
		\leq \int_{n}^\infty \frac{C}{x^{2-\alpha}} \,dx +\frac{C}{n^{2-\alpha}}
		\leq\frac{C_1}{n^{1-\alpha}}~.
	\end{equation}
	Because $1/n^{1-\alpha}$ is the leading term, there is a constant $C_1>0$ that does not depend on $n$ such that the last inequality holds.
	Last, if $p_j =o(j^{-2})$, we use the second, corresponding claim of Proposition \ref{lem:missingmass} then apply the same logic as above, yielding a bound of $o(1/\sqrt n)$.
\end{proof}

\begin{prop}\label{lem:missingmass}
	If $p_j \lesssim j^{-1/\alpha}$ for $\alpha \in (0,1)$, and $m\in\N$, then there is a constant $C=C_{m,\alpha}(p)>0$ such that
	\begin{equation}
		\sum_j p_j^m (1-p_j)^k \leq \frac{C}{ k^{m-\alpha} }
	\end{equation}
	holds uniformly over $k\in\N$ where $C$ does not depend on $k$.
	Similarly, if $p_j =o(j^{-1/\alpha})$, then $\sum_j p_j^m (1-p_j)^k = o\paren{ k^{-(m-\alpha)} }$ as $k\rightarrow \infty$.
\end{prop}

\begin{proof}[Proof of Proposition~\ref{lem:missingmass}]
	Assume $\alpha \in (0,1)$.
	Let $G(x) := \abs{\curly{ j\in\N: p_j>x}}$ be the number of probability mass points of $p$ greater than $x$.
	Because $p_j \lesssim j^{-1/\alpha}$, there is some $J\in\N$ and $M>0$ such that for all $j\geq J$, $p_j\leq Mj^{-1/\alpha}$.
	Because $p_j$ is decreasing, if $x\leq p_J$, then
	\begin{equation}
		G(x) = \abs{\curly{ j\in\N: p_j>x}} \leq \abs{\curly{ j\in\N : Mj^{-1/\alpha} >x}} \leq M^\alpha x^{-\alpha}~.
	\end{equation}
	Let $u\leq p_J$.
	Converting the sum into a Riemann-Stieltjes integral, then applying integration by parts yields
	\begin{align}
		\sum_{j: p_j<u} p_j^m
		&= \int_0^u x^m\,d\paren{-G(x)} = \brac{-x^m G(x)}_0^u +m \int_0^u x^{m-1} G(x)\,dx \\
		&\leq M^\alpha m\int_0^u x^{m-1-\alpha} \,dx 
		= \frac{M^\alpha m u^{m-\alpha} }{m-\alpha}~.
	\end{align}
	The inequality follows because of the bound on $G(x)$, $-u^mG(u)\leq 0$ and $x^mG(x)\rightarrow 0$ as $x\rightarrow 0$.

	Again, converting the summation to a Riemann-Stieltjes integral, and applying integration by parts and the inequality established above,
	\begin{align}
		&\sum_j p_j^m (1-p_j)^k \leq \sum_j p_j^m e^{-kp_j}
		= \int_0^\infty x^m e^{-kx} \,d\paren{-G(x)} \label{line:sumint} \\
		&= \brac{ \paren{\sum_{j:p_j<x} p_j^m} e^{-kx} }_0^\infty +\int_0^\infty \paren{\sum_{j:p_j<x} p_j^m} ke^{-kx} \,dx \label{line:intparts} \\
		&= \int_0^{p_J} \paren{\sum_{j:p_j<x} p_j^m} ke^{-kx} \,dx +\int_{p_J}^\infty \paren{\sum_{j:p_j<x} p_j^m} ke^{-kx} \,dx \label{line:intsplit} \\
		&\leq \frac{M^\alpha mk}{m-\alpha} \int_0^{p_J} x^{m-\alpha} e^{-kx} \,dx +\int_{p_J}^\infty ke^{-kx} \,dx \label{line:sumbound} \\
		&\leq \frac{ M^\alpha m }{(m-\alpha) k^{m-\alpha} } \int_0^\infty x^{m-\alpha} e^{-x} \,dx +e^{-kp_J} \label{line:intsubs} \\
		 &= \frac{ M^\alpha m \Gamma(m-\alpha) }{ k^{m-\alpha} } +e^{-kp_J} \leq \frac{C_{m,\alpha}(p)}{k^{m-\alpha}}~.
	\end{align}
	Line~\ref{line:sumint} uses the inequality $(1-x)^n \leq e^{-nx}$ for $x\in [0,1]$ and $n\geq 0$, then converts the summation to an integral.
	Line~\ref{line:intparts} applies integration by parts where the first term converges to zero then line~\ref{line:intsplit} uses an integral property on the second term. 
	Line~\ref{line:sumbound} bounds the first summation using the inequality above and bounds the second summation by one.
	The first term of line~\ref{line:intsubs} extends the upper limit of integration to $\infty$ then uses substitution to simplify the integral to $\Gamma(m-\alpha+1)$; the second term evaluates its corresponding integral.
	The last equality uses the property $\Gamma(x+1) = x\Gamma(x)$.
	Finally, the $1/k^{m-\alpha}$ term dominates for large $k$; choosing $C_{m,\alpha}(p)$ sufficiently large to hold for small $k$ yields the final inequality.

	If $p_j =o(j^{-1/\alpha})$, then $G(x) =o(x^{-\alpha})$ as $x\rightarrow 0$.
	The remainder of the proof follows the same logic as above, so $\sum_j p_j^m (1-p_j)^k = o\paren{ k^{-(m-\alpha)} }$ as $k\rightarrow \infty$ as well.
\end{proof}

\begin{prop} \label{prop:tightbias}
	The bias upper bound from Proposition \ref{bias} is tight for $\alpha\in (0,1)$.
\end{prop}
\begin{proof}
	Let $p_j = C j^{-1/\alpha}$ be the zeta distribution where $C=1/\zeta(1/\alpha)$ with $\alpha\in (0,1)$.
	We begin by bounding the expectation below.
	Using substitution first with $u=Cx^{-1/\alpha}$, then $u=1-e^{-t}$.
	We bound the resulting integral below using $1-e^{-t}\leq t$, then substitution again.
	\begin{align}
		&\E\brac{(1-p(X))^k} 
		= \sum_{j=1}^\infty C j^{-1/\alpha} \paren{ 1-Cj^{-1/\alpha} }^k
		\asymp \int_1^\infty C x^{-1/\alpha} \paren{ 1-Cx^{-1/\alpha} }^k \,dx \\
		&= \alpha C^\alpha \int_0^C u^{-\alpha} (1-u)^k \,du 
		= \alpha C^\alpha \int_0^{-\log(1-C)} \paren{1-e^{-t}}^{-\alpha} e^{-(k+1)t} \,dt \\
		&\geq \alpha C^\alpha \int_0^{-\log(1-C)} t^{-\alpha} e^{-(k+1)t} \,dt 
		= \frac{\alpha C^\alpha}{(k+1)^{1-\alpha}} \int_0^{-(k+1)\log(1-C)} v^{-\alpha} e^{-v} \,dv \\
		&\geq \frac{C}{(k+1)^{1-\alpha}} ~.
	\end{align}
	The final integral approaches $\Gamma(1-\alpha)$ as $k$ increases.
	We allow all constants to be absorbed by $C$.
	Using the calculation above, we bound the estimator's bias below.
	\begin{align}
		&\sum_{k=n}^\infty \frac{ \E\brac{ (1-p(X))^k } }{k} 
		\geq \sum_{k=n+1}^\infty \frac{C}{k^{2-\alpha}}
		\geq \int_{n+1}^\infty \frac{C}{x^{2-\alpha}} \,dx
		\asymp \frac{1}{n^{1-\alpha}}~.
	\end{align}
	Because this lower bound matches the upper bound from Proposition \ref{bias}, the estimator's bias bound is tight.
\end{proof}

\begin{prop}[Reciprocal Binomial Moment] \label{prop:recipmoment}
	If $U\sim\text{Binomial}(n,p)$, then
	\begin{equation}
		\E\brac{\frac{1}{U+1}} = \frac{1-(1-p)^{n+1}}{(n+1)p}~.
	\end{equation}
\end{prop}
\begin{proof}
	Using $\frac{n+1}{k+1}\binom{n}{k} =\binom{n+1}{k+1}$ and reindexing $j=k+1$,
	\begin{align}
		\E\brac{\frac{1}{U+1}}
		&= \sum_{k=0}^n \frac{1}{k+1} \binom{n}{k} p^k(1-p)^{n-k} \\
		&= \frac{1}{(n+1)p} \sum_{j=1}^{n+1} \binom{n+1}{j} p^j(1-p)^{n+1-j}
		= \frac{1-(1-p)^{n+1}}{(n+1)p}~,
	\end{align}
	the last equality using the binomial theorem on the full sum from $j=0$ to $n+1$ and subtracting the $j=0$ term $(1-p)^{n+1}$.
\end{proof}

\begin{prop}\label{propHarmonicBinomial}
	If $n\in\N$ and $p,q\in\R$, then
	\begin{equation}
		\sum_{k=0}^n J(k) \binom{n}{k} p^kq^{n-k}= \sum_{k=1}^n \frac{ (p+q)^n - q^k (p+q)^{n-k} }{k}~.
	\end{equation}
	Note: This is a generalization of Proposition~\ref{prop:mathind}.
\end{prop}
\begin{proof}
	Let $S_n := \sum_{k=0}^n J(k) \binom{n}{k} p^k q^{n-k}$.
	Using the identity Pascal's identity $\binom{n+1}{k} = \binom{n}{k}+\binom{n}{k-1}$, the absorption identity $r\binom{n}{r} = n\binom{n-1}{r-1}$, and the binomial theorem,
	\begin{align}
		S_{n+1}
	&= \sum_{k=0}^{n+1} J(k) \binom{n+1}{k} p^k q^{n+1-k} \\
	&= q\sum_{k=0}^{n} J(k) \binom{n}{k} p^k q^{n-k} +p\sum_{k=0}^{n} J(k) \binom{n}{k} p^k q^{n-k} +\sum_{k=1}^{n+1} \frac{1}{k} \binom{n}{k-1} p^kq^{n+1-k} \\
	&= (p+q) S_n +\frac{(p+q)^{n+1} -q^{n+1}}{n+1}~.
	\end{align}
	Dividing by $(p+q)^{n+1}$, we have
	\begin{equation}
		\frac{S_{n+1}}{(p+q)^{n+1}}= \frac{S_n}{(p+q)^n} +\frac{1-\paren{\frac{q}{p+q}}^{n+1}}{n+1}~.
	\end{equation}
	Because $S_1=p$, we use induction to conclude that
	\begin{equation}
		S_n = (p+q)^n \sum_{k=1}^n \frac{1-\paren{\frac{q}{p+q}}^k}{k}~.
	\end{equation}
\end{proof}

\begin{prop}\label{propReciprocalBinomial}
	If $n\in\N$ and $p,q\in\R$, then
	\begin{equation}
		\sum_{k=1}^n \frac{1}{k} \binom{n}{k} p^kq^{n-k} = \sum_{k=1}^n \frac{(p+q)^k q^{n-k} -q^n}{k}~.
	\end{equation}
\end{prop}

\begin{proof}
	Let $S_n :=\sum_{k=1}^n \frac{1}{k} \binom{n}{k} p^kq^{n-k}$.
	Using the identity Pascal's identity $\binom{n+1}{k} = \binom{n}{k}+\binom{n}{k-1}$, the absorption identity $r\binom{n}{r} = n\binom{n-1}{r-1}$, and the binomial theorem,
	\begin{align}
		S_{n+1}
	&=\sum_{k=1}^{n+1} \frac{1}{k} \binom{n+1}{k} p^kq^{n+1-k} \\
	&=\sum_{k=1}^{n+1} \frac{1}{k} \binom{n}{k} p^kq^{n+1-k} +\sum_{k=1}^{n+1} \frac{1}{k} \binom{n}{k-1} p^kq^{n+1-k} \\
	&=\sum_{k=1}^{n} \frac{1}{k} \binom{n}{k} p^kq^{n+1-k} +\frac{1}{n+1}\sum_{k=1}^{n+1} \binom{n+1}{k} p^kq^{n+1-k} \\
	&= qS_n +\frac{(p+q)^{n+1}-q^{n+1}}{n+1}~.
	\end{align}
	Dividing by $q^{n+1}$, we have
	\begin{align}
		\frac{S_{n+1}}{q^{n+1}}
	&= \frac{S_n}{q^n} +\frac{(p+q)^{n+1} q^{-(n+1)} -1 }{n+1}~.
	\end{align}
	Because $S_1=p$, using induction, we conclude that
	\begin{align}
		S_n = \sum_{k=1}^n \frac{(p+q)^k q^{n-k} -q^n}{k}~.
	\end{align}
\end{proof}

\begin{prop} \label{prop:sumtoint}
	If $p,q\in (0,1)$ and $p+q<1$, then
	\begin{align}
		&\sum_{m=1}^\infty \sum_{k=m+1}^\infty \frac{ (1-p)^m \paren{1-\frac{q}{1-p}}^k + (1-q)^m \paren{1-\frac{p}{1-q}}^k }{mk} \\
		&= \log(p) \log(q) -\log(q)\log(1-q) -\log(p)\log(1-p) \\
		&\quad +\zeta(2) -\li2(1-p) -\li2(1-q) 
	\end{align}
	where $\zeta$ is the zeta function and $\li2$ is the dilogarithm.
\end{prop}
\begin{proof}
	Because of symmetry in $p$ and $q$, we will work through one summand and apply the symmetric results to the other.
	Begin by representing the left summand expression as an integral:
	\begin{align}
		&\sum_{m=1}^\infty \sum_{k=m+1}^\infty \frac{ (1-p)^m \paren{1-\frac{q}{1-p}}^k }{mk} \\
		&= -\sum_{m=1}^\infty \sum_{k=m+1}^\infty \int_{1-q}^p \brac{  \frac{q (1-z)^{m-2} \paren{1-\frac{q}{1-z}}^{k-1} }{m} +\frac{ (1-z)^{m-1} \paren{1-\frac{q}{1-z}}^k }{k} } \,dz \\
		&= -\int_{1-q}^p \sum_{m=1}^\infty \brac{  \frac{ (1-z)^{m-1} \paren{1-\frac{q}{1-z}}^m }{m} +\sum_{k=m+1}^\infty \frac{ (1-z)^{m-1} \paren{1-\frac{q}{1-z}}^k }{k} } \,dz \label{line:geomSeries} \\
		&= -\int_{1-q}^p \sum_{m=1}^\infty \sum_{k=m}^\infty \frac{ (1-z)^{m-1} \paren{1-\frac{q}{1-z}}^k }{k}  \,dz \\
		&= \int_{1-q}^p \sum_{m=1}^\infty \sum_{k=m}^\infty \int_{1-z}^q  (1-z)^{m-2} \paren{1-\frac{y}{1-z}}^{k-1} \,dy \,dz \label{line:intrep2} \\
		&= \int_{1-q}^p \int_{1-z}^q  \sum_{m=1}^\infty \frac{ (1-z)^{m-1} \paren{1-\frac{y}{1-z}}^{m-1} }{y} \,dy \,dz \label{line:geomSeries2} \\
		&= \int_{1-q}^p \int_{1-z}^q  \frac{ 1 }{y(y+z)} \,dy \,dz \label{line:geomSeries3}~.
	\end{align}
	Line~\ref{line:geomSeries} uses Fubini's theorem (and absolute convergence for $p,q\in (0,1)$), then distributes the inner summation to each of the terms within the brackets and evaluates the sum on the left side term using the geometric series.
	Arrive at the following line by recognizing that the left expression is the summand value corresponding to $k=m$ in the inner summation.
	Line~\ref{line:intrep2} writes the summand from the prior line as a integral.
	Line~\ref{line:geomSeries2} uses Fubini's theorem, then the geometric series.
	Line~\ref{line:geomSeries3} uses the exponent power of a product rule then the geometric series identity.

	Repeat the process above for the other symmetric expression and sum both expressions.
	We switch the order of integration for the second integral on the line below, then sum the integrands and evaluate:
	\begin{align}
		& \int_{1-q}^p \int_{1-z}^q  \frac{ 1}{y(y+z)} \,dy \,dz +\int_{1-p}^q \int_{1-y}^p  \frac{ 1 }{z(y+z)} \,dz \,dy \\ 
		&= \int_{1-q}^p \int_{1-z}^q  \frac{ 1}{y(y+z)} \,dy \,dz +\int_{1-q}^p \int_{1-z}^q \frac{ 1}{z(y+z)} \,dy \,dz \\
		&= \int_{1-q}^p \int_{1-z}^q  \frac{1}{yz} \,dy \,dz = \int_{1-q}^p \frac{ \log(q) -\log(1-z) }{z} \,dz \\
		&= \log(p) \log(q) -\log(q)\log(1-q) +\li2(p) -\li2(1-q) \\
		&= \log(p) \log(q) -\log(q)\log(1-q) -\log(p)\log(1-p) \\
		&\quad +\zeta(2) -\li2(1-p) -\li2(1-q)~.
	\end{align}
	The last line uses the dilogarithm reflection identity:
	\begin{equation}
		\li2(z) +\li2(1-z)= \zeta(2) -\log(z) \log(1-z) 
	\end{equation}
	where $\li2(z) := \sum_{m=1}^\infty \frac{z^m}{m^2}$ is the dilogarithm function and $\zeta(z) := \sum_{m=1}^\infty m^{-z}$ is the zeta function.
\end{proof}

\section{Variance Proofs}\label{apdx:var}

\begin{proof}[Proof of Theorem~\ref{var}]
	First, $p_j \lesssim j^{-1/\alpha}$ for $\alpha \in (0,1)$ implies $\var[\log p(X)]< \infty$.
	To see this, $p_j (\log p_j)^2 =\mathcal{O} \paren{j^{-1/\alpha} (\log j)^2} =\mathcal{O}(1/j^\gamma)$ for $\gamma\in (1, 1/\alpha)$ as the logarithmic growth is dominated by $j^{1/\alpha-\gamma}$.
	Because $\gamma>1$, the sum over $j$ is finite.

	Moving to estimator variance, because the sample is i.i.d.\, one can decompose the estimator variance as
	\begin{align}
		&\var\brac{ \hat{H}(\mathbf X_n) } =\var\brac{ \frac{1}{n} \sum_{i=1}^n \brac{ J(n-1) -J\paren{m^{(i)} -1} } } \\
		&= \frac{1}{n} \var\brac{ J(n-1) -J\paren{m^{(1)}-1} } \\
		&\quad +\frac{n-1}{n} \cov\brac{ J(n-1)-J\paren{m^{(1)}-1}, J(n-1)-J\paren{m^{(2)}-1} }~. \label{line:vardecomp}
	\end{align}

	Beginning with the local variance term, let $X=X^{(1)}$ and $m_X$ be $\text{Binomial}(n-1, p(X))$.
	Then $m^{(1)}$ is equal to $m_X+1$ in distribution, so that
	\begin{align}
		&\var\brac{ J(n-1)-J\paren{m^{(1)}-1} }= \var\brac{ J(n-1)-J(m_X) } \\
		&= \E\brac{ (J(n-1)-J(m_X))^2 } -\paren{ \E\brac{J(n-1)-J(m_X)} }^2 \label{line:locvardecomp} \\
		&= \E\brac{ \paren{\log p(X)}^2 } -\sum_{m=n}^\infty \sum_{k=1}^{m-1} \frac{ \E\brac{(1-p_X)^m} }{k(m-k)} +\sum_{m=1}^{n-1} \sum_{k=n-m}^{n-1} \frac{ \E\brac{(1-p_X)^m}}{mk} \label{line:cauchyProd} \\
		&\quad -\paren{\E[\log p(X)]}^2 -2\E\brac{\log p(X)}\sum_{m=n}^\infty \frac{ \E\brac{(1-p_X)^m}  }{m} -\paren{ \sum_{m=n}^\infty \frac{ \E\brac{(1-p_X)^m} }{m} }^2 \label{line:logsq} \\
		&= \var[\log p(X)] +\mathcal O\paren{\frac{\log n}{n^{1-\alpha}} } ~. \label{line:varRemainder}
	\end{align}
	The equality on line~\ref{line:cauchyProd} follows from Proposition~\ref{prop:scndmnt}, plugging in $n-1$ for $n$, on the first expected value of line~\ref{line:locvardecomp} and Proposition~\ref{prop:mathind}, also plugging in $n-1$ for $n$, on the second expected value. 
	Apply Proposition \ref{prop:varremainders} to bound both remainder terms of line \ref{line:cauchyProd}.
	Both of the terms of line~\ref{line:logsq} are related to the estimator's bias, which is bounded in Proposition~\ref{bias}.

	Moving to the covariance term of line~\ref{line:vardecomp}, let $X^{(2)} =Y$ and $(m_X, m_Y, n-2-m_X-m_Y)$ be a multinomial random vector of $n-2$ trials with probabilities $(p(X), p(Y), 1-p(X)-p(Y))$.
	Given $X\neq Y$, then $m^{(1)} =m_X+1$ and $m^{(2)}=m_Y+1$.
	But, given $X =Y$, then $m^{(1)} =m_X +2$.
	With the law of total covariance and the independence of $X$ and $Y$, using $J(m_X+1) = J(m_X) +1/(m_X+1)$ when $X=Y$, we have
	\begin{align}
		&\cov\brac{ J(n-1)-J\paren{m^{(1)}-1}, J(n-1)-J\paren{m^{(2)}-1} } \\
		&= \cov\brac{ J(n-2) -J(m_X), J(n-2) -J(m_Y) } \\
		&\quad +\E_X \brac{ p_X \paren{ \var\brac{\frac{1}{m_X+1}} -2\cov\brac{ J(n-2) -J(m_X), \frac{1}{m_X+1}} } } 
		= \mathcal O\paren{\frac{\log n}{n^{2-\alpha}} } ~.
	\end{align}
	The last bound follows from Lemma~\ref{lem:cov} and Proposition~\ref{prop:addterms}.
\end{proof}

\begin{prop}\label{prop:scndmnt}
	Let $M$ be a binomial random variable of $n$ trials each with probability $p$.
	Then
	\begin{align}
		\E \brac{ J(n)-J(M) }^2  
		&= \sum_{m=2}^n \sum_{k=1}^{m-1} \frac{ (1-p)^m }{k(m-k)} +\sum_{m=1}^n \sum_{k=n-m+1}^n \frac{ (1-p)^m }{mk} \\
		&= (\log p)^2 -\sum_{m=n+1}^\infty \sum_{k=1}^{m-1} \frac{(1-p)^m}{k(m-k)} +\sum_{m=1}^n \sum_{k=n-m+1}^n \frac{ (1-p)^m }{mk} ~.
	\end{align}
\end{prop}

\begin{proof}
	We use induction, let
	\begin{equation}
		S_n := \sum_{m=0}^n \brac{ J(n)-J(m) }^2 \binom{n}{m} p^m (1-p)^{n-m}~.
	\end{equation}
	Using the identities, $J(n+1) = J(n)+\frac{1}{n+1}$ and $\binom{n+1}{m} = \binom{n}{m} +\binom{n}{m-1}$, summation reindexing,
	\begin{align}
	&S_{n+1} = \sum_{m=0}^{n+1} \brac{J(n+1) -J(m)}^2 \binom{n+1}{m} p^m(1-p)^{n+1-m} \\
	&= \sum_{m=0}^{n+1} \paren{\brac{J(n) -J(m)}^2 +\frac{1}{(n+1)^2} +\frac{2 [J(n)-J(m)] }{n+1} } \binom{n+1}{m} p^m(1-p)^{n+1-m} \\
	&= \sum_{m=0}^{n} \brac{J(n) -J(m)}^2 \binom{n}{m} p^m(1-p)^{n+1-m} \\
	&\quad +\sum_{m=0}^{n} \brac{J(n) -J(m+1)}^2 \binom{n}{m} p^{m+1}(1-p)^{n-m} \\
	&\quad +\frac{1}{(n+1)^2} +\sum_{m=0}^{n+1} \frac{2 [J(n)-J(m)] }{n+1}  \binom{n+1}{m} p^m(1-p)^{n+1-m} \\
	&= S_n +\sum_{m=0}^n \paren{ \frac{1}{(m+1)^2} -\frac{2[ J(n)-J(m) ]}{m+1} } \binom{n}{m} p^{m+1}(1-p)^{n-m} \\
	&\quad +\frac{1}{(n+1)^2} +\sum_{m=0}^{n+1} \frac{2 [J(n)-J(m)] }{n+1}  \binom{n+1}{m} p^m(1-p)^{n+1-m} \\
	&= S_n +\frac{1}{n+1}\sum_{m=1}^{n+1} \paren{ \frac{1}{m} -2[ J(n)-J(m-1) ] } \binom{n+1}{m} p^{m}(1-p)^{n+1-m} \\
	&\quad +\frac{1}{(n+1)^2} +\frac{1}{n+1} \sum_{m=1}^{n+1} 2 [J(n)-J(m)]  \binom{n+1}{m} p^m(1-p)^{n+1-m} \\
	&\quad +\frac{ 2J(n) (1-p)^{n+1} }{n+1} \\
	&= S_n +\frac{1}{(n+1)^2} -\frac{1}{n+1}\sum_{m=1}^{n+1} \frac{1}{m} \binom{n+1}{m} p^{m}(1-p)^{n+1-m} +\frac{ 2J(n) (1-p)^{n+1}}{n+1}  \\
	 &= S_n -\frac{(1-p)^{n+1}}{n+1} \sum_{k=1}^{n+1} \frac{ (1-p)^{-k} -1}{k} +\frac{ 2J(n+1) (1-p)^{n+1}}{n+1} +\frac{1- 2(1-p)^{n+1}}{(n+1)^2} \label{trickyrecip} \\
	 &= S_n +\frac{ 3J(n+1) (1-p)^{n+1}}{n+1} -\frac{1}{n+1} \sum_{k=1}^{n+1} \frac{ (1-p)^{n+1-k} }{k} +\frac{1- 2(1-p)^{n+1}}{(n+1)^2}
\end{align}
	using Proposition~\ref{propReciprocalBinomial} on line~\ref{trickyrecip}.
	Together with $S_1 = 1-p$, then
	\begin{align}
		S_n &=\sum_{m=1}^n \sum_{k=1}^m \frac{ 3(1-p)^m -(1-p)^{m-k} }{m k} +\sum_{m=1}^n \frac{1-2(1-p)^m}{m^2} \\
		    &= \sum_{m=1}^n \sum_{k=1}^m \frac{2(1-p)^m}{mk} -\sum_{m=1}^n \frac{2(1-p)^m}{m^2} \label{line:ind2} \\
		    &\quad +\sum_{m=1}^n \sum_{k=1}^m \frac{(1-p)^m -(1-p)^{m-k}}{mk} +\sum_{m=1}^n \frac{1}{m^2} \label{line:varrmdr}~.
	\end{align}

	Working with these lines separately, on line~\ref{line:ind2}, re-indexing, canceling, and using the fact that $\frac{1}{mk} +\frac{1}{m(m-k)} = \frac{1}{k(m-k)}$, we have
	\begin{align}
		&\sum_{m=1}^n \sum_{k=1}^m \frac{2(1-p)^m}{mk} -\sum_{m=1}^n \frac{2(1-p)^m}{m^2} 
		= \sum_{m=2}^n \sum_{k=1}^{m-1} \frac{(1-p)^m}{k(m-k)} \\
		&= (\log p)^2 -\sum_{m=n+1}^\infty \sum_{k=1}^{m-1} \frac{(1-p)^m}{k(m-k)}~.
	\end{align}
	The last line follows by recognizing that the resulting sum on the prior line are the first terms of the Cauchy product for the squared Taylor expansion of $\log p$; the subtracted double sum corresponds to those added to complete the Cauchy product.

	On line~\ref{line:varrmdr}, with re-indexing and cancelling
	\begin{align}
		\sum_{m=1}^n \sum_{k=1}^m \frac{(1-p)^m -(1-p)^{m-k}}{mk} +\sum_{m=1}^n \frac{1}{m^2} 
		 &=\sum_{m=1}^{n} \sum_{k=n-m+1}^n \frac{ (1-p)^m }{mk}~.
	\end{align}
\end{proof}

\begin{prop} \label{prop:varremainders}
	If $p_j \lesssim j^{-1/\alpha}$ for $\alpha\in (0,1)$, then there is a $C_1>0$ that does not depend on $n$ such that
	\begin{equation}
		\sum_{m=1}^n \sum_{k=n-m+1}^n \frac{ \E\brac{(1-p_X)^m} }{mk} \leq \frac{C_1}{n^{1-\alpha}} ~.
	\end{equation}
	Moreover, if $\ell\in\N$, then there are $C_2, C_3>0$ that do not depend on $n$ such that
	\begin{equation}
		\sum_{m=1}^n \sum_{k=n-m+1}^n \frac{ \sum_i p_i^\ell (1-p_i)^{m+k} }{mk} \leq \frac{C_2}{n^{\ell-\alpha}}
	\end{equation}
	and 
	\begin{equation}
		\sum_{m=n+1}^\infty \sum_{k=1}^{m-1} \frac{ \sum_i p_i^\ell (1-p_i)^m  }{k(m-k)} 
		\leq \frac{ C_3 \log n }{n^{\ell-\alpha}}~.
	\end{equation}
\end{prop}

\begin{proof}
	Applying Proposition~\ref{lem:missingmass},
	\begin{align}
		&\sum_{m=1}^n \sum_{k=n-m+1}^n \frac{ \E\brac{(1-p_X)^m} }{mk} 
		\leq \sum_{m=1}^n \frac{C}{m^{2-\alpha} } \sum_{k=n-m+1}^n \frac{1}{k} \\
		&\leq \sum_{m=1}^n \frac{C}{m^{2-\alpha} } \brac{ \int_{n-m+1}^n \frac{1}{x} \,dx +\frac{1}{n-m+1} }~. \label{line:intboundharm}
	\end{align}
	The last line uses the bound $\sum_{m=a}^b f(m) \leq \int_a^b f(x) \,dx +f(a)$ for decreasing $f$ and $a<b$; moving forward we refer to this bound as the summation bound.
	Distributing the sum, we work with each separately.
	First, we evaluate the integral, then apply a Taylor expansion:
	\begin{align}
		&\sum_{m=1}^n \frac{C}{m^{2-\alpha} } \int_{n-m+1}^n \frac{1}{x} \,dx
		= \sum_{m=1}^n \frac{C}{m^{2-\alpha} } \log\paren{\frac{n}{n-m+1}} \\
		&= \sum_{m=1}^n \frac{C}{m^{2-\alpha} } \sum_{k=1}^\infty \frac{1}{k} \paren{\frac{m-1}{n}}^k
		\leq \sum_{k=1}^\infty \frac{C}{k n^k} \sum_{m=1}^n m^{k-2+\alpha} \\
		&\leq \sum_{k=1}^\infty \frac{C}{k n^k} \brac{ \int_1^n x^{k-2+\alpha} \,dx +1}
		= \sum_{k=1}^\infty \frac{C\paren{n^{k-1+\alpha} -1} }{k(k-1+\alpha) n^k} -C\log\paren{1-\frac{1}{n}} \label{line:firstsep} \\
		&\leq Cn^{\alpha-1} \sum_{k=1}^\infty \frac{1}{k(k-1+\alpha)} +\frac{C}{n-1}
		\leq \frac{Cn^{\alpha-1}}{\alpha} \sum_{k=1}^\infty \frac{1}{k^2} +\frac{C}{n-1} \leq \frac{C_1}{n^{1-\alpha}} ~. \label{line:fambound}
	\end{align}
	The equality from line~\ref{line:firstsep} uses $\int_1^n x^{k-2+\alpha}\,dx=(n^{k-1+\alpha}-1)/(k-1+\alpha)$, which holds uniformly for every $k\geq1$ since $k-1+\alpha\geq\alpha>0$; the remaining $+1$ term sums exactly to $-C\log(1-1/n)$.
	The first inequality of line~\ref{line:fambound} drops the $-1$ in the numerator (valid since $n^{k-1+\alpha}\geq1$) and bounds $-\log(1-1/n) \leq 1/(n-1)$.
	The second inequality of line~\ref{line:fambound} uses $k-1+\alpha \geq \alpha k$, which holds because $(k-1)(1-\alpha)\geq0$ for every integer $k\geq1$; the resulting sum is bounded by $\zeta(2)$, and $1/(n-1) \leq 2/n^{1-\alpha}$ for $n\geq2$, so there is a constant $C_1$ that does not depend on $n$ that holds uniformly for all $n$.

	Moving to the second sum from line~\ref{line:intboundharm}, we apply the summation bound,
	\begin{equation}
		\sum_{m=1}^n \frac{C}{m^{2-\alpha} (n-m+1) }
		\leq \int_1^n \frac{C}{x (n-x+1) } \,dx +\frac{C}{n} = \frac{ 2C\log n }{n+1} +\frac{C}{n} ~.
	\end{equation}
	Taken together, all terms are dominated by the final bound of line~\ref{line:fambound}, completing the first claim.

	Moving to the next claim, we re-index, use partial fraction decomposition and Proposition~\ref{lem:missingmass}:
	\begin{align}
		&\sum_{m=1}^n \sum_{k=n-m+1}^n \frac{ \sum_i p_i^\ell (1-p_i)^{m+k} }{mk} 
		= \sum_{m=n+1}^{2n} \sum_{k=m-n}^n \frac{  \sum_i p_i^\ell (1-p_i)^m }{ k(m-k) } \\
		&\leq \sum_{m=n+1}^{2n} \frac{C}{m^{\ell+1-\alpha}} \sum_{k=m-n}^n \brac{ \frac{1}{k} +\frac{1}{m-k} } 
		= \sum_{m=n+1}^{2n} \frac{2C}{m^{\ell+1-\alpha}} \sum_{k=m-n}^n \frac{1}{k} \\
		&\leq \sum_{m=n+1}^{2n} \frac{2C}{m^{\ell+1-\alpha}} \brac{ \log\paren{\frac{n}{m-n}} +\frac{1}{m-n} } \label{line:harmonicIneq} \\
		&\leq \int_{n}^{2n} \frac{2C}{x^{\ell+1-\alpha}} \log\paren{\frac{n}{x-n}} \,dx +\int_{n+1}^{2n} \frac{2C}{x^{\ell+1-\alpha}} \,dx \\
		&= \frac{2C}{n^{\ell-\alpha}} \int_0^1 \frac{ -\log(1-u) }{ (2-u)^{\ell+1-\alpha} } \,du +\frac{2C}{(\ell-\alpha)} \brac{ \frac{1}{(n+1)^{\ell-\alpha}} -\frac{1}{ (2n)^{\ell-\alpha}} } \label{line:intSubs}
		\leq \frac{C_2}{n^{\ell-\alpha}}~.
	\end{align}
	The inequality from line~\ref{line:harmonicIneq} follows from the summation bound $\sum_{k=m-n}^n \frac{1}{k} \leq \int_{m-n}^n \frac{dx}{x} +\frac{1}{m-n} = \log(n/(m-n)) +1/(m-n)$.
	The equality from line~\ref{line:intSubs} follows from integral substitution with $u= 2-x/n$; the integral is finite for $\ell-\alpha >0$.
	Because all terms are dominated $1/n^{\ell-\alpha}$, there is a constant $C_2$ that does not depend on $n$ such that the final inequality holds for all $n$.

	Finally, on the last claim, we use partial fractions, the Proposition~\ref{lem:missingmass}, then integral bounds,
	\begin{align}
		&\sum_{m=n+1}^\infty \sum_{k=1}^{m-1} \frac{ \sum_i p_i^\ell (1-p_i)^m  }{k(m-k)} 
		= \sum_{m=n+1}^\infty \sum_{k=1}^{m-1} \frac{ 2\sum_i p_i^\ell (1-p_i)^m  }{mk} \\
		&\leq \sum_{m=n+1}^\infty \sum_{k=1}^{m-1} \frac{ 2C }{m^{\ell+1-\alpha} k} 
		\leq \int_n^\infty \frac{2C \log x }{x^{\ell+1-\alpha} } \,dx +\int_n^\infty \frac{2C}{x^{\ell+1-\alpha} } \,dx \\
		&= \frac{ 2C \brac{(\ell-\alpha) \log(n) +1} }{ (\ell-\alpha)^2 n^{\ell-\alpha}} + \frac{ 2C }{ (\ell-\alpha) n^{\ell-\alpha} } \leq \frac{ C_3 \log n }{n^{\ell-\alpha}}~.
	\end{align}
\end{proof}

\begin{lem} \label{lem:cov}
	Let $\mathbf X_n, X, Y$	be an i.i.d.\ sample from $p$ where $p_j \lesssim j^{-1/\alpha}$ for $\alpha\in (0,1)$.
	Then
	\begin{equation}
		\cov\brac{ J(n)-J(m_X), J(n)-J(m_Y) } \lesssim \frac{1}{n^{2-\alpha}}~.
	\end{equation}
\end{lem}

\begin{proof}[Proof of Lemma~\ref{lem:cov}]
	Using the law of total covariance, 
	\begin{align}
		&\cov\brac{J(n)-J(m_X), J(n)-J(m_Y)} \\
		&= \E_{X,Y}\brac{ \cov\brac{J(n)-J(m_X), J(n)-J(m_Y) \middle| X, Y} } \\
		&\quad +\cov\brac{ \E\brac{ J(n)-J(m_X) \middle| X,Y }, \E\brac{ J(n)-J(m_Y) \middle| X,Y } } \label{line:indepcov} \\
		&= \sum_i p_i^2 \var\brac{J(n)-J(m_i)} +\sum_i \sum_{j\neq i} p_i p_j \cov\brac{J(n)-J(m_i), J(n)-J(m_j)}~. \label{line:varcovar}
	\end{align}
	The last equality follows because the conditional expectations are functions of independent random variables, so are uncorrelated (line~\ref{line:indepcov}), and by writing the expectation over $X$ and $Y$ as the sums over $X=Y$ and $X\neq Y$.

	Beginning with the variance term of line~\ref{line:varcovar}, using Propositions~\ref{prop:scndmnt}~and~\ref{prop:mathind} then canceling terms, we have
	\begin{align}
		&\sum_i p_i^2 \var\brac{J(n)-J(m_i)} \\
		&= \sum_i p_i^2 \brac{ \E\brac{J(n)-J(m_i)}^2 -\paren{\E\brac{J(n) -J(m_i)}}^2 } \\
		&= \sum_i p_i^2 \brac{ \sum_{m=1}^n \sum_{k=1}^{n-m} \frac{ (1-p_i)^{m+k} }{mk} +\sum_{m=1}^n \sum_{k=n-m+1}^n \frac{(1-p_i)^m}{mk} -\paren{\sum_{m=1}^n \frac{(1-p_i)^m}{m} }^2 } \\
		&= \sum_{m=1}^n \sum_{k=n-m+1}^n \frac{ \sum_i p_i^2 (1-p_i)^m -\sum_i p_i^2 (1-p_i)^{m+k} }{mk} \\
		&= \sum_{m=1}^n \sum_{k=n-m+1}^n \frac{ \sum_i p_i^2 (1-p_i)^m }{mk} +\mathcal O\paren{ \frac{1}{n^{2-\alpha}}} ~. \label{line:varrate}
	\end{align}
	The last line follows from the second claim of Proposition~\ref{prop:varremainders}.

	Moving to the covariance from line~\ref{line:varcovar}, with Proposition~\ref{covarMulti} (stated and proved later in this appendix), we have
	\begin{align}
		&\sum_i \sum_{j\neq i} p_i p_j \cov\brac{J(n)-J(m_i), J(n)-J(m_j)} \\
		&= -\sum_{m=n+1}^\infty \frac{ \sum_i \sum_{j\neq i} p_i p_j \brac{1 -(1-p_i)^m} \brac{ 1-(1-p_j)^m } }{m^2} \label{line:cov1} \\
		&\quad -\sum_{m=n+1}^\infty \frac{ \sum_i \sum_{j\neq i} p_i p_j \brac{ (1-p_i)^m \log(1-p_i) +(1-p_j)^m \log(1-p_j) } }{m} \label{line:cov2} \\
		&\quad +\sum_{m=n+1}^\infty \frac{ \sum_i \sum_{j\neq i} p_i p_j \brac{ \mathcal O(p_i(1-p_i)^{m-1} (1-p_j)^m) + \mathcal O(p_j(1-p_j)^{m-1} (1-p_i)^m) } }{m} \label{line:cov3} \\
		&= -\sum_{m=n+1}^\infty \frac{ \sum_i p_i (1-p_i) }{m^2} +\mathcal O\paren{\frac{1}{n^{2-\alpha}}}~. \label{line:covrate} 
	\end{align}
	We explain the equality from line~\ref{line:covrate} taking lines~\ref{line:cov1}, \ref{line:cov2}, and \ref{line:cov3} separately.
	Beginning with line~\ref{line:cov1}, we use $\sum_{j\neq i} p_j = 1-p_i$, Proposition~\ref{lem:missingmass}, and integral bounds:
	\begin{align}
		&-\sum_{m=n+1}^\infty \frac{ \sum_i \sum_{j\neq i} p_i p_j \brac{1 -(1-p_i)^m} \brac{ 1-(1-p_j)^m } }{m^2} \\
		&= -\sum_{m=n+1}^\infty \frac{ \sum_i p_i (1-p_i) -2\sum_i p_i (1-p_i)^{m+1} +\brac{ \sum_i p_i(1-p_i)^m}^2 -\sum_i p_i^2(1-p_i)^{2m} }{m^2} \\
		&= -\sum_{m=n+1}^\infty \frac{ \sum_i p_i (1-p_i) -\mathcal O\paren{ m^{\alpha-1}} +\brac{ \mathcal O\paren{m^{\alpha-1}} }^2 -\mathcal O\paren{ m^{\alpha-2}} }{m^2} \\
		&= -\sum_{m=n+1}^\infty \frac{ \sum_i p_i (1-p_i) }{m^2} +\sum_{m=n+1}^\infty \mathcal O\paren{ \frac{1}{m^{3-\alpha}} }
		= -\sum_{m=n+1}^\infty \frac{ \sum_i p_i (1-p_i) }{m^2} +\mathcal O\paren{\frac{1}{n^{2-\alpha}}}~.
	\end{align}
	Moving to line~\ref{line:cov2}, using the fact that $-\log(1-x) \leq \frac{x}{1-x}$ for $x\in [0,1)$ and Proposition~\ref{lem:missingmass},
	\begin{align}
		&-\sum_{m=n+1}^\infty \frac{ \sum_i \sum_{j\neq i} p_i p_j \brac{ (1-p_i)^m \log(1-p_i) +(1-p_j)^m \log(1-p_j) } }{m} \\
		&= -\sum_{m=n+1}^\infty \frac{ 2\sum_i p_i (1-p_i)^{m+1} \log(1-p_i) }{m} 
		\leq \sum_{m=n+1}^\infty \frac{ 2\sum_i p_i^2 (1-p_i)^m }{m} \\
		&\lesssim \sum_{m=n+1}^\infty \frac{1}{m^{3-\alpha}}
		\leq \frac{1}{n^{2-\alpha}}~.
	\end{align}
	Last, looking at line~\ref{line:cov3} and using Proposition~\ref{lem:missingmass},
	\begin{align}
		&\sum_{m=n+1}^\infty \frac{ \sum_i \sum_{j\neq i} p_i p_j \brac{ \mathcal O(p_i(1-p_i)^{m-1} (1-p_j)^m) + \mathcal O(p_j(1-p_j)^{m-1} (1-p_i)^m) } }{m} \\
		&= \sum_{m=n+1}^\infty \frac{ \mathcal O\paren{ \sum_i p_i^2 (1-p_i)^{m-1} \brac{ -p_i(1-p_i)^m +\sum_j p_j(1-p_j)^m } } }{m} \\
		&= \sum_{m=n+1}^\infty \frac{ \mathcal O\paren{ m^{\alpha-3}} +\mathcal O\paren{m^{2\alpha-3}} }{m} =\mathcal O\paren{\frac{1}{n^{3-2\alpha}}}~.
	\end{align}

	Putting these pieces together, we have
	\begin{align}
		&\cov\brac{J(n) -J(m_X), J(n) -J(m_Y)} \\
		&= \sum_i p_i^2 \var\brac{J(n)-J(m_i)} +\sum_i \sum_{j\neq i} p_i p_j \cov\brac{J(n)-J(m_i), J(n)-J(m_j)} \\
		&= \sum_{m=1}^n \sum_{k=n-m+1}^n \frac{ \sum_i p_i^2 (1-p_i)^m }{mk} -\sum_{k=n+1}^\infty \frac{ \sum_i p_i (1-p_i) }{k^2} +\mathcal O\paren{\frac{1}{n^{2-\alpha}}} \label{line:varcocomb} \\
		&= \sum_{m=1}^n \sum_{k=n+1}^\infty \frac{ \sum_i p_i^2 (1-p_i)^m }{k(k-m)} - \sum_{m=1}^\infty \sum_{k=n+1}^\infty \frac{ \sum_i p_i^2 (1-p_i)^m }{k^2} +\mathcal O\paren{\frac{1}{n^{2-\alpha}}} \label{line:sumtransf} \\
		&= \sum_{m=1}^n \sum_{k=n+1}^\infty \frac{ m \sum_i p_i^2 (1-p_i)^m }{k^2(k-m)}  -\sum_{m=n+1}^\infty \sum_i p_i^2 (1-p_i)^m \sum_{k=n+1}^\infty \frac{1}{k^2} +\mathcal O\paren{\frac{1}{n^{2-\alpha}}} \label{line:fracsubtract} \\
		&= \sum_{m=1}^n \sum_{k=n+1}^\infty \frac{ m \sum_i p_i^2 (1-p_i)^m }{k^2(k-m)} +\mathcal O\paren{\frac{1}{n^{2-\alpha}}} \lesssim \frac{1}{n^{2-\alpha}} \label{line:simpintbound} ~.
	\end{align}
	Line~\ref{line:varcocomb} pulls from lines~\ref{line:varrate}~and~\ref{line:covrate}.
	Line~\ref{line:sumtransf} uses $\sum_{k=n-m+1}^n \frac{1}{k} = m\sum_{k=n+1}^\infty \frac{1}{k(k-m)}$, applied to the $1/(mk)$ term so that the $m$ cancels, on the first triple sum term and the geometric series on the second.
	Line~\ref{line:fracsubtract} follows by subtracting the values corresponding to $m=1,\dots, n$ of the second double summation from the first.
	The first equality of line~\ref{line:simpintbound} uses an integral and Proposition~\ref{lem:missingmass} to bound the remaining negative terms by $\mathcal O(1/n^{2-\alpha})$.
	We show the final bound of line~\ref{line:simpintbound} below following a similar pattern to Proposition~\ref{prop:varremainders},
	\begin{equation}
		\sum_{m=1}^n \sum_{k=n+1}^\infty \frac{ m \sum_i p_i^2 (1-p_i)^m }{k^2(k-m)} 
		= \sum_{k=n+1}^\infty \frac{ \sum_i p_i^2 (1-p_i) }{k^2(k-1)} +\sum_{m=2}^n \sum_{k=n+1}^\infty \frac{ m \sum_i p_i^2 (1-p_i)^m }{k^2(k-m)} 
	\end{equation}

	For the $m=1$ term, 
	\begin{equation}
		\sum_{k=n+1}^\infty \frac{ \sum_i p_i^2 (1-p_i) }{k^2(k-1)}
		\leq \sum_{k=n+1}^\infty \frac{1}{k^2(k-1)}
		\leq \int_{n}^\infty \frac{1}{x^2(x-1)} \,dx \lesssim \frac{1}{n^2}~.
	\end{equation}
	For terms where $m= 2,\dots, n$, start by considering the inner summation over $k$.
	We use the fact that $\sum_{i=n+1}^\infty f(i) \leq f(n+1) +\int_{n+1}^\infty f(x)\,dx$ for decreasing $f$ so that
	\begin{align}
		&\sum_{k=n+1}^\infty \frac{ m }{k^2(k-m)} 
		\leq \frac{m}{(n+1)^2 (n+1-m)} +\int_{n+1}^\infty \frac{m}{x^2(x-m)} \,dx \\
		&=  \frac{m}{(n+1)^3 \paren{ 1-\frac{m}{n+1}}} -\frac{1}{m} \log\paren{ 1-\frac{m}{n+1}} -\frac{1}{n+1} \\
		&= \frac{1}{(n+1)^2} \sum_{k=1}^\infty \paren{\frac{m}{n+1}}^k +\frac{1}{n+1} \sum_{k=1}^\infty \frac{1}{k+1} \paren{\frac{m}{n+1}}^k \label{line:geomlog2} ~.
	\end{align}
	The equality from line~\ref{line:geomlog2} uses the geometric series on the first summation.
	For the second summation, it applies a Taylor expansion then cancels the $k=1$ summand with $1/(n+1)$.
	Now, applying Proposition~\ref{lem:missingmass}, the inequality above, and an integral to bound the summation,
	\begin{align}
		&\sum_{m=2}^n \sum_{k=n+1}^\infty \frac{ m \sum_i p_i^2 (1-p_i)^m }{k^2(k-m)}
		\leq \sum_{m=2}^n \sum_{k=1}^\infty \brac{ \frac{ Cm^{k-2+\alpha} }{(n+1)^{k+2}} +\frac{C m^{k-2+\alpha}}{(k+1) (n+1)^{k+1}} } \\
		&\leq C \int_1^n \sum_{k=1}^\infty \brac{ \frac{x^{k-2+\alpha} }{(n+1)^{k+2}} +\frac{x^{k-2+\alpha}}{(k+1) (n+1)^{k+1}} } \,dx \\
		&= C \sum_{k=1}^\infty \paren{ \int_1^n x^{k-2+\alpha}\,dx } \brac{ \frac{1}{(n+1)^{k+2}} +\frac{1}{(k+1)(n+1)^{k+1}} } \\
		&\leq Cn^{\alpha-1} \sum_{k=1}^\infty \frac{1}{k-1+\alpha} \paren{\frac{n}{n+1}}^k \brac{ \frac{1}{(n+1)^{2}} +\frac{1}{(k+1)(n+1)} } \label{line:firstsep2} \\
		&\leq \frac{Cn^{\alpha-1}}{\alpha(n+1)^2} \sum_{k=1}^\infty \frac{1}{k} \paren{\frac{n}{n+1}}^k +\frac{Cn^{\alpha-1}}{\alpha(n+1)} \sum_{k=1}^\infty \frac{1}{k(k+1)} \\
		&= \frac{Cn^{\alpha-1}\log(n+1) }{\alpha(n+1)^2} +\frac{Cn^{\alpha-1}}{\alpha(n+1)} \label{line:summations}
		\lesssim \frac{1}{n^{2-\alpha}}~.
	\end{align}
	Line~\ref{line:firstsep2} evaluates $\int_1^n x^{k-2+\alpha}\,dx=(n^{k-1+\alpha}-1)/(k-1+\alpha) \leq n^{k-1+\alpha}/(k-1+\alpha)$, which holds uniformly for every $k\geq1$ since $k-1+\alpha\geq\alpha>0$, so no separate case is needed at $k=1$.
	The next inequality uses $k-1+\alpha\geq\alpha k$, which holds because $(k-1)(1-\alpha)\geq0$ for every integer $k\geq1$, and, in the second sum, drops $(n/(n+1))^k\leq1$, which then telescopes since $1/(k(k+1)) =1/k-1/(k+1)$.
	The equality of line~\ref{line:summations} uses $\sum_{k=1}^\infty r^k/k=-\log(1-r)$ with $r=n/(n+1)$, so $-\log(1-r)=\log(n+1)$.
	Both terms are dominated by $1/n^{2-\alpha}$: the first since $\log(n)/n^{3-\alpha} =o(1/n^{2-\alpha})$, and the second is exactly of order $1/n^{2-\alpha}$.
\end{proof}

\begin{prop}\label{covarMulti}
	Let $(M, K, n-M-K)$ be a multinomial random vector of $n$ trials with probabilities $(p, q, 1-p-q)$.
	Then
	\begin{align}
		&\cov\brac{ J(n)-J(M), J(n)-J(K) } \\
		&= -\sum_{m=n+1}^\infty \frac{ \brac{1 -(1-p)^m} \brac{ 1-(1-q)^m } }{m^2}+\frac{ (1-p)^m \log(1-p) +(1-q)^m \log(1-q) }{m} \\
		&\quad +\sum_{m=n+1}^\infty \frac{ \mathcal O(p(1-p)^{m-1} (1-q)^m) + \mathcal O(q(1-q)^{m-1} (1-p)^m) }{m}~.
	\end{align}
\end{prop}

\begin{proof}
	Consider the joint expectation.
	Begin by applying Proposition~\ref{prop:multiProd} (stated and proved later in this appendix) to simplify:
	\begin{align}
		&\E\brac{ [J(n)-J(M)] [J(n)-J(K)] } \\
		&= \sum_{m=1}^n \sum_{k=1}^m \frac{ (1-p)^m \paren{1-\frac{q}{1-p}}^k + (1-q)^m \paren{1-\frac{p}{1-q}}^k }{mk} \label{line:compTaylor} 
		+\sum_{m=1}^n \frac{ 1-(1-p)^m -(1-q)^m}{m^2} \\
		&= -\sum_{m=1}^n \frac{ (1-p)^m \brac{\log(q) -\log(1-p)} + (1-q)^m \brac{ \log(p) -\log(1-q)} }{m} \label{line:sumtoint} \\
		&\quad -\sum_{m=1}^n \sum_{k=m+1}^\infty \frac{ (1-p)^m \paren{1-\frac{q}{1-p}}^k + (1-q)^m \paren{1-\frac{p}{1-q}}^k }{mk} \label{line:sumtoint2} \\
		&\quad +\sum_{m=1}^n \frac{ 1-(1-p)^m -(1-q)^m}{m^2} \\
		&= -\log(p) \log(q) -\sum_{m=1}^n \frac{ (1-p)^m \log(q) + (1-q)^m \log(p) }{m} \label{line:propSumtoint} \\
		&\quad +\sum_{m=n+1}^\infty \sum_{k=m+1}^\infty \frac{ (1-p)^m \paren{1-\frac{q}{1-p}}^k + (1-q)^m \paren{1-\frac{p}{1-q}}^k  }{mk} \label{line:upperRemain} \\
		&\quad +\sum_{m=1}^n \frac{ (1-p)^m \log(1-p) + (1-q)^m \log(1-q) }{m} \label{line:partLog} \\
		&\quad +\log(p)\log(1-p) +\log(q)\log(1-q) \\ 
		&\quad +\sum_{m=1}^n \frac{ 1-(1-p)^m -(1-q)^m}{m^2} -\zeta(2) +\li2(1-p) +\li2(1-q) \label{line:dilog} ~.
	\end{align}
	Line~\ref{line:sumtoint} extends the upper limit of the inner summation from line~\ref{line:compTaylor} to $\infty$, writes the two corresponding Taylor series as natural logarithms, then uses the logarithm quotient rule; line~\ref{line:sumtoint2} subtracts the added terms.
	The equality from line~\ref{line:propSumtoint} follows by applying Proposition~\ref{prop:sumtoint} to the double summation of line~\ref{line:sumtoint2} and rearranging terms.

	We simplify each line separately.
	Starting with line~\ref{line:propSumtoint}, write each natural logarithm as a Taylor series then cancel terms:
	\begin{align}
		&-\log(p) \log(q) -\sum_{m=1}^n \frac{ (1-p)^m \log(q) + (1-q)^m \log(p) }{m} \\
		&= -\sum_{m=1}^\infty \sum_{k=1}^\infty \frac{ (1-p)^m (1-q)^k }{mk} +\sum_{m=1}^n \sum_{k=1}^\infty \frac{ (1-p)^m (1-q)^k + (1-q)^m (1-p)^k }{mk} \\
		&= \sum_{m=1}^n \sum_{k=1}^n \frac{ (1-p)^m (1-q)^k }{mk} -\sum_{m=n+1}^\infty \sum_{k=n+1}^\infty \frac{ (1-p)^m (1-q)^k }{mk} ~.
	\end{align}

	Moving to line~\ref{line:upperRemain}, note that $(a-b)^k =a^k +\mathcal O(ka^{k-1}b)$ for $a>b$ with $a=(1-p)(1-q)$ and $b=pq$, so that
	\begin{align}
		&\sum_{m=n+1}^\infty \sum_{k=m+1}^\infty \frac{ (1-p)^m \paren{1-\frac{q}{1-p}}^k }{mk} 
		=\sum_{m=n+1}^\infty \sum_{k=m+1}^\infty \frac{ (1-p)^{m-k} (1-p-q)^k }{mk} \\
		&= \sum_{m=n+1}^\infty \sum_{k=m+1}^\infty \frac{ (1-p)^m (1-p)^k }{mk} +\frac{ \mathcal O\paren{ pq(1-p)^{m-1}(1-q)^{k-1} } }{m} \\
		&= \sum_{m=n+1}^\infty \sum_{k=m+1}^\infty \frac{ (1-p)^m (1-q)^k }{mk} +\sum_{m=n+1}^\infty \frac{ \mathcal O\paren{ p(1-p)^{m-1}(1-q)^m } }{m} 
	\end{align}
	where the last equality uses the geometric series identity.
	Using this and re-indexing, we have
	\begin{align}
		&\sum_{m=n+1}^\infty \sum_{k=m+1}^\infty \frac{ (1-p)^m \paren{1-\frac{q}{1-p}}^k + (1-q)^m \paren{1-\frac{p}{1-q}}^k  }{mk} \label{line:upperRemain2} \\
		&= \sum_{m=n+1}^\infty \sum_{k=m+1}^\infty \frac{  (1-p)^m (1-q)^k +(1-q)^m (1-p)^k }{mk} \\
		&\quad +\sum_{m=n+1}^\infty \frac{ \mathcal O\paren{ p(1-p)^{m-1}(1-q)^m } + \mathcal O\paren{ q(1-q)^{m-1}(1-p)^m } }{m} \\
		&= \sum_{m=n+1}^\infty \sum_{k=n+1}^\infty \frac{ (1-p)^m (1-q)^k }{mk} -\sum_{m=n+1}^\infty \frac{(1-p)^m(1-q)^m}{m^2} \\
		&\quad +\sum_{m=n+1}^\infty \frac{ \mathcal O\paren{ p(1-p)^{m-1}(1-q)^m } + \mathcal O\paren{ q(1-q)^{m-1}(1-p)^m } }{m} ~.
	\end{align}

	Moving to line~\ref{line:partLog}, 
	write $\log(p)$ and $\log(q)$ as Taylor series then cancel terms:
	\begin{align}
		&\sum_{m=1}^n \frac{ (1-p)^m \log(1-p) + (1-q)^m \log(1-q) }{m} 
		+\log(p)\log(1-p) +\log(q)\log(1-q) \\
		&= -\sum_{m=n+1}^\infty \frac{ (1-p)^m \log(1-p) +(1-q)^m \log(1-q) }{m} ~.
	\end{align}

	Last, line~\ref{line:dilog}, writing the zeta function and dilogarithm functions as summations then canceling terms,
	\begin{align}
		&\sum_{m=1}^n \frac{ 1-(1-p)^m -(1-q)^m}{m^2} -\zeta(2) +\li2(1-p) +\li2(1-q) \\
		&= \sum_{m=1}^n \frac{ 1-(1-p)^m -(1-q)^m}{m^2} -\sum_{m=1}^\infty \frac{ 1-(1-p)^m -(1-q)^m}{m^2} \\
		&= -\sum_{m=n+1}^\infty \frac{ 1-(1-p)^m -(1-q)^m}{m^2} ~.
	\end{align}

	Now we move to the covariance, with the calculations above to simplify the joint expectation and Proposition~\ref{prop:mathind} to simplify the product of expectations, we have
	\begin{align}
		&\cov\brac{ J(n)-J(M), J(n)-J(K) } \\
		&= \E\brac{ [J(n)-J(M)] [J(n)-J(K)] } - \E\brac{ J(n)-J(M) } \E\brac{ J(n)-J(K) } \label{line:covExpt} \\ 
		&= \sum_{m=1}^n \sum_{k=1}^n \frac{ (1-p)^m (1-q)^k }{mk} -\sum_{m=n+1}^\infty \frac{ 1 - (1-p)^m -(1-q)^m +(1-p)^m(1-q)^m}{m^2} \\
		&\quad +\sum_{m=n+1}^\infty \frac{ \mathcal O\paren{ p(1-p)^{m-1}(1-q)^m } + \mathcal O\paren{ q(1-q)^{m-1}(1-p)^m } }{m} \\
		&\quad -\sum_{m=n+1}^\infty \frac{ (1-p)^m \log(1-p) +(1-q)^m \log(1-q) }{m} -\sum_{m=1}^n \sum_{k=1}^n \frac{ (1-p)^m (1-q)^k }{mk} \\
		&= -\sum_{m=n+1}^\infty \frac{ \brac{1 -(1-p)^m} \brac{ 1-(1-q)^m } }{m^2}+\frac{ (1-p)^m \log(1-p) +(1-q)^m \log(1-q) }{m} \\
		&\quad +\sum_{m=n+1}^\infty \frac{ \mathcal O(p(1-p)^{m-1} (1-q)^m) + \mathcal O(q(1-q)^{m-1} (1-p)^m) }{m}~.
	\end{align}
\end{proof}

\begin{prop}\label{prop:multiProd}
	Let $(M, K, n-M-K)$ be a multinomial random vector of $n$ trials with probabilities $(p, q, 1-p-q)$.
	Then
	\begin{align}
		&\E\brac{ \brac{ J(n)-J(M)} \brac{ J(n)-J(K) } } =\sum_{m=1}^n \frac{ 1-(1-p)^m -(1-q)^m}{m^2} \\
		&\quad +\sum_{m=1}^n \sum_{k=1}^m \frac{ (1-q)^m \paren{1-\frac{p}{1-q}}^k + (1-p)^m \paren{1-\frac{q}{1-p}}^k }{mk}~.
	\end{align}
\end{prop}

\begin{proof}[Proof]
	Using Proposition~\ref{prop:mathind} on line~\ref{line:partialJoint},
	\begin{align}
		&\E\brac{ \brac{J(n)-J(M)} \brac{J(n)-J(K)} } \\
		&= \sum_{m=0}^n \sum_{k=0}^{n-m} \brac{J(n)-J(m)} \brac{J(n)-J(k)} \binom{n}{m,k} p^m q^k (1-p-q)^{n-m-k} \\
		&= \sum_{m=0}^n \brac{J(n)-J(m)} \binom{n}{m} p^m(1-p)^{n-m} \\
		&\quad \times \sum_{k=0}^{n-m} \brac{J(n)-J(k)} \binom{n-m}{k} \paren{\frac{q}{1-p}}^k \paren{1-\frac{q}{1-p}}^{n-m-k} \label{line:partialJoint} \\
		&= \sum_{m=0}^n \brac{J(n)-J(m)} \paren{J(n)-\sum_{k=1}^{n-m} \frac{1-\paren{1-\frac{q}{1-p}}^k}{k} } \binom{n}{m} p^m(1-p)^{n-m}~.\label{line:jointExp1Simp}
	\end{align}

	Proceeding with induction on the simplified joint expectation, let $S_n$ equal line~\ref{line:jointExp1Simp}.
	We start by breaking $S_{n+1}$ into four terms with the fact that $J(n+1) = J(n)+\frac{1}{n+1}$.
	The details underlying the simplification of line~\ref{maincov} are below.
	\begin{align}
	&S_{n+1} = \sum_{m=0}^{n+1} \brac{J(n+1)-J(m)} \paren{ J(n+1) -\sum_{k=1}^{n+1-m} \frac{1-\paren{1-\frac{q}{1-p}}^k}{k} } \\
	&\quad\quad\quad \times \binom{n+1}{m} p^m(1-p)^{n+1-m} \\
	&= \sum_{m=0}^{n+1} \brac{J(n)-J(m)} \paren{J(n)-\sum_{k=1}^{n+1-m} \frac{1-\paren{1-\frac{q}{1-p}}^k}{k} } \binom{n+1}{m} p^m(1-p)^{n+1-m} \label{maincov} \\
	&\quad +\frac{1}{n+1} \sum_{m=0}^{n+1} \paren{J(n)-\sum_{k=1}^{n+1-m} \frac{1-\paren{1-\frac{q}{1-p}}^k}{k} } \binom{n+1}{m} p^m(1-p)^{n+1-m} \\
	&\quad +\frac{1}{n+1} \sum_{m=0}^{n+1} \brac{J(n)-J(m)} \binom{n+1}{m} p^m (1-p)^{n+1-m} +\frac{1}{(n+1)^2} \\
	&= S_n +\frac{1}{n+1} \sum_{m=0}^{n+1} \brac{J(n)-J(m)} \paren{1-\frac{q}{1-p}}^{n+1-m}  \binom{n+1}{m} p^m(1-p)^{n+1-m} \label{line:genmathind} \\
	&\quad +\frac{(1-p)^{n+1}}{n+1} \paren{ J(n) -\sum_{k=1}^{n+1} \frac{1-\paren{1-\frac{q}{1-p}}^k}{k} } +\frac{1}{(n+1)^2} \\
	&= S_n +\frac{1}{n+1} \sum_{k=1}^{n+1} \frac{ (1-q)^{n+1} \paren{1-\frac{p}{1-q}}^k +(1-p)^{n+1} \paren{1-\frac{q}{1-p}}^k}{k} \\
	&\quad +\frac{1 -(1-p)^{n+1} -(1-q)^{n+1} }{(n+1)^2}~.
	\end{align}
	The penultimate equality cancels terms using the details provided below and the last equality uses Proposition~\ref{propHarmonicBinomial} on the summation term of line~\ref{line:genmathind}.

	With induction using the fact that $S_1=1-p-q$, we arrive at the desired claim:
	\begin{align}
		S_n
		&= \sum_{m=1}^n \sum_{k=1}^m \frac{ (1-q)^m \paren{1-\frac{p}{1-q}}^k + (1-p)^m \paren{1-\frac{q}{1-p}}^k }{mk} \\
		&\quad +\sum_{m=1}^n \frac{ 1-(1-p)^m -(1-q)^m}{m^2}~.
	\end{align}

	Expanding on the calculations from line~\ref{maincov}, we use the identity $\binom{n+1}{m} =\binom{n}{m} +\binom{n}{m-1}$, re-indexing, and Proposition~\ref{propReciprocalBinomial},
	\begin{align}
	&\sum_{m=0}^{n+1} \brac{J(n)-J(m)} \paren{J(n)-\sum_{k=1}^{n+1-m} \frac{1-\paren{1-\frac{q}{1-p}}^k}{k} } \\
	&\quad \brac{ \binom{n}{m} +\binom{n}{m-1}} p^m(1-p)^{n+1-m} \\
	&= \sum_{m=0}^{n} \brac{J(n)-J(m)} \paren{J(n)-\sum_{k=1}^{n+1-m} \frac{1-\paren{1-\frac{q}{1-p}}^k}{k} } \binom{n}{m} p^m(1-p)^{n+1-m} \\
	&\quad +\sum_{m=0}^{n} \brac{J(n)-J(m+1)} \paren{J(n)-\sum_{k=1}^{n-m} \frac{1-\paren{1-\frac{q}{1-p}}^k}{k} } \binom{n}{m} p^{m+1} (1-p)^{n-m} \\
	&= \sum_{m=0}^{n} \brac{J(n)-J(m)} \paren{J(n)-\sum_{k=1}^{n-m} \frac{1-\paren{1-\frac{q}{1-p}}^k}{k} } \binom{n}{m} p^m(1-p)^{n+1-m} \\
	&\quad -\sum_{m=0}^{n} \brac{J(n)-J(m)} \paren{\frac{1-\paren{1-\frac{q}{1-p}}^{n+1-m}}{n+1-m} } \binom{n}{m} p^m(1-p)^{n+1-m} \\
	&\quad +\sum_{m=0}^{n} \brac{J(n)-J(m)} \paren{J(n)-\sum_{k=1}^{n-m} \frac{1-\paren{1-\frac{q}{1-p}}^k}{k} } \binom{n}{m} p^{m+1} (1-p)^{n-m} \\
	&\quad -\sum_{m=0}^{n} \frac{1}{m+1} \paren{J(n)-\sum_{k=1}^{n-m} \frac{1-\paren{1-\frac{q}{1-p}}^k}{k} } \binom{n}{m} p^{m+1} (1-p)^{n-m} \\
	&= S_n -\frac{1}{n+1} \sum_{m=0}^{n} \brac{J(n)-J(m)} \paren{1-\paren{1-\frac{q}{1-p}}^{n+1-m} } \binom{n+1}{m} p^m(1-p)^{n+1-m} \\
	&\quad -\frac{1}{n+1} \sum_{m=0}^{n} \paren{J(n)-\sum_{k=1}^{n-m} \frac{1-\paren{1-\frac{q}{1-p}}^k}{k} } \binom{n+1}{m+1} p^{m+1} (1-p)^{n-m} \\
	&= S_n -\frac{1}{n+1} \sum_{m=0}^{n+1} \brac{J(n)-J(m)} \binom{n+1}{m} p^m(1-p)^{n+1-m} \\
	&\quad +\frac{1}{n+1} \sum_{m=0}^{n+1} \brac{J(n)-J(m)} \paren{1-\frac{q}{1-p}}^{n+1-m}  \binom{n+1}{m} p^m(1-p)^{n+1-m} \\
	&\quad -\frac{1}{n+1} \sum_{m=0}^{n+1} \paren{J(n)-\sum_{k=1}^{n+1-m} \frac{1-\paren{1-\frac{q}{1-p}}^k}{k} } \binom{n+1}{m} p^{m} (1-p)^{n+1-m} \\
	&\quad +\frac{(1-p)^{n+1}}{n+1} \paren{ J(n) -\sum_{k=1}^{n+1} \frac{1-\paren{1-\frac{q}{1-p}}^k}{k} } ~.
	\end{align}
\end{proof}

\begin{prop} \label{prop:addterms}
	Let $m_X$ be $\text{Binomial}(n, p(X))$ given $X$ and $p_j \lesssim j^{-1/\alpha}$.
	Then
	\begin{align}
		\E_X \brac{ p_X \paren{ \var\brac{\frac{1}{m_X+1}} -2\cov\brac{ J(n) -J(m_X), \frac{1}{m_X+1}} } } 
		= \mathcal O\paren{\frac{\log n}{n^{2-\alpha}} } ~.
	\end{align}
\end{prop}

\begin{proof}
	Assume $M$ is $\text{Binomial}(n, p)$.
	Using Proposition \ref{propReciprocalBinomial}, we have
	\begin{align}
		&\E\brac{ \paren{ \frac{1}{M+1} }^2 }
		= \sum_{m=0}^n \paren{ \frac{1}{m+1} }^2 \binom{n}{m} p^m (1-p)^{n-m} \\
		&= \frac{1}{(n+1)p} \sum_{m=1}^{n+1} \frac{1}{m}  \binom{n+1}{m} p^m (1-p)^{n+1-m} \\
		&= \frac{1}{(n+1)p} \sum_{m=0}^n \frac{(1-p)^m}{n-m+1} -\frac{ J(n+1) (1-p)^{n+1} }{(n+1)p} ~.
	\end{align}
	Using Proposition~\ref{prop:recipmoment}, $\E[1/(M+1)] =[1-(1-p)^{n+1}]/[(n+1)p]$, 
	\begin{align}
		&\E_X \brac{ p_X \var\brac{\frac{1}{m_X+1}}  } 
		=\E_X \brac{ p_X \paren{ \E\brac{ \paren{ \frac{1}{m_X+1} }^2 } -\paren{\E\brac{\frac{1}{m_X+1}} }^2 }} \\
		&= \E\brac{ \frac{1}{n+1} \sum_{m=0}^n \frac{(1-p_X)^m}{n-m+1} -\frac{ J(n+1)(1-p_X)^{n+1} }{(n+1)} -\frac{ \paren{ 1 -(1-p_X)^{n+1} }^2 }{(n+1)^2 p_X } } \\
		&= \E\brac{ \frac{1}{n+1} \sum_{m=0}^n \frac{(1-p_X)^m}{n-m+1}  -\frac{1 -2(1-p_X)^{n+1} +(1-p_X)^{2(n+1)} }{(n+1)^2 p_X } } +\mathcal{O}\paren{\frac{\log n}{n^{2-\alpha}}} \label{line:expcompbound} \\
		&= \E\brac{ \frac{1}{(n+1)^2} \sum_{m=0}^n \frac{m (1-p_X)^m}{n-m+1} -\frac{ (1-p_X)^{n+1} [1-(1-p_X)^{n+1} ] }{(n+1)^2 p_X} } +\mathcal{O}\paren{\frac{\log n}{n^{2-\alpha} }} \\
		&\leq \frac{1}{(n+1)^2} \sum_{m=1}^n \frac{C m^\alpha}{n-m+1} +\mathcal{O}\paren{\frac{\log n}{n^{2-\alpha}}} \\
		&\leq \frac{C}{(n+1)^2} \int_1^n \frac{x^\alpha}{n-x+1}\,dx +\mathcal{O}\paren{\frac{\log n}{n^{2-\alpha}}} \\
		&= \frac{C}{(n+1)^2} \int_1^n \frac{(n+1-u)^\alpha}{u}\,du +\mathcal{O}\paren{\frac{\log n}{n^{2-\alpha}}} \\
		&= \frac{Cn^\alpha}{(n+1)^2} \int_1^n \frac{1}{u}\,du +\mathcal{O}\paren{\frac{\log n}{n^{2-\alpha}}} 
		= \mathcal{O}\paren{\frac{\log n}{n^{2-\alpha}}}~.
	\end{align}
	Equation~\ref{line:expcompbound} follows from the previous line by bounding the center term within the expectation using Proposition~\ref{lem:missingmass}, $\E[(1-p_X)^{n+1}] =\mathcal{O}(1/n^{1-\alpha})$ and expanding the square.
	For the following equality, note that $1/p =\sum_{m=0}^n (1-p)^m +(1-p)^{n+1}/p$ using the geometric series, separating the terms $m=0,\dots,n$, factoring $(1-p)^{n+1}$ from the remaining terms, applying the geometric series again.
	Apply this equality to the 1 in the numerator of the second term in the expectation, combining the summations and remaining terms in the numerator.
	For the following inequality, use Proposition~\ref{lem:missingmass} on the summation; on the following term use $\sum_j (1-p_j)^{n+1} [1-(1-p_j)^{n+1}] \leq \sum_j [1-(1-p_j)^{n+1}] =\mathcal{O}(n^{\alpha})$.
	Proof of the last bound can be found at the end of the proof for Proposition~\ref{prop:plugbias}.

	Using Propositions~\ref{prop:mathind} and \ref{propReciprocalBinomial}, then canceling terms, we have
	\begin{align}
		&\E\brac{ \frac{ J(n) -J(M) }{M+1} }
		= \sum_{m=0}^n \frac{ J(n) -J(m) }{m+1} \binom{n}{m} p^m (1-p)^{n-m} \\
		&= \frac{1}{(n+1)p} \sum_{m=1}^{n+1} \brac{ J(n) -J(m) +\frac{1}{m} } \binom{n+1}{m} p^m (1-p)^{n+1-m} \\
		&= \frac{1}{(n+1)p} \sum_{m=1}^n \brac{ \frac{(1-p)^m}{m} +\frac{ (1-p)^m}{n-m+1} } 
		+\mathcal{O}\paren{(1-p)^{n+1}}~.
	\end{align}

	Then
	\begin{align}
		& \E\brac{ p_X \cov\brac{ J(n) -J(m_X), \frac{1}{m_X+1}} } \\
		&= \E\brac{ p_X \paren{ \E\brac{ \frac{ J(n) -J(m_X) }{m_X+1} } -\E\brac{ J(n) -J(m_X) }  \E\brac{ \frac{1}{m_X+1} } } } \\
		&= \E\brac{\frac{1}{n+1} \sum_{m=1}^n \brac{ \frac{(1-p_X)^m}{m} +\frac{ (1-p_X)^m}{n-m+1} } +\mathcal{O}\paren{(1-p_X)^{n+1}} } \\
		&\quad -\E\brac{ \sum_{m=1}^n \frac{ (1-p_X)^m }{m} \frac{ 1-(1-p_X)^{n+1}}{n+1}  } \\
		&= \frac{1}{n+1} \sum_{m=1}^n \frac{ \E\brac{ (1-p_X)^m } }{n-m+1} +\mathcal{O}\paren{(1-p_X)^{n+1}} = \mathcal{O}\paren{\frac{\log n}{n^{2-\alpha}}}~.
	\end{align}
	The bound for the last summation follows a similar logic to before.
\end{proof}

\section{Proofs for Asymptotic Efficiency}\label{apdx:clt}

\begin{proof}[Proof of Theorem~\ref{clt}]
	For random variables $U$ and $V$, define the bounded Lipschitz metric as
	\begin{equation}
		d(U,V) := \sup_{f\in\text{BL}(\R)} \abs{ \E[f(U)] -\E[f(V)] }
	\end{equation}
	where 
	\begin{equation}
		\text{BL}(\R) := \curly{f:\R \rightarrow \R: \sup_{x\in\R} \abs{f(x)}\leq 1, \sup_{x\neq y} \frac{\abs{f(x)-f(y)}}{\abs{x-y}}\leq 1}~.
	\end{equation}
	Define the random variable $Z\sim N(0, \var[\log p(X)])$ and the oracle estimator as $H^*(\mathbf X_n) := -\frac{1}{n} \sum_{i=1}^n \log p\paren{X^{(i)}}$.
	Using the triangle inequality,
	\begin{align}
		d\paren{ \sqrt n\brac{\hat H\paren{\mathbf X_n}-H(X)}, Z}
		&\leq d\paren{ \sqrt n\brac{\hat H\paren{\mathbf X_n}-H(X)}, \sqrt n\brac{H^*\paren{\mathbf X_n}-H(X)}} \\
		&\quad +d\paren{ \sqrt n\brac{H^*\paren{\mathbf X_n}-H(X)}, Z}~. \label{convmetric}
	\end{align}
	We proceed by showing that both converge to zero.
	Beginning with the first term, using the Lipschitz property,
	\begin{align}
		&d\paren{ \sqrt n\brac{\hat H\paren{\mathbf X_n}-H(X)}, \sqrt n\brac{H^*\paren{\mathbf X_n}-H(X)}} \\
		&\leq \sqrt n \E \abs{\hat H\paren{\mathbf X_n}-H^*\paren{\mathbf X_n}} 
		\leq \sqrt{ n\E \paren{\hat H\paren{\mathbf X_n}-H^*\paren{\mathbf X_n}}^2} \\
		&= \sqrt n \sqrt{\paren{\E\brac{\hat H\paren{\mathbf X_n}}-H(X)}^2 +\var\brac{\hat H\paren{\mathbf X_n}-H^*\paren{\mathbf X_n}}} \\
		&= \sqrt n \sqrt{ o\paren{\frac{1}{n}} + \mathcal{O}\paren{\frac{\log n}{n^{1.5}} }} \rightarrow 0~.
	\end{align}
	The last line follows from Proposition~\ref{bias} and Proposition~\ref{prop:vardiff}, and converges to zero when $p_j =o(j^{-2})$.

	Now we consider the second metric term.
	With a Berry-Esseen-style, Wasserstein central limit theorem~\cite[Theorem 3.2]{barbour2005},
	\begin{equation}
		d\paren{ \sqrt n\brac{H^*\paren{\mathbf X_n}-H(X)}, Z} 
		\leq \frac{3 \E \abs{-\log p(X) - H(X)}^3}{\sqrt n \var\brac{\log p(X)} } \rightarrow 0~.
	\end{equation}
	Using the integral test, it is straightforward to show that $\E\abs{-\log p(X) - H(X)}^3<\infty$ for $p_j =o(j^{-2})$:
	$\sum j^{-2} (\log j)^3$ converges because $\int x^{-2} (\log x)^3\,dx$ does.

	Together, these show that $d\paren{ \sqrt n\brac{\hat H\paren{\mathbf X_n}-H(X)}, Z} \rightarrow 0$ as $n\rightarrow\infty$. 
	Using~\cite[Thm. 8.3.2]{bogachev2007measure}, $\sqrt n\brac{\hat H\paren{\mathbf X_n}-H(X)}\leadsto Z$.
\end{proof}

\begin{prop}\label{prop:vardiff}
	If $p_j \lesssim j^{-1/\alpha}$ for $\alpha\in (0,1)$, then
	\begin{equation}
		\var\brac{\hat H\paren{\mathbf X_n}-H^*\paren{\mathbf X_n}} \lesssim \frac{\log n}{n^{2-\alpha}} ~.
	\end{equation}
\end{prop}

\begin{proof}
	Using a variance property, we separate terms into variance and covariance:
	\begin{align}
		&\var\brac{\hat H\paren{\mathbf X_n}-H^*\paren{\mathbf X_n}} \\
		&= \var\brac{\hat H\paren{\mathbf X_n}} +\var\brac{H^*\paren{\mathbf X_n}} - 2\cov\brac{\hat H\paren{\mathbf X_n}, H^*\paren{\mathbf X_n}} ~.
	\end{align}

	From Theorem~\ref{var}, 
	\begin{equation}
		\var\brac{\hat H\paren{\mathbf X_n}} = \frac{\var\brac{ \log p(X) }}{n} +\mathcal O\paren{\frac{\log n}{n^{2-\alpha}} } ~.
	\end{equation}
	Because the sample is i.i.d.,
	\begin{equation}
		\var\brac{H^*\paren{\mathbf X_n}}
		= \var\brac{-\frac{1}{n}\sum_{i=1}^n \log p\paren{X^{(i)}}} 
		= \frac{1}{n} \var\brac{\log p(X)}~.
	\end{equation}

	Working on the covariance terms, recall that $m^{(i)} =\sum_{j=1}^n I\paren{X^{(i)}=X^{(j)}}$ and is distributed as $m^{(i)}-1 \sim \text{Binomial}(n-1, p(X^{(i)}))$ given $X^{(i)}$ as $X^{(i)}$ is always counted in the sum.
	Then
	\begin{align}
		&\cov\brac{\hat H\paren{\mathbf X_n}, H^*\paren{\mathbf X_n}} \\
		&= \cov\brac{ \frac{1}{n}\sum_{i=1}^n \brac{J(n-1) -J\paren{m^{(i)}-1}}, -\frac{1}{n}\sum_{i=1}^n \log p\paren{X^{(i)}}} \\
		&= \frac{1}{n^2} \sum_{i=1}^n \sum_{k=1}^n \cov\brac{J(n-1)- J\paren{m^{(i)}-1}, -\log p\paren{X^{(k)}}} \\
		&= \frac{1}{n} \cov\brac{ J(n-1)-J\paren{m^{(1)}-1}, -\log p\paren{X^{(1)}}} \label{coveq}\\
		&\quad +\frac{n-1}{n} \cov\brac{J(n-1)-J\paren{m^{(1)}-1}, -\log p\paren{X^{(2)}}}~. \label{covneq}
	\end{align}

	Considering the covariance term of line~\ref{coveq}, we begin with the joint expectation.
	Using Proposition~\ref{prop:mathind},
	\begin{align}
		&\E\brac{\paren{J\paren{m^{(1)}-1}-J(n-1)} \log p\paren{X^{(1)}}} \\
		&=\E_{X^{(1)}}\brac{\E_{X^{(2)},\dots,X^{(n)} }\brac{ J\paren{m^{(1)}-1}-J(n-1)\middle| X^{(1)}} \log p\paren{X^{(1)}}} \\
		&=\E_{X^{(1)}}\brac{ \paren{\log p\paren{X^{(1)}} +\sum_{k=n}^\infty \frac{ (1-p_X)^k}{k}} \log p\paren{X^{(1)}}} \\
		&= \E\brac{ \paren{ \log p(X) }^2} +\sum_{k=n}^\infty \frac{ \E\brac{ (1-p_X)^k \log p(X) }}{k}~.
	\end{align}
	With this, and using Proposition~\ref{prop:mathind} again, along with Lemma~\ref{lem:missingmass}, and $\log x \leq x-1$,
	\begin{align}
		&\cov\brac{J\paren{m^{(1)}-1}-J(n-1), \log p\paren{X^{(1)}}} \\
		&= \E\brac{\paren{J\paren{m^{(1)}-1}-J(n-1)} \log p\paren{X^{(1)}}} \\
		&\quad -\E\brac{J\paren{m^{(1)}-1}-J(n-1)} \E\brac{\log p\paren{X^{(1)}}} \\
		&= \E\brac{ \paren{ \log p(X) }^2} -[H(X)]^2 +\sum_{m=n}^\infty \frac{ \E\brac{ (1-p_X)^m \log p(X) } +H(X) \E\brac{(1-p_X)^m} }{m} \\
		&= \var\brac{ \log p(X) } +\mathcal O\paren{\frac{1}{n^{1-\alpha}}}~.
	\end{align}

	Moving to the covariance term of line~\ref{covneq},
	let $m_{-1}^{(2)} := \sum_{i=2}^n I\paren{X^{(2)} =X^{(i)}}$.
	Note that $m_{-1}^{(2)}-1$ is distributed as $\text{Binomial}(n-2, p(X^{(2)}))$ conditioning on $X^{(2)}$.
	If $X^{(1)}=X^{(2)}$, then $m^{(2)} =m_{-1}^{(2)}+1$ and if $X^{(1)}\neq X^{(2)}$, then $m^{(2)} =m_{-1}^{(2)}$.
	Below we use the law of total expectation to separate these cases.
	We apply $J(z) =J(z-1)+\frac{1}{z}$ where $X^{(1)}=X^{(2)}$ to isolate the binomial random variable within the harmonic number function, then reapply the law of total expectation in the other direction, rejoining the cases, with the remainder left out.
	Lastly, we apply Proposition~\ref{prop:mathind}.
	\begin{align}
		&\E\brac{\paren{J\paren{m^{(2)}-1}-J(n-1)}\log p\paren{X^{(1)}}} \\
		&=\P\paren{X^{(1)}=X^{(2)}}\E\brac{\paren{J\paren{m_{-1}^{(2)}} -J(n-1)}\log p\paren{X^{(1)}} \middle| X^{(1)}=X^{(2)}} \\
		&\quad +\P\paren{X^{(1)}\neq X^{(2)}} \E\brac{\paren{J\paren{m_{-1}^{(2)}-1}-J(n-1)}\log p\paren{X^{(1)}} \middle| X^{(1)}\neq X^{(2)} } \\
		&=\P\paren{X^{(1)}=X^{(2)}} \E\brac{\paren{J\paren{m_{-1}^{(2)}-1}-J(n-1)+\frac{1}{m_{-1}^{(2)}}}\log p\paren{X^{(1)}}\middle| X^{(1)}=X^{(2)}} \\
		&\quad +\P\paren{X^{(1)}\neq X^{(2)}} \E\brac{\paren{ J\paren{m_{-1}^{(2)}-1}-J(n-1)}\log p\paren{X^{(1)}}\middle| X^{(1)} \neq X^{(2)}} \\
		&=\E\brac{\paren{J\paren{m_{-1}^{(2)}-1} -J(n-1)} \log p\paren{X^{(1)}}} \\
		&\quad +\P\paren{X^{(1)}=X^{(2)}} \E\brac{\frac{\log p\paren{X^{(1)}}}{m_{-1}^{(2)}} \middle| X^{(1)} = X^{(2)} } \label{line:eqremainder} \\
		&= H(X) \paren{ \sum_{m=1}^{n-2} \frac{ \E\brac{(1-p_X)^m}}{m} +\frac{1}{n-1}} -\frac{ H(X) +\E\brac{ (1-p_X)^{n-1} \log(p_X)} }{n-1} \\
		&= H(X) \sum_{m=1}^{n-2} \frac{ \E\brac{(1-p_X)^m}}{m} -\frac{ \E\brac{ (1-p_X)^{n-1} \log(p_X)} }{n-1}~.
	\end{align}
	We simplify line \ref{line:eqremainder} using Proposition~\ref{prop:recipmoment}, $\E\brac{\frac{1}{U+1}} =\frac{1-(1-p)^{n+1}}{(n+1)p}$ for $U\sim\text{Binomial}(n,p)$,
	\begin{align}
		&\P\paren{X^{(1)}=X^{(2)}} \E\brac{\frac{\log p\paren{X^{(1)}}}{m_{-1}^{(2)}} \middle| X^{(1)} = X^{(2)} } \\
		&= \sum_i p_i^2 \log(p_i) \sum_{m=0}^{n-2} \frac{1}{m+1} \binom{n-2}{m} p_i^m (1-p_i)^{n-2-m} \\
		&= \sum_i p_i \log(p_i) \frac{ 1-(1-p_i)^{n-1} }{n-1} = -\frac{ H(X) +\E\brac{ (1-p_X)^{n-1} \log(p_X)} }{n-1} ~.
	\end{align}
	Using the joint expectation calculation above all lower order terms cancel.
	\begin{align}
		&\cov\brac{J\paren{m^{(2)}-1}-J(n-1), \log p\paren{X^{(1)}}} \\
		&= \E\brac{\paren{J\paren{m^{(2)}-1}-J(n-1)}\log p\paren{X^{(1)}}} \\
		&\quad -\E\brac{ \paren{J\paren{m^{(2)}-1}-J(n-1)}} \E\brac{\log p\paren{X^{(1)}}} \\
		&= H(X) \sum_{m=1}^{n-2} \frac{ \E\brac{(1-p_X)^m}}{m} -\frac{ \E\brac{ (1-p_X)^{n-1} \log(p_X)} }{n-1} -H(X) \sum_{m=1}^{n-1} \frac{ \E\brac{(1-p_X)^m}}{m} \\
		&= -H(X) \frac{ \E\brac{(1-p_X)^{n-1}}}{n-1} -\frac{ \E\brac{ (1-p_X)^{n-1} \log(p_X)} }{n-1} =\mathcal O\paren{\frac{1}{n^{2-\alpha}}} ~.
	\end{align}
	The bound for the last line uses $\log x \leq x-1$ and Proposition~\ref{lem:missingmass}.

	Collecting all of these calculations,
	\begin{align}
		&\var\brac{\hat H\paren{\mathbf X_n}-H^*\paren{\mathbf X_n}} \\
		&= \var\brac{\hat H\paren{\mathbf X_n}} +\var\brac{H^*\paren{\mathbf X_n}} - 2\cov\brac{\hat H\paren{\mathbf X_n}, H^*\paren{\mathbf X_n}} \\
		&= \frac{\var\brac{ \log p(X) }}{n} +\frac{\var\brac{ \log p(X) }}{n} -\frac{2\var\brac{ \log p(X) }}{n}  +\mathcal O\paren{\frac{\log n}{n^{2-\alpha}} } \\
		&= \mathcal O\paren{\frac{\log n}{n^{2-\alpha}} } ~.
	\end{align}
\end{proof}

\section{Plug-in Estimator Proofs}\label{apdx:plug}

\begin{proof}[Proof of Proposition~\ref{prop:plugbias}]
	We first characterize the exact bias of the plug-in estimator.
	Taking a similar approach to Proposition~\ref{bias}, let $M$ be a binomial random variable of $n-1$ trials with success probability $p$.
	Using the Frullani integral representation $\log(x) = \int_0^\infty \frac{e^{-t}-e^{-xt}}{t}\,dt$ for $x>0$, and applying Fubini's theorem to take the expectation inside the integral,
	\begin{align}
		&-\E\brac{\log\paren{\frac{M+1}{n}}} 
		= \E\brac{ \int_0^\infty \frac{ e^{-(M+1)t} -e^{-nt} }{t}\,dt } \\
		&= \int_0^\infty \frac{ e^{-t}\paren{1-p+pe^{-t}}^{n-1} -e^{-nt} }{t} \,dt \\
		&= \int_0^1 \frac{ (1-pu)^{n-1} -(1-u)^{n-1} }{-\log(1-u)} \,du \label{line:plugsub} \\
		&= \int_0^1 \frac{ (1-pu)^{n-1} -(1-u)^{n-1} }{u} \paren{\frac{-u}{\log(1-u)}} \,du \\
		&= \int_0^1 \frac{ (1-pu)^{n-1} -(1-u)^{n-1} }{u} \paren{1-\sum_{k=1}^\infty |G_k| u^k} \,du \label{line:gregory} \\
		&= -\log p -\sum_{m=n}^\infty \frac{(1-p)^m}{m} -\sum_{k=1}^\infty |G_k|\brac{ \frac{B_p(k,n)}{p^k} -B(k,n) }~. \label{line:plugerror}
	\end{align}
	Line~\ref{line:plugsub} applies the substitution $u=1-e^{-t}$.
	Line~\ref{line:gregory} Taylor expands $\frac{-u}{\log(1-u)}$, where $G_k$ is the $k$th Gregory coefficient.
	The Gregory coefficient terms of line~\ref{line:plugerror} follow from distributing the summation of line~\ref{line:gregory} and recognizing the incomplete beta function ($B_p(k,n) =\int_0^p u^{k-1} (1-u)^{n-1} \,du$) and beta function integrals.
	The remaining term of line~\ref{line:plugerror} follows using Fubini's theorem and the geometric series:
	\begin{align}
		&\int_0^1 \frac{ (1-pu)^{n-1} -(1-u)^{n-1} }{u} \,du \\
		&= (n-1) \int_0^1 \int_p^1 (1-tu)^{n-2} \,dt \,du 
		= \int_p^1 \frac{ 1-(1-t)^{n-1} }{t} \,dt \\
		&= \int_p^1 \frac{1}{t}\,dt -\int_p^1 \sum_{k=0}^\infty (1-t)^{n-1+k} \,dt 
		= -\log p -\sum_{m=n}^\infty \frac{(1-p)^m}{m}~. \label{line:plugfrullani}
	\end{align}

	With this identity, we directly calculate the plug-in estimator's bias.
	Recall that $m^{(i)}$ given $X^{(i)}$ is distributed as $m^{(i)}-1 \sim \text{Binomial}(n-1,p_X)$ because $X^{(i)}$ contributes to its own count.
	Applying the law of total expectation,
	\begin{align}
		&\E\brac{ \hat H_{\plug}(\mathbf{X}_n) } 
		= -\E\brac{ \frac{1}{n} \sum_{i=1}^n \log\paren{\frac{m^{(i)}}{n}} } 
		= -\E\brac{\E\brac{ \log\paren{\frac{m^{(1)}}{n}} \middle| X^{(1)} } } \\
		&= \E\brac{ -\log p_X -\sum_{m=n}^\infty \frac{(1-p_X)^m}{m} -\sum_{k=1}^\infty |G_k|\paren{ \frac{B_{p_X}(k,n)}{p_X^k} -B(k,n) } }~.
	\end{align}
	Recognizing that $\E[-\log p_X] =H(X)$ establishes the first claim of Proposition~\ref{prop:plugbias}.

	To establish the bounds, the proof of \cite[Thm. 7]{antos2001convergence} bounds the plug-in's negative bias as $\mathcal{O}\paren{n^{-(1-\alpha)}}$ for $p_j\lesssim j^{-1/\alpha}$, using $q=1/\alpha$.
	Thus, the sum of our two bias error terms is upper-bounded by $\mathcal{O}\paren{n^{-(1-\alpha)}}$.

	Focusing on the Gregory coefficient summation, we revert to its integral form to provide a lower bound.
	Because $|G_k|>0$ for all $k$, and $(1-p_ju)^{n-1}-(1-u)^{n-1}\geq0$ for all $p_j,u\in[0,1]$, all terms are non-negative.
	Isolating the $k=1$ term (where $|G_1|=1/2$) yields
	\begin{align}
		&\sum_{k=1}^\infty |G_k|\paren{ \E\brac{\frac{B_{p_X}(k,n)}{p_X^k}} -B(k,n) } \\
		&\geq \frac{1}{2} \sum_j p_j \int_0^1 \brac{ (1-p_j u)^{n-1} -(1-u)^{n-1} } \,du \label{line:pluglower} \\
		&= \frac{1}{2n} \paren{-1+\sum_j \brac{ 1-(1-p_j)^n } }~.
	\end{align}
	The inequality of line~\ref{line:pluglower} drops all $k\geq2$ terms, each non-negative.

	Finally, we bound the summation $\sum_j\brac{1-(1-p_j)^n}$.
	Splitting the sum at an arbitrary integer $k$ and applying Bernoulli's inequality ($1-nx\leq(1-x)^n$ for $x\in[0,1]$),
	\begin{align}
		&\sum_j \brac{ 1-(1-p_j)^n }
		= \sum_{j=1}^k \brac{ 1-(1-p_j)^n } +\sum_{j=k+1}^\infty \brac{ 1-(1-p_j)^n } \\
		&\leq \sum_{j=1}^k 1 +\sum_{j=k+1}^\infty np_j
		\leq k + Cn\int_k^\infty x^{-1/\alpha}\,dx~. \label{line:plugthreshold}
	\end{align}
	Setting the threshold $k=(Cn)^\alpha$ in line~\ref{line:plugthreshold} yields $\mathcal{O}(n^\alpha)$, which makes the lower bound $\mathcal{O}\paren{n^{-(1-\alpha)}}$ and completes the proof.
\end{proof}

\begin{prop}\label{prop:plugbiastight}
	The additional bias upper bound from Proposition~\ref{prop:plugbias} is tight for $\alpha\in (0,1)$.
\end{prop}
\begin{proof}
	Let $p_j = Cj^{-1/\alpha}$ be the zeta distribution from Proposition~\ref{prop:tightbias}, where $C=1/\zeta(1/\alpha)$.
	Using $1-x\leq e^{-x}$, for $j\leq k:=\lfloor (nC/\log 2)^\alpha \rfloor$ we have $np_j\geq\log 2$, so $(1-p_j)^n\leq e^{-np_j}\leq 1/2$, hence $1-(1-p_j)^n\geq 1/2$.
	Thus,
	\begin{equation}
		\sum_j \brac{1-(1-p_j)^n} \geq \sum_{j=1}^k \brac{1-(1-p_j)^n} \geq \frac{k}{2} \asymp n^\alpha~.
	\end{equation}
	Applying (\ref{eq:plugbiaslower}),
	\begin{equation}
		\sum_{k=1}^\infty |G_k| \paren{\E\brac{\frac{B_{p_X}(k,n)}{p_X^k}} -B(k,n)}
		\geq \frac{\brac{\sum_j 1-(1-p_j)^n} -1}{2n}
		\gtrsim \frac{n^\alpha-1}{n} \asymp \frac{1}{n^{1-\alpha}}~.
	\end{equation}
	Because this lower bound matches the upper bound of Proposition~\ref{prop:plugbias}, the additional bias bound is tight.
\end{proof}

\begin{prop}[Variance Bounds applying Efron-Stein] \label{thm:varbound}
	If $\mathbf{X}_n$ is a random sample, then
	\begin{equation}
		\var\brac{ \hat H\paren{ \mathbf X_n } } \lesssim \frac{ \paren{\log n }^2}{n} ~.
	\end{equation}
	If $p_j \lesssim j^{-1/\alpha}$ for $\alpha\in (0,1)$, then
	\begin{equation}
		\var\brac{ \hat H\paren{ \mathbf X_n } }, \var\brac{ \hat H_{\plug} \paren{ \mathbf X_n } }  \lesssim \frac{1}{n} ~.
	\end{equation}
\end{prop}

\begin{proof}
	Let $\mathbf X_n = \paren{X^{(1)}, \dots, X^{(n)}}$ and $\mathbf Y_n = \paren{Y^{(1)}, \dots, Y^{(n)}}$ be two independent, random samples from $p$.
	Define $\mathbf X_n^{(i)} = \paren{X^{(1)}, \dots, X^{(i-1)}, Y^{(i)}, X^{(i+1)}, \dots, X^{(n)}}$ to be the sample sharing all points with $\mathbf X_n$ except with $X^{(i)}$ replaced by $Y^{(i)}$.

	Both estimators take the form $\hat H_f(\mathbf{X}_n) =f(n) -\frac{1}{n} \sum_j m_j f(m_j)$, with $f(k) =J(k-1)$ for the harmonic estimator and $f(k) =\log(k)$ for the plug-in, where $m_j$ is the count of symbol $j$ in $\mathbf{X}_n$.
	Consider $\hat H_f(\mathbf{X}_n) -\hat H_f\paren{\mathbf{X}_n^{(i)}}$.
	If $X^{(i)} = Y^{(i)}$, then $\hat H_f\paren{\mathbf X_n} - \hat H_f\paren{\mathbf X_n^{(i)}} = 0$ because of identical counts in both samples.
	If $X^{(i)}\neq Y^{(i)}$, then other than the counts of $X^{(i)}$ and $Y^{(i)}$, the counts are identical between the samples and will cancel.
	Let $m_X^a$ and $m_Y^a$ be the counts of $X^{(i)}$ and $Y^{(i)}$ in $\mathbf{X}_n$, respectively, and $m_X^b$ and $m_Y^b$ be the counts of $X^{(i)}$ and $Y^{(i)}$ in $\mathbf{X}_n^{(i)}$, respectively.
	Further, $m_X^b =m_X^a-1$ and $m_Y^b =m_Y^a+1$.
	Then, because all identical counts cancel,
	\begin{align}
		&\hat H_f\paren{\mathbf X_n} - \hat H_f\paren{\mathbf X_n^{(i)}} \\
		&= -\frac{m_X^a f(m_X^a) }{n} -\frac{m_Y^a f(m_Y^a) }{n} +\frac{m_X^b f(m_X^b) }{n} +\frac{m_Y^b f(m_Y^b) }{n} \\
		&= \frac{ (m_Y^a+1) f(m_Y^a+1) -m_Y^a f(m_Y^a) }{n} -\frac{ (m_X^b+1) f(m_X^b+1) -m_X^b f(m_X^b) }{n} \\
		&= \frac{ f(m_Y^a+1) -f(m_X^b+1) +\epsilon 1_{f=\log} }{n} \label{eq:fcases} \\
		&= \frac{ \brac{ f(n) -f(m_X^b+1) } -\brac{ f(n) -f(m_Y^a+1) } +\epsilon 1_{f=\log} }{n} ~.
	\end{align}
	Expanding on line~\ref{eq:fcases}, if $f(k)=J(k-1)$, using $J(m)-J(m-1)=\frac{1}{m}$, then
	\begin{align}
		&(m+1) f(m+1) -m f(m)
		=(m+1) J(m) -m J(m-1) \\
		&= m [J(m) -J(m-1)] +J(m)
		= J(m)+1
		= f(m+1) +1~.
	\end{align}
	Note that the constant 1 terms cancel exactly.
	If $f(k)=\log k$, then 
	\begin{align}
		&(m+1) f(m+1) -m f(m)
		=(m+1) \log(m+1) -m \log(m) \\
		&=(m+1) \log(m+1) -m \brac{\log(m+1) -\log\paren{1+\frac{1}{m}} } \\
		&= \log(m+1) +m\log\paren{1+\frac{1}{m}} ~.
	\end{align}
	Here, $0\leq m\log(1+1/m) <1$ for $m\geq 0$,\footnote{We may define $0 \log 0 = 0$, so the expression above is zero when $m=0$.} so $-1\leq \epsilon \leq 1$.

	In both cases, $X^{(i)}=Y^{(i)}$ and $X^{(i)}\neq Y^{(i)}$,
	\begin{equation}
		\hat H_f\paren{\mathbf X_n} - \hat H_f\paren{\mathbf X_n^{(i)}} 
		= \frac{ \brac{ f(n) -f(m_X^{(i)}+1) } -\brac{ f(n) -f(m_Y^{(i)}+1) } +\epsilon 1_{f=\log} }{n} 
	\end{equation}
	where $m_Y^{(i)}$ is the count of $Y^{(i)}$ in $\mathbf{X}_n$ and $m_X^{(i)}$ is the count of $X^{(i)}$ in $\mathbf{X}_n^{(i)}$.

	Using the Efron-Stein inequality \cite[Thm 3.1]{boucheron2013concentration},
	\begin{align}
		&\var\brac{\hat H_f\paren{\mathbf X_n}} 
		\leq \frac{1}{2} \sum_{i=1}^n \E\brac{\hat H_f(X_n) - \hat H_f\paren{X_n^{(i)}}}^2 \\
		&= \frac{1}{2} \sum_{i=1}^n \E\brac{ \paren{\frac{ \brac{ f(n) -f(m_X^{(i)}+1) } -\brac{ f(n) -f(m_Y^{(i)}+1) } +\epsilon 1_{f=\log} }{n} }^2 } \\
		&= \frac{1}{2} \sum_{i=1}^n \E\brac{ \frac{ \paren{ \brac{ f(n) -f(m_X^{(i)}+1) } -\brac{ f(n) -f(m_Y^{(i)}+1) } }^2 +1_{f=\log} }{n^2} } \label{line:eqexpcancel} \\
		&\leq \frac{1}{2} \sum_{i=1}^n \E\brac{ \frac{ 4 \brac{ f(n) -f(m_X^{(i)}+1) }^2 +1_{f=\log} }{n^2} } \label{line:diffbinomineq} \\
		&= \frac{1}{2n} \E\brac{  4 \brac{ f(n) -f(m_X^{(1)}+1) }^2 +1_{f=\log} } ~.
	\end{align}
	The equality from line~\ref{line:eqexpcancel} follows from expanding the square and canceling first order terms as $X$ and $Y$ are identically distributed.
	The inequality from line~\ref{line:diffbinomineq} follows from $(a-b)^2 \leq 2a^2+2b^2$; the two expressions combine because $X$ and $Y$ are identically distributed.

	We finish the proof by bounding the expectation.
	Recall that $m_X$ is a binomial random variable with $n-1$ trials and probability $p_X$, as it is the count of $X^{(1)}$ in $\mathbf{X}_n^{(1)}$.

	\textbf{Case 1:} $f(m) =J(m-1)$.
	Under any regime, $J(n-1) -J\paren{m_X^{(1)}} \leq J(n-1) \leq \log(n-1) +1$.
	\begin{equation}
		\var\brac{\hat H\paren{\mathbf X_n}} 
		\leq \frac{ 2\E\brac{ \paren{ J(n-1) -J\paren{m_X^{(1)}} }^2 } }{n}
		\leq \frac{2 \paren{\log(n-1) +1}^2}{n} ~.
	\end{equation}

	If $p_j =\mathcal{O}(j^{-1/\alpha})$ for $\alpha \in (0,1)$, then, using Proposition~\ref{prop:scndmnt} on the binomial expectation, and Proposition~\ref{prop:varremainders} on the resulting remainder, we have
	\begin{align}
		&\E\brac{ \paren{f(n) - f\paren{m_{X}^{(1)}+1} }^2 }
		= \E\brac{\E\brac{ \paren{J(n-1) - J\paren{m_{X}^{(1)}} }^2 \middle| X } } \\
		&\leq \E\brac{ \paren{\log p_X}^2 +\sum_{m=1}^{n-1} \sum_{k=n-m}^{n-1} \frac{(1-p_X)^m}{mk} } 
		= \E\brac{\paren{\log p_X}^2} +\mathcal{O}\paren{\frac{1}{n^{1-\alpha}}}~.
	\end{align}
	So,
	\begin{equation}
		\var\brac{\hat H_f\paren{\mathbf X_n}} 
		\leq \frac{ 2\E\brac{ \paren{ f(n) -f\paren{m_X^{(1)}+1} }^2 } }{n}
		= \frac{ 2 \E\brac{\paren{\log p_X}^2} }{n} +\mathcal{O}\paren{\frac{1}{n^{2-\alpha}}}~.
	\end{equation}

	\textbf{Case 2:} $f(m)=\log(m)$.
	Assume $p_j =\mathcal{O}(j^{-1/\alpha})$ for $\alpha \in (0,1)$.
	Then
	\begin{align}
		&\E\brac{ \paren{f(n) - f\paren{m_{X}^{(1)}+1} }^2 }\\
		&= \sum_j p_j \E\brac{ \log^2\paren{ \frac{m_j^{(1)}+1}{n}} } \\
		&= \sum_{j=1}^k p_j \E\brac{ \log^2\paren{ \frac{m_j^{(1)}+1}{n}} } +\sum_{j=k+1}^\infty p_j \E\brac{ \log^2\paren{ \frac{m_j^{(1)}+1}{n}} } \\
		&\leq \sum_{j=1}^k p_j \paren{ \paren{\log p_j}^2 +\frac{1}{np_j} } +\log^2\paren{\frac{1}{n}} \sum_{j=k+1}^\infty p_j \label{line:logbounds} \\
		&\leq \E\brac{ \paren{\log p_X}^2} +\frac{k}{n}  +\log^2\paren{\frac{1}{n}} \int_{k}^\infty cx^{-1/\alpha}\,dx \\
		&\leq \E\brac{(\log p_X)^2} +\mathcal{O}\paren{\frac{\log^2(n)}{n^{1-\alpha}}}~.
	\end{align}
	We use Proposition~\ref{prop:logsqbound}, $\E\brac{\log^2\paren{\frac{M+1}{n}}} \leq \log^2 p +\frac{1}{np}$ for $M\sim\text{Binomial}(n-1,p)$.
	The final line follows from setting $k=n^\alpha$.

	So,
	\begin{align}
		\var\brac{\hat H_{\plug} \paren{\mathbf X_n}} 
		&\leq \frac{ 2\E\brac{ \paren{ \log(n) -\log\paren{m_X^{(1)}+1} }^2 +\epsilon^2 } }{n} \\
		&\leq \frac{ 2 \E\brac{\paren{\log p_X}^2} +2 }{n} +\mathcal{O}\paren{\frac{(\log n)^2}{n^{2-\alpha}}}~.
	\end{align}
\end{proof}

\begin{prop} \label{prop:logsqbound}
	If $M\sim\text{Binomial}(n-1,p)$, then
	\begin{equation}
		\E\brac{\log^2\paren{\frac{M+1}{n}}} \leq \log^2(p) +\frac{1}{np}~.
	\end{equation}
\end{prop}
\begin{proof}
	Let $X=(M+1)/n$ and write $X=py$ for $y>0$.
	Using the two inequalities $\log(y) \geq 1-1/y$ and $\log^2(y) \leq (y-1)^2/y$, both valid for all $y>0$, and multiplying the first by $2\log(p)<0$, which flips it,
	\begin{align}
		\log^2(X) = \log^2p +2\log(p)\log(y) +\log^2(y)
		\leq \log^2p +2\log(p)\paren{1-\frac{p}{X}} +\frac{(X-p)^2}{pX}~.
	\end{align}
	Using Proposition~\ref{prop:recipmoment} with $n\to n-1$,
	\begin{equation}
		\E\brac{\frac1X} = \E\brac{\frac{n}{M+1}} = \frac{1-(1-p)^n}{p}~,
	\end{equation}
	so $\E[1-p/X] =(1-p)^n \geq 0$, and because $2\log(p)<0$,
	\begin{equation}
		\E\brac{2\log(p)\paren{1-\frac{p}{X}}} = 2\log(p)(1-p)^n \leq 0~.
	\end{equation}
	Using $\E[X] =p+(1-p)/n$,
	\begin{equation}
		\E\brac{\frac{(X-p)^2}{X}} = \E[X] -2p +p^2\E\brac{\frac1X} = \frac{1-p}{n} -p(1-p)^n \leq \frac1n~,
	\end{equation}
	dropping the non-positive $-p(1-p)^n$ term and using $1-p<1$.
	Combining,
	\begin{equation}
		\E\brac{\log^2(X)} \leq \log^2p +0 +\frac1p\cdot\frac1n = \log^2(p) +\frac{1}{np}~.
	\end{equation}
\end{proof}

\begin{prop}\label{prop:varcomparison}
	Let $f, g: \R \rightarrow \R$ be non-decreasing functions and $X$ be a real-valued random variable. If $h(x) := f(x) - g(x)$ is also non-decreasing, then
	\begin{equation}
		\var\brac{f(X)} \geq \var\brac{g(X)} ~.
	\end{equation}
\end{prop}
\begin{proof}
	With the variance of a sum property,
	\begin{align}
		\var\brac{f(X)}
		&= \var\brac{g(X) + h(X)} \\
		&= \var\brac{g(X)} +\var\brac{h(X)} +2\cov\brac{ g(X), h(X) } \\
		&\geq \var\brac{g(X)}
	\end{align}
	where the inequality follows because $g$ and $h$ are non-decreasing functions of a single random variable, ensuring $\cov\brac{ g(X), h(X) } \geq 0$ via Chebyshev's association inequality \cite[Thm. 2.14]{boucheron2013concentration}.
\end{proof}

\section*{Acknowledgements}
The author would like to thank Liza Levina and Ji Zhu for their support with this project, Larry Wasserman and Cosma Shalizi for reviewing an earlier draft of this paper, and the reviewers whose critical eye strengthened this manuscript.
This work is dedicated to the memory of Grant Mesner.

\section*{Funding}
The author was supported by NSF RTG Grant DMS-1646108, NIH NIDA Grant U01 DA03926, and NIH NIAID 2K24AI134413.

\bibliographystyle{IEEEtran}
\bibliography{graph}
\end{document}